\documentclass[10pt]{amsart}
\usepackage{stmaryrd}

\usepackage{amsmath}
\usepackage{amsthm} 
\usepackage{graphicx}
\usepackage{color}
\usepackage{scalerel}
\usepackage[dvipsnames]{xcolor}

\usepackage{comment}
\usepackage{bm}

\usepackage{enumerate}

\usepackage{enumitem}

\usepackage{fullpage}

\usepackage{amsfonts}
\usepackage{amssymb}
\usepackage[hidelinks]{hyperref}
\usepackage{thmtools}
\usepackage[capitalize]{cleveref}

\usepackage{nicefrac}
\usepackage{xfrac}

\definecolor{LightGray}{rgb}{.6,.6,.6}

\usepackage{tipa}
\usepackage[normalem]{ulem}

\usepackage{mathtools}

\newcommand{\Fraisse}{Fra\"{i}ss\'{e}}

\newcommand{\defas}{:=} 
\newcommand{\st}{\,:\,} 

\DeclareMathOperator*{\forkindep}{\raise0.2ex\hbox{\ooalign{\hidewidth$\vert$\hidewidth\cr\raise-0.9ex\hbox{$\smile$}}}}

\DeclareMathOperator{\ar}{ar}

\newcommand{\dcl}{\operatorname{dcl}}
\newcommand{\tp}{\operatorname{tp}}

\DeclareTextCommand{\DZ}{OT2}{D2}
\newcommand{\dfs}{\mathrm{dfs}}
\newcommand{\Rado}{\mathrm{Rado}}

\newcommand{\Lc}{\mathcal{L}}

\newcommand{\weak}{adequate}
\newcommand{\strong}{excellent}

\newtheorem*{claim-star}{Claim}
\newtheorem*{theorem-non}{Theorem}
\newtheorem{theorem}{Theorem}[section] 
\newtheorem{lemma}[theorem]{Lemma}

\newtheorem{prop-def}[theorem]{Proposition-Definition}
\newtheorem{corollary}[theorem]{Corollary}
\newtheorem{fact}[theorem]{Fact}
\newtheorem{fact-eh}[theorem]{Fact(?)}

\newtheorem{question}[theorem]{Question}
\newtheorem{proposition}[theorem]{Proposition}
\newtheorem{proposition-eh}[theorem]{Proposition(?)}
\newtheorem*{theorem-star}{Theorem}
\newtheorem*{conjecture-star}{Conjecture}
\newtheorem*{lemma-star}{Lemma}

\theoremstyle{definition}
\newtheorem{definition}[theorem]{Definition}

\newtheorem{remark}[theorem]{Remark}
\newtheorem{convention}[theorem]{Convention}
\newtheorem{problem}[theorem]{Problem}
\theoremstyle{remark}
\newtheorem{claim}[theorem]{Claim}

\newtheorem*{warning}{Warning}

\newcommand{\inv}{\mathrm{inv}}
\newcommand{\supp}{\mathrm{supp}}

\newcommand{\str}{\mathrm{Str}}
\newcommand{\Nats}{\mathbb{N}}
\newcommand{\Leb}{\mathfrak{m}}

\DeclareMathOperator{\Sym}{Sym}

\newcommand{\defn}[1]{\textbf{#1}}
\newcommand{\w}{\omega}

\allowdisplaybreaks 

\title{Generic sampling and invariant measures on the space of $k$-uniform hypergraphs}

\author[N. Ackerman]{Nathanael Ackerman}
\address{Harvard University\\
Cambridge, MA,
USA}
\email{nate@aleph0.net}

\author[C. Freer]{Cameron Freer}
\address{Massachusetts Institute of Technology\\
Cambridge, MA,
USA}
\email{freer@mit.edu}

\author[K. Gannon]{Kyle Gannon$^{\dagger}$}
\thanks{$^{\dagger}$ Corresponding author; Supported by the Fundamental Research Funds for the Central Universities, Peking University, grant no. 7100604835 and by the National Natural Science Fund of China, grant no. 12501001}
\address{$^{\dagger}$ Beijing International Center for Mathematical Research (BICMR) \\ Peking University \\ Beijing, China.}
\email{kgannon@bicmr.pku.edu.cn} 

\author[J. Hanson]{James E. Hanson}
\address{ Iowa State University \\
  Ames, IA, USA }
\email{jameseh@iastate.edu} 

\author[R. Patel]{Rehana Patel}
\address{Wesleyan University\\
Middletown, CT,
USA}
\email{rpatel@wesleyan.edu}

\newcommand{\kmpow}[1]{\mathbb{P}_{#1}}
\newcommand{\EmptySet}{\emptyset}

\begin{document}

\begin{abstract}
We prove a model-theoretic representation theorem for the distribution of an ergodic 
exchangeable $k$-uniform hypergraph: 
every such measure arises as the pushforward of the countably-iterated Morley product of a global Borel-definable Keisler measure over the countable universal homogeneous $k$-uniform hypergraph. We show this by starting with a Borel $k$-hypergraphon $W$ and constructing a Keisler measure $\mu_{W}$ such that
generic sampling with respect to $\mu_{W}$ yields the same 
$\Sym(\mathbb{N})$-invariant measure 
as does the standard hypergraphon sampling procedure 
with respect to $W$. When $k = 2$, our results give a new representation theorem for ergodic exchangeable graphs via Keisler measures over a monster model of the Rado graph. 
\end{abstract}

\maketitle

\section{Introduction}

In many areas of mathematics, a key role is played by 
$\Sym(\mathbb{N})$-invariant measures on the space of countable $\mathcal{L}$-structures (for a fixed countable language $\mathcal{L}$), where
$\Sym(\mathbb{N})$ acts via the logic action 
by permuting the underlying set. Representation theorems for such invariant measures have been studied since the foundations of modern probability theory.
Each of these representation theorems 
arises from a somewhat different perspective on invariant measures, and thereby provides new insights.
The first such result was in the 1930s, due to de Finetti \cite{deFinetti}, in his celebrated representation theorem about exchangeable sequences, which plays a foundational role in Bayesian statistics \cite{MR0440641}.
This was extended in the 1970s by Aldous \cite{MR637937} and Hoover \cite{Hoover79} to a representation for finite-dimensional arrays.
In the 2000s, Lov\'asz and others \cite{lovasz2012large} used deep results in combinatorics, such as Szemer\'edi's regularity lemma \cite{MR369312}, to provide a representation via \emph{graphons} for invariant measures on the spaces of graphs.
A slightly different representation, which extends to hypergraphs, was provided by Elek and Szegedy  \cite{MR2964622} using ultraproducts.

Invariant measures play an important role in connecting probability and logic. In particular, work by Gaifman in the 1960s \cite{gaifman1964concerning}, extended by Scott and Krauss \cite{ScottKrauss}, viewed $\Sym(\mathbb{N})$-invariant measures as ``symmetric probabilistic structures'', and began the study of their model theory.
Recent work by Ackerman, Freer, and Patel has shown that many results of infinitary logic have analogues for such probabilistic structures \cite{complete-classification}.

Gannon and Hanson \cite{gannon2024model}
introduced the method of \emph{generic sampling}, which can be thought of as a model-theoretic sampling procedure over a type space corresponding to a certain structure. Given a global Borel-definable Keisler measure $\mu$, one can iterate this measure with respect to the Morley product countably many times to produce a measure on the type space with infinitely many free variables. While the Morley product of a Keisler measure with itself is not always invariant under permutation of variables, when it is, iterating this product on said measure gives rise to a $\Sym(\mathbb{N})$-invariant measure.
 
In this paper we show that every ergodic $\Sym(\mathbb{N})$-invariant measure on the space of $k$-uniform hypergraphs (for fixed $k$) arises via the one described in the previous paragraph. This yields a new representation theorem for such invariant measures, providing additional ties to model theory. These results are new even in the case of graphs.

\subsection{Background}

\emph{Generic sampling}, which was introduced in \cite{gannon2024model}, can be thought of as a model-theoretic sampling procedure over a certain type space corresponding to a first-order structure. Given a global Borel-definable Keisler measure $\mu$, one iterates this measure with respect to the Morley product countably many times to produce a measure $\kmpow{\mu}$ on the type space in countably many variables. It is important to note that the measure $\kmpow{\mu}$ is not the standard product measure $\mu^{\omega}$ on the product of type spaces. The measure $\kmpow{\mu}$ is over a larger space which admits a canonical continuous surjection onto the standard product, whose pushforward along this map is $\mu^{\omega}$ (see \cite[Proposition 3.3]{gannon2024model}). The main purpose of working within this larger space is to ensure that (hyper)edge relations, as well as more complex definable relations, are measurable. This construction often, but not always, results in a 
measure $\kmpow{\mu}$ that is not invariant under permutation of variables. This is generally a feature, not a bug. In the model theory vernacular, one can plausibly interpret this procedure as generating a \emph{random Morley sequence} relative to the fixed measure $\mu$. Since Morley sequences outside of the context of stable theories are often non-commutative objects, many of the properties of $\kmpow{\mu}$ are consistent with our usual model-theoretic intuition.

The main result of this paper is a representation theorem for exchangeable hypergraphs. This result 
demonstrates that sampling with respect to hypergraphons can be encoded in the framework of generic sampling over certain type spaces. These type spaces are over monster models of the theory of the countable universal homogeneous $k$-uniform hypergraph --- and in the case of graphons, over a monster model of the theory of the Rado graph. More explicitly, we let $T_{k}$ be the theory of the countable universal homogeneous $k$-uniform hypergraph, i.e., $T_{k}$ is the theory of the \Fraisse\ limit of all finite $k$-uniform hypergraphs in the language with a single $k$-ary relation symbol. When $k=2$, the theory $T_{2}$ is just the theory of the Rado graph. Let $\mathcal{U}$ be a monster model of $T_{k}$ and $M$ a countable elementary submodel. We show that if $W$ is a $k$-ary (Borel) hypergraphon  
then there exists an \emph{excellent} global $M$-invariant Keisler measure $\mu_{W}$ such that generic sampling with respect to $\mu_{W}$ is equivalent to hypergraphon sampling with respect to $W$. We remark that \emph{excellence} is a technical term which essentially means that generic sampling with respect to $\mu_{W}$ is well-behaved and in particular {self-commutative}, i.e., invariant under permutation of variables. To be precise about the equivalence, if we let $\mathbb{G}(\mathbb{N}, W)$ be the $\Sym(\mathbb{N})$-invariant measure on the space of $k$-uniform hypergraphs on $\mathbb{N}$ generated by $W$, then 
\begin{equation*}
    g_{*}(\kmpow{\mu_{W}}) = \mathbb{G}(\mathbb{N}, W), 
\end{equation*}
where $g$ is the map taking a type $p$ to the induced structure on a realization of $p$, and $g_{*}$ denotes the pushforward along this map. 

We provide proofs for the representation theorem in both the setting of graphs (the case of graphons) and that of $k$-uniform hypergraphs (the case of hypergraphons) separately. We do this for two reasons: The proof of the hypergraph variant has both a conceptual and notational jump in complexity due to the fact that one must take into account non-trivial {lower-arity data} that is non-existent in the graph variant. Thus, firstly, if the reader is interested only in the graph variant of the theorem and not in the hypergraph one, then they do not need to decode a technical framework to understand the proof.  Secondly, even though the hypergraph variant has non-trivial notational complexity, the two proofs follow essentially the same rough path. Hence if the reader wants an overview of the hypergraph variant without getting bogged down in notation, they may start with the graph variant to have the general idea of the proof in mind.

The outline of the proof is as follows: 
\begin{enumerate}
    \item Given a Borel $k$-uniform hypergraphon $W$, we define a global $M$-invariant Keisler measure $\mu_{W}$ (Definition \ref{def:measure} for graphon variant; Definition \ref{def:measure2} for hypergraphon). Broadly speaking, we treat an arbitrary tuple $\bar{c}$ of parameters from outside of $M$ as if it had been randomly generated by sampling the hypergraphon, with the result of this sampling represented by the type of $\bar{c}$ over $M$. In the binary case, it is quite easy to verify that our construction results in a Keisler measure (Proposition \ref{prop:inv}). In the hypergraphon setting, this is a \emph{little nightmare} (Proposition \ref{prop:nightmare}; \cref{sec:nightmare-proof}). 
    \item We next demonstrate that $\mu_{W}$ is Borel-definable over a countable model and thus well-behaved with respect to the Morley product. This good behavior follows directly from general results about \emph{Keisler measures in the wild} \cite{CGH}, and in particular the observation that measures that are Borel-definable over a countable model, in a countable language, are \emph{tame}. This observation more or less follows from the fact that type spaces over countable models, again in a countable language, are Polish spaces. We recall that in general, arbitrary global Borel-definable measures over the random $k$-uniform hypergraphs can be \emph{arbitrarily wild} (e.g., failure of products being Borel-definable, failure of associativity) and so it is quite satisfying that these measures exist.  
    \item We then prove an important technical lemma which allows us to compute integrals of functions with respect to $\mu_{W}$ over finite collections of parameters from outside of $M$. In particular, this amounts to computing the Radon-Nikodym derivatives of restrictions of $\mu_{W}$ to sets of parameters of the form $M\bar{c}$ (see \cref{lemma:key}; \cref{lemma:key3}). From this lemma we derive that $\mu_{W}$ is excellent, i.e., self-commutes. As a consequence, the measure $\kmpow{\mu_{W}}$ is invariant under the natural action of $\Sym(\mathbb{N})$ (by permuting variables), and thus the pushforward of $\kmpow{\mu_{W}}$ to the space of ($k$-uniform hyper)graphs on $\mathbb{N}$ is $\Sym(\mathbb{N})$-invariant. 
    \item We then prove a second important technical lemma which allows us to compute the integral of functions with respect to the iterated Morley product of $\mu_{W}$ (\cref{lemma:integral}; \cref{lemma:hyper-int}). By (3), we may already conclude that $g_{*}(\kmpow{\mu_{W}})$ is a $\Sym(\mathbb{N})$-invariant measure on the space of $k$-uniform hypergraphs on $\mathbb{N}$. Therefore, it suffices to prove that for all $n \in \mathbb{N}$, the projections of the measures $g_{*}(\kmpow{\mu_{W}})$ and $\mathbb{G}(\mathbb{N}, W)$ to distributions on the space of $k$-uniform hypergraphs on $\{1,...,n\}$ are the same. This is done in \cref{theorem:2main,theorem:main}. 
    \item A specific instance of a general result of Ackerman, Freer, and Patel \cite{ackerman2016invariant}
    characterizes those ($k$-uniform hyper)graphs $H$ for which there is a $\Sym(\mathbb{N})$-invariant measure concentrated on the collection of (hyper)graphs on  $\mathbb{N}$ that are isomorphic to $H$ (\cite[Theorem 1]{ackerman2016invariant}). Our results give a characterization in terms of the existence of certain Keisler measures (\cref{cor:AFP1,cor:AFP2}).
    \item It is known that all ergodic $\Sym(\mathbb{N})$-invariant measures that concentrate on the space of countably infinite ($k$-uniform hyper)graphs can be obtained via a 
    process of ($k$-uniform hyper)graphon sampling. By \cref{theorem:2main,theorem:main}, we conclude that all of these measures can also be obtained via the process of generic sampling (\cref{cor:ergodic,cor:hyper_ergodic}).
\end{enumerate}

From the model-theoretic perspective, the measures constructed in this paper are interesting in their own right. We provide a large collection of measures outside of the NIP context that are inherently tame. Given a ($k$-uniform hyper)graphon $W$, the measure $\mu_{W}$ is Borel-definable, self-associative, and self-commutes with respect to the Morley product. However, none of the $\mu_{W}$ are generically stable (not even $\dfs$; see \cite[Corollary 4.10]{conant2020remarks}). This follows from the fact that 
countable universal homogeneous $k$-uniform hypergraphs do not admit any non-trivial generically stable measures \cite[Corollary 4.10]{conant2020remarks}. 

The paper is outlined as follows: In Section~\ref{sec:prelim}, we provide some necessary preliminaries. In Section~\ref{sec:zero}, we resolve an important loose end from \cite{gannon2024model} by observing a $0-1$ law of generic sampling with respect to excellent Keisler measures (Proposition \ref{prop:01}). In Section~\ref{sec:graphon}, we first consider the class of $\EmptySet$-invariant measures which were completely classified by Albert in \cite{albert1989random}. We prove that generic sampling with respect to this class of measures typically results in a copy of the Rado graph (\cref{prop:emptyset4}). However, there is also an interesting edge case where a different phenomenon occurs, namely, the construction of \emph{threshold graphs}. We then prove the exchangeable graph representation theorem following the sketch provided above, culminating with Theorem \ref{theorem:2main}. Finally, in Section~\ref{sec:hypergraphon}, we prove the exchangeable hypergraph representation theorem (\cref{theorem:main}), again following the recipe provided above. A key technical step in the proof of the main result of Section~\ref{sec:hypergraphon} is provided in Appendix~\ref{sec:nightmare-proof}.

In future work, we plan to generalize our representation theorem to characterize all invariant measures on the space of structures in a fixed finite relational language. We also plan to study the relationship between the choice of the structure $M$ over which we consider Keisler measures 
and the collection of $\Sym(\mathbb{N})$-invariant measures which can arise from iterated Morley products of those measures.

\section*{Acknowledgement} 

We would like to thank Forte Shinko for several important discussions and some helpful comments. He also initially suggested a connection between the two areas of research present within this work. We would also like to thank Vincenzo Mantova who, together with the third author, computed \cref{prop:emptyset3} while waiting for a table in Leeds. 

\section{Preliminaries}
\label{sec:prelim}

We assume some basic background in model theory. Almost all of our notation is standard. In Sections 4 and 5, some notation is provided at the beginning of the section for the reader's convenience. In particular, the beginning of Section 4 has some background and conventions regarding graphons while Section 5 likewise has similar comments regarding hypergraphons. 

We write $\Leb$ to denote Lebesgue measure restricted to the interval $[0,1]$. Throughout the entire article, we fix $\mathcal{U}$ a monster model of a first-order theory $T$ and $M$ a small elementary substructure of $\mathcal{U}$. The letters $x,y$ and $z$ represent singletons of variables while $\bar{x},\bar{y}$ and $\bar{z}$ represent finite tuples of variables, e.g., $\bar{x} = x_1,...,x_n$. 
The number of variables in a specific tuple will be clear from the context. We recall that if $A \subseteq \mathcal{U}$ then we let $\mathcal{L}_{\bar{x}}(A)$ denote the collection of formulas with free variables among $\bar{x}$ and parameters from $A$. For any positive integer $n$, we let $[n] = \{1, \ldots, n\}$. If $X$ is a set and $a$ is an element in $X$ we let $\delta_{a}$ denote the \textbf{Dirac measure} on $X$ concentrating on $a$.

\subsection{Keisler measures}

We refer the reader to \cite[Chapter~7]{Guide} for general background on Keisler measures, but recall many of the essential definitions here. A large number of these definitions originate from \cite{NIP1,NIP2,NIP3}. 

\begin{definition} Fix $A \subseteq \mathcal{U}$ and a tuple of variables $\bar x$. A \defn{Keisler measure} (in the variables $\bar x$ over $A$) is a finitely additive probability measure on $\mathcal{L}_{\bar x}(A)$ (modulo logical equivalence). Equivalently, it is a finitely additive probability measure on the $A$-definable subsets of $\mathcal{U}^{|\bar x|}$. We let $\mathfrak{M}_{\bar x}(A)$ denote the collection of Keisler measures on $\mathcal{L}_{\bar x}(A)$. If $\mu$ is in  $\mathfrak{M}_{\bar x}(A)$ and we want to emphasize the free variables, we will write $\mu$ as either $\mu(\bar x)$ or $\mu_{\bar x}$. When $\mu \in \mathfrak{M}_{\bar x} (A)$ and $\bar y$ is another ordered tuple of variables such that $|\bar y | = | \bar x |$, we let $\mu(\bar y)$ denote the corresponding measure in $\mathfrak{M}_{\bar y }(A)$ that results from the appropriate change of variables. 
\end{definition}

We remark that there is a one-to-one correspondence between Keisler measures on $\mathcal{L}_{\bar x}(A)$ and regular Borel probability measures on the associated Stone space, $S_{\bar x}(A)$. More explicitly, for any Keisler measure $\mu$ on $\mathcal{L}_{\bar x}(A)$, there exists a unique regular Borel probability measure $\tilde{\mu}$ on $S_{\bar x}(A)$ such that for any formula $\varphi(\bar x) \in \mathcal{L}_{\bar x}(A)$, we have that $\mu(\varphi(\bar x)) = \tilde{\mu}([\varphi(\bar x)])$ where $[\varphi(\bar x)] = \{ p \in S_{\bar x}(A)\st\varphi(\bar x) \in p\}$ (see \cite[Page~99]{Guide}). We will often identify $\mu$ and $\tilde{\mu}$, as well as $\varphi(\bar{x})$ and $[\varphi(\bar{x})]$, without comment. 

When $\varphi(\bar x, \bar y)$ is a partitioned formula with \emph{variable} variables $\bar x$ and \emph{parameter} variables $\bar y$, we let $\varphi^{*}(\bar y,\bar x)$ denote the same formula but with \emph{variable} variables $\bar y$ and \emph{parameter} variables $\bar x$.

\begin{definition}\label{cheat} Let $\mu \in \mathfrak{M}_{\bar x}(\mathcal{U})$ and $A$ be a small subset of $\mathcal{U}$. 
(Throughout the text, $A$ will often be a small elementary submodel $M$ of $\mathcal{U}$.)
\begin{enumerate}
    \item The measure $\mu$ is \defn{$A$-invariant} if for any $\bar a, \bar b \in \mathcal{U}^{| \bar y |}$ such that $\bar a \equiv_{A} \bar b$ and for any $\mathcal{L}$-formula $\varphi(\bar x,\bar y)$, we have that $\mu(\varphi(\bar x, \bar a)) = \mu(\varphi(\bar x, \bar b))$. We let $\mathfrak{M}_{\bar x}^{\inv}(\mathcal{U},A)$ be the collection of measures in $\mathfrak{M}_{\bar x}(\mathcal{U})$ which are $A$-invariant. Note that if $\mu$ is $A$-invariant and $B$ is a small subset of $\mathcal{U}$ with $A \subseteq B$, then $\mu$ is also $B$-invariant. 
    \item Let $\mu \in \mathfrak{M}^{\inv}_{\bar x}(\mathcal{U},A)$. Then for any $\Lc(A)$-formula $\varphi(\bar x, \bar y)$, we can define the map $F_{\mu,A}^{\varphi}\colon S_{\bar y}(A) \to [0,1]$ given by $F_{\mu,A}^{\varphi}(q) = \mu(\varphi(\bar x,\bar b))$ where $\bar b \in \mathcal{U}^{\bar{y}}$ is such that $\bar b \models q$. We remark that this map is well-defined since $\mu$ is $A$-invariant. 
    \item Let $\mu \in \mathfrak{M}^{\inv}_{\bar x}(\mathcal{U},A)$. We say that $\mu$ is \defn{Borel-definable (over $A$)} if for every $\mathcal{L}$-formula $\varphi(\bar x,\bar y)$, the map $F_{\mu,A}^{\varphi}$ is a Borel function. 
    \item Let $\mu \in \mathfrak{M}^{\inv}_{\bar x}(\mathcal{U},A)$. We say that $\mu$ is \defn{definable (over $A$)} if for every $\mathcal{L}$-formula $\varphi(\bar x,\bar y)$, the map $F_{\mu,A}^{\varphi}$ is a continuous function. 
    \item Let $\mu \in \mathfrak{M}^{\inv}_{\bar x}(\mathcal{U},A)$ and $\nu \in \mathfrak{M}_{\bar y}(\mathcal{U})$. Suppose that $\mu$ is Borel-definable over $A$. We define the \defn{Morley product} $\mu \otimes \nu$ on $\mathcal{L}_{\bar{x}\bar{y}}(\mathcal{U})$ as follows: For any formula $\varphi(\bar x,\bar y) \in \mathcal{L}_{\bar x \bar y}(\mathcal{U})$, we have that 
    \begin{equation*}
        (\mu \otimes \nu)(\varphi(\bar x,\bar y)) = \int_{S_{\bar y}(A')} F_{\mu,A'}^{\varphi}\, d\nu|_{A'},
    \end{equation*}
    where $A'$ is any small elementary submodel containing $A$ and all the parameters from $\varphi$. The measure $\nu|_{A'}$ is the regular Borel probability measure corresponding to the restriction of $\nu$ to $\mathcal{L}_{\bar y}(A')$. We remark that this product is well-defined and $\mu \otimes \nu \in \mathfrak{M}_{\bar{x}\bar{y}}(\mathcal{U})$ (see \cite[Page 106]{Guide}).  In practice, we often drop the $A'$ from the notation when there is no possibility of confusion, e.g., we write $F_{\mu}^{\varphi}$ instead of $F_{\mu,A'}^{\varphi}$ and $\nu$ instead of $\nu|_{A'}$.
\end{enumerate}
\end{definition}

\begin{remark} The Morley product is an extension of the standard product measure on $\mathcal{L}_{\bar x}(\mathcal{U}) \times \mathcal{L}_{\bar y}(\mathcal{U})$. More explicitly,  
if $A$ is a small subset of $\mathcal{U}$, if $\mu \in \mathfrak{M}_{\bar x}^{\inv}(\mathcal{U},A)$ is Borel-definable, and if $\nu \in \mathfrak{M}_{\bar y}(\mathcal{U})$, then for any $\varphi(\bar x) \in \mathcal{L}_{\bar x}(\mathcal{U})$ and $\psi(\bar y) \in \mathcal{L}_{\bar y}(\mathcal{U})$, we have that 
    \begin{equation*}
        (\mu \otimes \nu)(\varphi(\bar x) \wedge \psi(\bar y)) = \mu(\varphi(\bar x)) \cdot \nu(\psi(\bar y)). 
    \end{equation*}
This follows immediately from the definition of the Morley product.
\end{remark}

We recall that the Morley product is often not self-commutative, i.e., it may be the case that $\mu(\bar x)\otimes \mu(\bar y) \neq \mu(\bar y) \otimes \mu(\bar x)$. We also remind the reader that while the Morley product is well-behaved in the NIP setting, some pathologies arise in the general context. For example, the Morley product of Borel-definable measures may fail to be Borel-definable \cite[Proposition 3.9]{CGH} and the Morley product can fail to be associative on triples of Borel-definable Keisler measures \cite[Proposition 3.10]{CGH}. In this paper, we are focused on structures with the independence property and so we need to be a little careful. The following definitions of \emph{adequacy} and \emph{excellence} are from \cite{gannon2024model}. The motivation behind these definitions was to define classes of measures that interact nicely with the Morley product, but are not constrained to the NIP context.

\begin{definition}\label{adex} Let $\mu \in \mathfrak{M}^{\inv}_{\bar x}(\mathcal{U},M)$ and let $(\bar{x}_i)_{i \geq 1}$ be a sequence of distinct tuples of variables such that $|\bar{x}_i| = |\bar{x}|$. We say that $\mu$ is \defn{$M$-\weak} (or \textbf{adequate over $M$}) when the following hold.
\begin{enumerate}
    \item $\mu$ is Borel-definable over $M$. 
    \item Any finite iteration of the Morley product of $\mu$ is Borel-definable over $M$, i.e.,\ for any $n \geq 1$, the measure $\mu^{(n+1)}(\bar{x}_1,...,\bar{x}_{n+1}) = \mu(\bar{x}_{n+1}) \otimes \mu^{(n)}(\bar{x}_1,...,\bar{x}_{n})$ is Borel-definable over $M$. 
    \item $\mu$ is self-associative, i.e.,\ the measure $\bigotimes_{i=1}^{n} \mu(\bar{x}_i)$ does not depend on the placement of parentheses. 
\end{enumerate}
We say that $\mu$ is \defn{$M$-\strong} (or \textbf{excellent over $M$}) if $\mu$ is $M$-\weak\ and $\mu$ is self-commuting, i.e.,\  $\mu(\bar{x}) \otimes \mu( \bar y) = \mu( \bar y) \otimes \mu( \bar x)$.
\end{definition}

Given an $M$-adequate Keisler measure $\mu$, we can associate to it a Keisler measure in infinitely many variables. This measure corresponds to \emph{generic sampling with respect to $\mu$} and is the central object of study in \cite{gannon2024model}. We let $\mathbf{x} = (x_i)_{i \geq 1}^{\omega}$. 

\begin{definition} Let $\mu \in \mathfrak{M}^{\inv}_{x}(\mathcal{U},M)$ and suppose that $\mu$ is $M$-adequate. We define the Keisler measure $\kmpow{\mu}$ on $\mathcal{L}_{\mathbf{x}}(\mathcal{U})$ as follows: 
\begin{enumerate}
    \item $\mu^{(1)}(x_1)= \mu(x_1)$. 
    \item $\mu^{(n+1)}(x_1,...,x_{n+1}) = \mu(x_{n+1}) \otimes \mu^{(n)}(x_1,...,x_{n})$. 
    \item $\kmpow{\mu} = \bigcup_{n=1}^{\omega} \mu^{(n)}(x_1,...,x_{n})$.
\end{enumerate} 
We note that $\mathbb{P}_{\mu}$ is a Keisler measure on $\mathcal{L}_{\mathbf{x}}(\mathcal{U})$ because it is defined as the union of partial functions (on increasingly larger domains).

Formally, the probability space we are interested in is $(S_{\mathbf{x}}(\mathcal{U}), \mathcal{B} ,\widehat{\kmpow{\mu}})$ where $\mathcal{B}$ is the Borel $\sigma$-algebra of $S_{\mathbf{x}}(\mathcal{U})$ and $\widehat{\kmpow{\mu}}$ is the unique regular Borel probability measure on $S_{\mathbf{x}}(\mathcal{U})$ such that for every $\mathcal{L}_{\mathbf{x}}(\mathcal{U})$-formula $\varphi(x_1,...,x_n)$, we have that 
\begin{equation*}
    \widehat{\kmpow\mu}([\varphi(x_1,...,x_n)]) = \kmpow{\mu}(\varphi(x_1,...,x_n)). 
\end{equation*}
We remark that $\widehat{\kmpow{\mu}}$ exists and is well-defined (e.g., \cite[416Q Proposition (b)]{fremlin2000measure}). As a convention, we will always identify $\widehat{\kmpow{\mu}}$ with $\kmpow{\mu}$.
\end{definition}

\subsection{Space of labeled \texorpdfstring{$\Lc$}{L}-structures}

In this subsection, we assume that $\mathcal{L}$ is a finite relational language. For $R \in \mathcal{L}$, we let $\ar(R)$ denote the arity of $R$. As usual, $\mathcal{U}$ is a monster model of a fixed $\mathcal{L}$-theory $T$. 

\begin{convention} Let $R \in \mathcal{L}$. We write $R^{1}(-)$ to denote $R(-)$ and we write $R^{-1}$ to denote $\neg R(-)$. 
\end{convention}

\begin{definition} We let $\str_{\mathcal{L}}$ be the \defn{space of labeled countable $\mathcal{L}$-structures}. This is the space of $\mathcal{L}$-structures with underlying universe $\mathbb{N}$. For any $\mathcal{L}$-formula $\varphi(x_1,...,x_n)$ and natural numbers $i_1,...,i_n$, we let 
\begin{equation*}
    \llbracket \varphi(i_1,...,i_n)\rrbracket := \{M \in \str_{\mathcal{L}} \st M \models \varphi(i_1,...,i_n)\}. 
\end{equation*}
We recall that $\str_{\mathcal{L}}$ is a compact Hausdorff space with the clopen subbasis 
\begin{equation*}
    \{\llbracket R^{\epsilon}(i_{1},...,i_{\ar(R)}) \rrbracket: \epsilon \in \{1,-1\}, R \in \mathcal{L}, \text{ and } i_{1},...,i_{\ar(R)} \in \mathbb{N}\}. 
\end{equation*}
\end{definition}

Given a type $p \in S_{\mathbf{x}}(\mathcal{U})$, one can associate to $p$ an $\mathcal{L}$-structure $p_{\mathcal{L}}$ as follows: 

\begin{definition} Fix $p \in S_{\mathbf{x}}(\mathcal{U})$. Consider $\mathcal{U}'$ where $\mathcal{U} \prec \mathcal{U}'$. Given a tuple $\mathbf{a} = (a_1,a_2,a_3,...) \in (\mathcal{U}')^{\omega}$ such that $\mathbf{a} \models p$, one can consider the $\mathcal{L}$-structure $\mathbf{a}_{\mathcal{L}}$ which is simply the induced structure from $\mathcal{U}'$ onto the set $\{a_i \st i \geq 1\}$. We remark that if $\mathbf{b} \in (\mathcal{U}')^{\omega}$ and $\mathbf{a}, \mathbf{b} \models p$ then $\mathbf{a}_{\mathcal{L}} \cong \mathbf{b}_{\mathcal{L}}$. In particular, the map $f(a_i) = b_i$ is an isomorphism between the induced structures.  Hence we let $p_{\mathcal{L}}$ be the isomorphism type of any/all the realizations of $p$.
\end{definition}
To simplify some of the arguments, we consider the following smaller (yet more natural) space of types given by 
\begin{equation*}
    S_{\mathbf{x}}^{+}(\mathcal{U}) := \{p \in S_{\mathbf{x}}(\mathcal{U}) \st (\forall i \neq j)\  (x_i \neq x_j) \in p \text{~and~} (\forall a \in \mathcal{U})\ (x =a) \not \in p\}. 
\end{equation*}

\begin{definition}
We say that a Keisler measures $\mu$ in $\mathfrak{M}_{x}(\mathcal{U})$ \defn{does not concentrate on points}
if for every $a \in \mathcal{U}$,
$\mu(x = a) = 0$. Note that $\mu$ is allowed to assign positive measure on particular non-realized types.
\end{definition}

We remark that if a Keisler measure $\mu \in \mathfrak{M}_{x}^{\inv}(\mathcal{U},M)$ is $M$-adequate and $\mu$ does not concentrate on points then $\supp(\kmpow{\mu}) \subseteq S_{\mathbf{x}}^{+}(\mathcal{U})$ (see \cite[Proposition 3.7]{gannon2024model}). This assumption is akin to assuming that our measure is \emph{nowhere degenerate}. In practice, it allows us to assume that we never sample the same point twice. All of the measures that interest us 
have this property.   

\begin{definition} 
Define the map 
$g\colon S_{\mathbf{x}}^{+}(\mathcal{U}) \to \str_{\mathcal{L}}$ to be given by $g(p) = M_{p}$ where for $R \in \mathcal{L}$, we have $R^{M_{p}} := \{(i_{1},...,i_{\ar(R)}) \st R(x_{i_{1}},...,x_{i_{\ar(R)}}) \in p\}$. 
Observe that $g$ is continuous, as the preimage of a clopen set in $\str_{\mathcal{L}}$ is a set of types extending a partial quantifier-free type. 
Moreover, if $\mu$ in $\mathfrak{M}_{x}^{\inv}(\mathcal{U},M)$ is an $M$-adequate measure that does not concentrate on points, then we can consider the pushforward of $\kmpow{\mu}$ along $g$  and construct a measure on $\str_{\mathcal{L}}$ (strictly speaking, this is the pushforward of the restriction of $\mathbb{P}_{\mu}$ to $S_{\mathbf{x}}^{+}(\mathcal{U})$). This pushforward measure is denoted as $g_{*}(\kmpow{\mu})$. 
\end{definition}

\begin{definition} Let $N$ be a countably infinite $\mathcal{L}$-structure. We define
\begin{equation*}
    \mathbb{A}_{N} := \{M \in \str_{\mathcal{L}} \st M \cong N\}.
\end{equation*}
A classic result of Scott shows that $\mathbb{A}_{N}$ is Borel \cite[II.16.6]{kechris2012classical}. We say that a Borel probability measure $\nu$ on $\str_{\mathcal{L}}$ \defn{concentrates on $N$} if $\nu(\mathbb{A}_{N}) = 1$.
Further, let 
\begin{equation*}
\mathbb{B}_{N} := \{p \in S_{\mathbf{x}}^{+}(\mathcal{U}) \st p_{\mathcal{L}} \cong N\}.
\end{equation*}
Notice that $g^{-1}(\mathbb{A}_{N}) = \mathbb{B}_{N}$.
Since $g$ is continuous, $\mathbb{B}_{N}$ is a Borel subset of $S_{\mathbf{x}}^{+}(\mathcal{U})$. We say that  $\kmpow{\mu}$ \defn{concentrates on $N$} if $\kmpow{\mu}(\mathbb{B}_{N}) = 1$. 
\end{definition}

\begin{definition} Let $\nu$ be a Borel probability measure on $\str_{\mathcal{L}}$. We say that $\nu$ is \defn{$\Sym(\mathbb{N})$-invariant} if $\nu$ is invariant under the natural action of $\Sym(\mathbb{N})$ on $\str_{\mathcal{L}}$, i.e., if for every Borel subset $B$ of $\str_{\mathcal{L}}$ and every $\sigma\in \Sym(\mathbb{N})$, we have that $\nu(B) = \nu(\sigma(B))$.
\end{definition}

The following fact is \cite[Theorem 4.16]{gannon2024model}.

\begin{fact}\label{fact:excellent} Suppose that $\mu \in \mathfrak{M}_{x}^{\inv}(\mathcal{U},M)$ and $\mu$ does not concentrate on points. If $\mu$ is $M$-excellent, then $g_*(\kmpow{\mu})$ is $\Sym(\mathbb{N})$-invariant.
\end{fact}

\section{0-1 law for excellent measures}
\label{sec:zero}

In this short section, we tie up a loose end from \cite{gannon2024model} by proving a $0$-$1$ law for generic sampling with respect to excellent measures. The heavy computational aspects of the proof have already been written elsewhere and so what remains is to understand  the relationship between generic sampling and $\Sym(\mathbb{N})$-invariant measures on $\str_{\mathcal{L}}$. These are inherently linked through the transfer map $g\colon S_{\mathbf{x}}^{+}(\mathcal{U}) \to \str_{\mathcal{L}}$ and the fact that the pushforward of an excellent measure along $g$ is $\Sym(\mathbb{N})$-invariant. 

\begin{definition} Let $\nu$ be a Borel probability measure on $\str_{\mathcal{L}}$. We say that $\nu$ is \defn{dissociated} if for every sequence of distinct positive integers $i_1,...,i_n,j_1,...,j_m$ and any  $\mathcal{L}$-formulas $\theta(x_1,...,x_n), \psi(x_1,...,x_m)$, we have that 
\begin{equation*}
    \nu(\llbracket\theta(i_1,...,i_n) \wedge \psi(j_1,...,j_m) \rrbracket )= \nu(\llbracket\theta(i_1,...,i_n)\rrbracket) \cdot \nu(\llbracket\psi(j_1,...,j_m)\rrbracket). 
\end{equation*}
\end{definition}

The next fact is a simplified version of \cite[Lemma 3.9]{gannon2024model}. 

\begin{fact}\label{fact:dissociated} Suppose that $\mu \in \mathfrak{M}_{x}^{\inv}(\mathcal{U},M)$ and $\mu$ is $M$-adequate. Then for any sequence of distinct positive integers $i_1,...,i_n,j_1,...j_m$ and for any $\mathcal{L}$-formulas $\theta(x_{i_1},...,x_{i_n})$ and $\psi(x_{j_1},...,x_{j_m})$, we have that 
\begin{equation*}
    \kmpow{\mu}(\theta(x_{i_1},...,x_{i_n}) \wedge \psi(x_{j_1},...,x_{j_m})) = \kmpow{\mu}(\theta(x_{i_1},...,x_{i_n})) \cdot \kmpow{\mu}( \psi(x_{j_1},...,x_{j_m})). 
\end{equation*}
\end{fact}

\begin{remark} While Fact \ref{fact:dissociated} is true in the context of adequate measures, we only need to apply the result to excellent measures. This restricted result is quite straightforward to prove and easier than the reference provided, so we supply a quick proof. Indeed, by using associativity, commutativity, and the fact that the Morley product is a separated amalgam, one has the following computation: 
\begin{align*}
    \kmpow{\mu}(\theta(x_{i_1},...,x_{i_n}) \wedge \psi(x_{j_1},...,x_{j_m}))  
    &\hspace*{-1.2pt}\overset{(a)}{=} \kmpow{\mu}(\theta(x_1,...,x_{n}) \wedge \psi(x_{n+1},...,x_{n + m})) \\ 
    &=\left(\mu^{(m)}_{(x_{n+1},...,x_{n+m})} \otimes \mu^{(n)}_{(x_{1},...,x_{n})}\right)(\theta(x_1,...,x_{n}) \wedge \psi(x_{n+1},...,x_{n+m})) \\ 
    &= \mu^{(m)}_{(x_{n+1},...,x_{n+m})} (\psi(x_{n+1},...,x_{n+m})) \cdot \mu^{(n)}_{(x_{1},...,x_{n})}(\theta(x_{1},...,x_{n})) \\ 
    &= \kmpow{\mu}(\psi(x_{n+1},...,x_{n+m})) \cdot \kmpow{\mu}(\theta(x_1,...,x_{n}))\\
    &\hspace*{-1.2pt}\overset{(b)}{=} \kmpow{\mu}(\psi(x_{j_1},...,x_{j_m})) \cdot \kmpow{\mu}(\theta(x_{i_1},...,x_{i_n})). 
\end{align*}
Equations $(a)$ and $(b)$ require self-commutativity, i.e., excellence.  
\end{remark}

\begin{definition} Suppose that $\nu$ is a Borel probability measure on $\str_{\mathcal{L}}$. We say that a Borel subset $B$ of $\str_{\mathcal{L}}$ is \defn{$\nu$-a.s.\ $\Sym(\mathbb{N})$-invariant} if for every $\sigma \in \Sym(\mathbb{N})$, we have that $\nu(B\triangle \sigma^{-1}(B)) = 0$. We say that a $\Sym(\mathbb{N})$-invariant probability measure $\nu$ on $\str_{\mathcal{L}}$ is \defn{ergodic} if  $\nu(B) \in \{0,1\}$ for every $\nu$-a.s. $\Sym(\mathbb{N})$-invariant Borel subset $B$ of $\str_{\mathcal{L}}$.
\end{definition}

A proof of the following fact can be found in \cite[Corollary 2.46]{ackerman2015representations} (see also \cite[Lemma 7.35]{kallenberg2005probabilistic}).

\begin{fact}\label{fact:ergodic} Suppose that $\nu$ is a $\Sym(\mathbb{N})$-invariant Borel probability measure on $\str_{\mathcal{L}}$. Then the following are equivalent: 
\begin{enumerate}
    \item $\nu$ is dissociated. 
    \item $\nu$ is ergodic. 
\end{enumerate}
\end{fact}

\begin{proposition}[0-1 Law]\label{prop:01} Let $\mu \in \mathfrak{M}^{\inv}_{x}(\mathcal{U},M)$, and suppose that $\mu$ does not concentrate on points and $\mu$ is $M$-excellent. Then  $\kmpow{\mu}(\mathbb{B}_{N}) \in \{0,1\}$ for any $\mathcal{L}$-structure $N$.
\end{proposition}

\begin{proof} Since $\mu$ is $M$-excellent and does not concentrate on points, we have that $g_{*}(\kmpow{\mu})$ is $\Sym(\mathbb{N})$-invariant by Fact \ref{fact:excellent}. By Fact \ref{fact:dissociated}, we conclude that $g_{*}(\kmpow{\mu})$ is dissociated, and so by Fact \ref{fact:ergodic}, we have that $g_{*}(\kmpow{\mu})$ is ergodic. The set $\mathbb{A}_{N}$ is clearly $\Sym(\mathbb{N})$-invariant and Borel (\cite[II.16.6]{kechris2012classical}). Thus, 
\begin{equation*}
    \kmpow\mu(\mathbb{B}_N) =\kmpow\mu(g^{-1}(\mathbb{A}_N))  = g_{*}(\kmpow{\mu})(\mathbb{A}_N) \in \{0,1\}. \qedhere 
\end{equation*}
\end{proof}

\noindent We are left with the following question: 

\begin{question} Is the conclusion of \cref{prop:01} true if the hypothesis of \emph{excellence} is replaced with \emph{adequacy}? What can be said about dissociated measures which are not $\Sym(\mathbb{N})$-invariant? 
\end{question}

\section{Graphons, generic sampling, and the Rado graph}
\label{sec:graphon}

The purpose of this section is to focus on generic sampling over the Rado graph. We first consider generic sampling with respect to $\EmptySet$-invariant Keisler measures. These measures were classified in the late 1980s by Albert \cite{albert1989random}. We prove that generic sampling with respect to such measures typically results in a copy of the Rado graph, and we also observe an interesting edge case phenomenon involving threshold graphs. We then move on to the main result of the section, the \emph{graphon representation theorem}.

Throughout this section, we let $T$ be the theory of the Rado graph in the language $\mathcal{L} = \{R(x,y)\}$. We let $\mathcal{U}$ be a monster model of $T$ and fix a countable elementary submodel $M$. 

\subsection{$\EmptySet$-Invariant measures}

In \cite[Theorem 1]{albert1989random}, Albert gave a classification of $\EmptySet$-invariant Keisler measures on the Rado graph. These measures are in one-to-one correspondence with Borel probability measures on $[0,1]$ (which we denote as $\mathcal{M}([0,1])$). More explicitly, for every Borel probability measure $\nu$ on $[0,1]$ there exists a unique $\EmptySet$-invariant Keisler measure $\mu_{\nu} \in \mathfrak{M}_{x}(\mathcal{U})$ such that for any sequence of distinct points $a_1,...,a_n,b_1,...,b_m$ from $\mathcal{U}$, we have 
\begin{equation*}
    \mu_{\nu} \left(\bigwedge_{i=1}^{n}R(x,a_i) \wedge \bigwedge_{j=1}^{m} \neg R(x,b_j) \right) = \int_{0}^{1} t^{n} (1-t)^{m} d\nu, 
\end{equation*}
and moreover, all $\EmptySet$-invariant Keisler measures on the Rado graph arise in this way. More generally, all the measures defined above, i.e., of the form $\mu_{\nu}$, are $\EmptySet$-definable, which implies that they are adequate over any model.

We now describe all measures $\mathbb{P}_{\mu_{\nu}}$ that arise via generic sampling (with respect to the measures $\mu_{\nu}$ described above).  
We first observe two straightforward facts, stated in the next proposition.
\begin{proposition}\label{prop:emptyset1and2}
Let $K_\infty$ denote a clique with a countably infinite vertex set, and $Q_\infty$ an anti-clique with a countably infinite vertex set. Then we have: 
\begin{enumerate}
    \item $\kmpow{\mu_{\delta_1}}(\mathbb{B}_{K_{\infty}}) = 1$. 
    \item  $\kmpow{\mu_{\delta_0}}(\mathbb{B}_{Q_{\infty})} =1$. 
\end{enumerate}
\end{proposition}

When $\nu$ is a non-trivial mixture of point masses on $[0, 1]$, the description of $\kmpow{\mu_\nu}$ is more complicated. The next proposition shows that when our Borel probability measure $\nu$ concentrates on $\{0,1\}$, then $\mathbb{P}_{\mu_{\nu}}$ does not concentrate on any countable graph. The intuition is that generic sampling with respect to $\mu_{\nu}$ behaves similarly to the construction of threshold graphs. We refer the reader to \cref{remark:threshold} for details.  

\begin{proposition}\label{prop:emptyset3} 
    Let $\nu$ be a Borel probability measure on $[0, 1]$ and suppose $\nu = r\delta_{1} + (1-r)\delta_{0}$ for some $r \in (0,1)$. Then for any $\mathcal{L}$-structure $N$, we have that $\kmpow{\mu_{\nu}}(\mathbb{B}_N) = 0$. Further, $g_{*}(\mathbb{P}_{\mu_{\nu}})$ is not $\Sym(\mathbb{N})$-invariant.
\end{proposition}

\begin{proof}
Let  $p_{+} \defas \{R(x,b) \st b \in \mathcal{U}\}$ and $p_{-} \defas \{\neg R(x,b) \st b \in \mathcal{U}\}$. Then we have
$\mu_{\nu} = r\delta_{p_{+}} + (1-r)\delta_{p_{-}}$.
Consider the projections $\pi_{i} \colon S_{\mathbf{x}}(\mathcal{U}) \to S_{x_i}(\mathcal{U})$. For each $m$ and $A \subseteq [m]$, consider the set
\begin{equation*}
    X_{A} \defas \bigcap_{i \in A} \pi_i^{-1}(\{p_+\}) \cap \bigcap_{i \in [m] \backslash A} \pi_i^{-1}(\{p_{-}\}). 
\end{equation*}
The following observations are immediate:  
\begin{enumerate}[label=($\alph*$)]
    \item For each $m$, the set $Y_{m} \defas \{X_{A} \st A \subseteq [m]\}$ forms a partition of $\supp(\kmpow{\mu_{\nu}})$. 
    \item If $k > m$, then $Y_{k}$ refines $Y_{m}$.
    \item If $A \subseteq [m]$, then  $\kmpow{\mu_{\nu}}(X_{A}) = r^{|A|}(1-r)^{|[m] \backslash A|}$. 
\end{enumerate}  

\begin{claim}
\label{claim:strange}
For each $m\geq 2$ and distinct $A, B \subseteq [m]$ such that $A$ and $B$ differ by an element other than $1$ (i.e., there exists some $k \in A \triangle B$ such that $k \neq 1$), if $p \in X_{A}$ and $q \in X_{B}$ then $p_{\mathcal{L}} \not \equiv q_{\mathcal{L}}$.
\end{claim}

Note that from the proof of \cref{claim:strange} below, when $A\triangle B = \{1\}$, the conclusion of
\cref{claim:strange} need not hold since the first sampled vertex
is vacuously both connected and disconnected to 
all vertices that arose before it.

Using \cref{claim:strange}, it follows that for any $\mathcal{L}$-structure $N$, we have $g_{*}(\kmpow{\mu_{\nu}})(\mathbb{A}_N) = 0$, and so $\kmpow{\mu_{\nu}}(\mathbb{B}_N) = 0$, establishing the first half of the proposition. Indeed, \cref{claim:strange} allows one to partition the support of $g_{*}(\mathbb{P}_{\mu_{\nu}})$ into collections of $\mathcal{L}$-structures such that each part of the partition has arbitrarily small measure.

Further, since $\nu = r\delta_1 + (1-r)\delta_0$, we have that
\begin{equation*}
    \mathbb{P}_{\mu_{\nu}}(R(x_3,x_2) \wedge R(x_2,x_1)) = r^2,
\end{equation*}
and
\begin{equation*}
    \mathbb{P}_{\mu_{\nu}}(R(x_3,x_2) \wedge R(x_3,x_1)) = r^2 + r(1-r) = r.  
\end{equation*}
Hence $g_{*}(\mathbb{P}_{\mu_{\nu}})$ is not $\Sym(\mathbb{N})$-invariant, establishing the second half of the proposition.

\begin{proof}[Proof of \cref{claim:strange}] We may interpret each element of $\supp(\mathbb{P}_{\mu_{\nu}})$ as instructions for constructing an infinite graph. If $p \in \supp(\mathbb{P}_{\mu_{\nu}})$, then $p_{\mathcal{L}}$ is isomorphic to the graph on $\{1,2,3,...\}$ built in stages as follows: At stage $\ell$ where $\ell \geq 2$, we add the vertex $\ell$ so that it is either adjacent to every vertex $j < \ell$ or is non-adjacent to every vertex $j < \ell$. Thus if $p \in \supp(\kmpow{\mu_{\nu}})$, then $p_{\mathcal{L}}$ has a definable clique and a definable anti-clique which are both non-trivial (i.e., have at least two vertices) with probability 1. Let
\begin{equation*}
    W_1(x) := \forall y \forall z ((R(x,y) \wedge R(x,z) \wedge y\neq z) \to R(y,z)). 
\end{equation*}
Notice that $W_1$ defines the set of vertices $\ell$ that are not adjacent to any vertices $j$ such that $j < \ell$, along with possibly an initial segment of elements that forms a clique. To remove this initial segment, we define
\begin{equation*}
    W_2(x) := W_1(x) \wedge \neg \exists y (W_1(y) \wedge R(x,y)). 
\end{equation*}
Notice that $W_2(x)$ defines an anti-clique and $\neg W_2(x)$ defines a clique. So, suppose that $A,B \subseteq [m]$ for some $m \geq 2$ with $p \in X_{A}$ and $q \in X_{B}$ where there exists $k \in A \triangle B$ such that $k \neq 1$. Let $k'$ be the smallest element of $A \triangle B$ such that $k' \neq 1$ and suppose without loss of generality that $k' \in A$. For $C \in \{A,B\}$, define $C' = C \cap \{1,...,k'\}$ and let $t = |A'|$. Let
\begin{equation*}
    \Psi := \exists x \exists!y_1,...,y_{t-1}   \bigwedge_{i=1}^{t-1} \left(\neg W_{2}(y_i) \wedge \neg R(x,y_i) \right). 
\end{equation*}
Then $q_{\mathcal{L}} \models \Psi$ while $p_{\mathcal{L}} \models \neg \Psi$, establishing the claim. 
\end{proof}
This completes the proof of \cref{prop:emptyset3}.
\end{proof}

\begin{remark}\cref{prop:emptyset1and2,prop:emptyset3} together show that the operation $\nu \mapsto \kmpow{\mu_\nu}$ is not linear. Specifically, for $r \in (0,1)$ and Borel probability measures $\nu_0$ and $\nu_1$ on $[0,1]$,
while $\nu = r \nu _0 + (1-r) \nu_1$ 
does imply that
$\mu_\nu = r \mu_{\nu_0} +  (1-r) \mu_{\nu_1}$,
it does not imply that
$\kmpow{\mu_\nu} = r \kmpow{\mu_{\nu_0}} + (1-r)\kmpow{\mu_{\nu_1}}$. 
\end{remark}

\begin{remark}\label{remark:threshold}
When $r\in (0, 1)$ and $\nu = r \delta_0 + (1- r)\delta_1$, while the measure $g_{*}(\mathbb{P}_{\mu_\nu})$ does not concentrate on the isomorphism class in $\str_{\mathcal{L}}$ of any particular graph, it does still concentrate on a particularly simple collection of graphs. Specifically, it concentrates on the collection of 
\emph{threshold graphs} \cite{diaconis2008threshold,blekherman2024threshold},
which are constructed according to the underlying ordering on $\Nats$. 
In this construction, for each vertex $k \in \Nats$, with (independent) probability $r$ the vertex $k$ has an edge with all the vertices $0, \ldots, k-1$, and otherwise an edge with none of them.
Moreover, $\mu_{\nu}$ is not $M$-excellent, and the measure $\kmpow{\mu_{\nu}}$ is not $\Sym(\mathbb{N})$-invariant.
\end{remark}

\begin{remark} It is straightforward to check that if $\nu = \delta_{r}$ for some $r \in (0,1)$, then $\mu_{\nu}$ corresponds to a weighted coin-flipping measure. More explicitly, for any sequence of distinct points $a_1,...,a_n,b_1,...,b_m$ in $\mathcal{U}$, we have
\begin{equation*}
    \mu_{\nu} \left(\bigwedge_{i =1}^{n} R(x,a_i) \wedge \bigwedge_{j = 1}^{m} \neg R(x,b_j) \right) = r^{n}(1-r)^{m}. 
\end{equation*}
Any such measure is $M$-excellent \cite[Example 2.4]{gannon2024model} and does not concentrate on points. Thus $g_{*}(\kmpow{\mu_{\nu}})$ is $\Sym(\mathbb{N})$-invariant by Fact \ref{fact:excellent}. 
\end{remark}

\begin{proposition}\label{prop:emptyset4}
Let $\nu \in \mathcal{M}([0,1])$ and suppose that $\nu$ is not 
the measure $r\delta_{1} + (1-r)\delta_{0}$ for any $r \in [0,1]$.
Then $\kmpow{\mu_{\nu}}(\mathbb{B}_{\Rado}) =1$, where $\Rado$ is the Rado graph. 
\end{proposition}

\begin{proof}
This is a generalization of \cite[Example 5.1]{gannon2024model}. First note that the substructure induced on any subset of the underlying set of a graph is also a graph. For each positive integer $n$ and subset $A$ of $[n]$, define the set 

\begin{equation*}
    K_{n,A} \coloneqq \bigcup_{\ell=1}^{\infty} \left[\bigwedge_{i \in A }R(x_i,x_\ell) \wedge \bigwedge_{j \in [n] \backslash A} \neg R(x_{j},x_\ell)\right]. 
\end{equation*}
For each such pair $(n,A)$, we will prove that $\kmpow{\mu_\nu}(K_{n,A}) = 1$. Fix a pair $(n,A)$. For each $\ell \in \mathbb{N}$, let 
\begin{equation*}\psi_{\ell}(\bar{x},x_{\ell}) \coloneqq  \bigvee_{i \in A} \neg R(x_i,x_\ell) \vee \bigvee_{j \in [n] \backslash A} R(x_{j},x_\ell). 
\end{equation*} 
Now consider the following computation: 
\begin{align*}
    \kmpow{\mu_\nu}(K_{n,A})  &=\kmpow{\mu} \left( \bigcup_{\ell=1}^{\infty} \left[\bigwedge_{i \in A}^{n}R(x_i,x_\ell) \wedge  \bigwedge_{j \in [n] \backslash A}\neg R(x_{j},x_\ell)\right]  \right) \\
    &=\lim_{k \to \infty} \kmpow{\mu_\nu} \left( \left[\bigvee_{\ell=1}^{k} \left( \bigwedge_{i \in A}R(x_i,x_\ell) \wedge \bigwedge_{j \in [n] \backslash A} \neg R(x_{j},x_\ell) \right) \right]  \right)\\
    &=\lim_{k \to \infty} \left( 1 - \kmpow{\mu_\nu} \left( \left[\bigwedge_{\ell=1}^{k} \left( \bigvee_{i \in A} \neg R(x_i,x_\ell) \vee \bigvee_{j \in [n] \backslash A} R(x_{j},x_\ell) \right) \right]  \right) \right) \\
    &=1 - \lim_{k \to \infty} \kmpow{\mu_\nu} \left( \left[\bigwedge_{\ell=1}^{k} \left( \bigvee_{i \in A} \neg R(x_i,x_\ell) \vee \bigvee_{j \in [n] \backslash A} R(x_{j},x_\ell) \right) \right]  \right)\\
    &\geq 1 - \lim_{\substack{k > n \\ k \to \infty}} \kmpow{\mu_\nu} \left( \left[\bigwedge_{\ell >  n }^{k} \left( \bigvee_{i \in A} \neg R(x_i,x_\ell) \vee \bigvee_{j \in [n] \backslash A} R(x_{j},x_\ell) \right) \right]  \right)\\
    &\overset{}{=}1 - \lim_{\substack{k > n \\ k \to \infty}} \int_{S_{\bar{x}}(M)} \prod_{\ell > n }^k F_{{\mu_{\nu}(x_{\ell})}}^{\psi_\ell} d\mu_\nu^{(n)} \\
    &\overset{}{=}1 - \lim_{\substack{k > n \\ k \to \infty}} \int_{S_{\bar{x}}(M)} r^{k-n} d\mu_\nu^{(n)} \\
    &= 1 - 0 \\
    &= 1.
\end{align*}

We notice that if $q \in S_{\bar{x}}(M)$ and $(\bar{a},\bar{b}) \models q((x_i)_{i \in A};(x_i)_{i \in [n]\backslash A})$, then we have 

\begin{align*}
    r = F_{\mu_{\nu}(x_{\ell})}^{\psi_\ell}(q) &= {\mu_{\nu}} \left(\bigvee_{i \in A} \neg R(a_i,x_\ell) \vee \bigvee_{j \in [n] \backslash A} R(b_j,x_\ell) \right) \\ &= 1 - {\mu_{\nu}} \left(\bigwedge_{i \in A} \neg R(a_i,x_\ell) \wedge \bigwedge_{j \in [n] \backslash A} R(b_j,x_\ell) \right) \\ &= 1 - \int_0^1 t^{|A|} (1-t)^{n-|A|} d\nu.
\end{align*}
The map $t \mapsto t^{|A|}(1-t)^{n-|A|}$ is a continuous function. By our hypothesis about $\nu$, there exists a point $t_{*} \in \supp(\nu)$ such that $t_{*} \not\in \{0, 1\}$. Thus, $ t_{*}^{|A|} (1-t_*)^{n-|A|} > 0$ and so $\int_0^1 t^{|A|} (1-t)^{n-|A|} d\nu > 0$. 
This implies that $r < 1$. 

Let $K = \bigcap_{n \geq 1} \bigcap_{A \subseteq n} K_{n,A}$. Then $\kmpow{\mu_{\nu}}(K) = 1$ and for every $p \in K$, we have $p_{\mathcal{L}} \cong \Rado$. This concludes the proof. 
\end{proof}

\begin{warning} There are $\EmptySet$-invariant measures that are not excellent. Simple examples can be found in \cref{prop:emptyset3}. We consider one more below. 
\end{warning}

Let $\Leb$ denote the restriction of Lebesgue measure to $[0,1]$.

\begin{lemma}
The measure $g_{*}(\kmpow{\mu_{\Leb}})$ is not $\Sym(\mathbb{N})$-invariant.
\end{lemma}
\begin{proof}
Let $b,c \in M$ and consider $\psi(x,y) := R(x,y) \wedge R(x,b) \wedge \neg R(y,c)$. Recall that $\psi^{*}(y,x)$ denotes the same formula as $\psi(x,y)$ but with the roles of the variables and parameters reversed. Then we have

\begin{align*}
    (\mu_{\Leb}(x) \otimes \mu_{\Leb}(y)) (R(x,y) \wedge R(x,b) \wedge \neg R(y,c)) = \int_{S_{y}(M)} F_{\mu_{\Leb}}^{\psi} d\mu_{\Leb}(y).
\end{align*}
Notice that
\begin{equation*}
    F_{\mu_{\Leb}}^{\psi}(q) = \mu_{\Leb}(R(x,d) \wedge R(x,b) \wedge \neg R(d,c)) = \mathbf{1}_{\neg R(y,c)} \cdot \int_{0}^{1} t^{2} d\Leb = \mathbf{1}_{\neg R(y,c)} \cdot \tfrac{1}{3}. 
\end{equation*}
Hence, 
\begin{equation*}
     \int_{S_{y}(M)} F_{\mu_{\Leb}}^{\psi} d\mu_\Leb(y) = \int_{S_{y}(M)} \mathbf{1}_{\neg R(y,c)} \cdot \tfrac{1}{3} d\mu_\Leb(y) = \tfrac{1}{3} \mu_{\Leb}(\neg R(y,c)) = \tfrac{1}{3} \cdot \tfrac{1}{2} = \tfrac{1}{6}. 
\end{equation*}
Now we compute the Morley product in the other order:
\begin{align*}
       \mu_{\Leb}(y) \otimes \mu_{\Leb}(x) (R(x,y) \wedge R(x,b) \wedge \neg R(y,c)) = \int_{S_{x}(M)} F_{\mu_{\Leb}}^{\psi^*} d\mu_\Leb(x). 
\end{align*}
Notice that 
\begin{equation*}
    F_{\mu_{\Leb}}^{\psi^{*}}(q) = \mu_{\Leb}(R(d,y) \wedge R(d,b) \wedge \neg R(y,c)) = \mathbf{1}_{R(x,b)} \cdot  \int_{0}^{1} t (1-t)d\Leb = \mathbf{1}_{R(x,b)} \cdot \tfrac{1}{6}.
\end{equation*}
Hence, 
\begin{equation*}
    \int_{S_{x}(M)} F_{\mu_{\Leb}}^{\psi^*} d\mu_\Leb(x) = \int_{S_{x}(M)} \mathbf{1}_{R(x,b)} \cdot \tfrac{1}{6}d\mu = \tfrac{1}{2} \cdot \tfrac{1}{6} = \tfrac{1}{12}. 
\end{equation*}
Notice that 
\begin{equation*}
    g_*(\kmpow{\mu_\Leb})(\llbracket R(4,3) \wedge R(4,2) \wedge \neg R(3,1)\rrbracket) = \tfrac{1}{6},
\end{equation*}
while
\begin{equation*}
    g_*(\kmpow{\mu_\Leb})(\llbracket R(3,4) \wedge R(3,2) \wedge \neg R(4,1)\rrbracket) = \tfrac{1}{12}. \qedhere 
\end{equation*}
\end{proof}

\subsection{Graphon representation theorem}

In this subsection, we prove the graphon representation theorem. We begin by fixing some conventions and recalling some basic machinery associated with graphons. 

Recall that $\mathcal{U}$ is a monster model of the theory of the Rado graph and $M$ is a countable elementary substructure. If $A \subseteq \mathcal{U}$, we let $S_{x}^{*}(A)$ denote the collection of types in $S_{x}(A)$ that are not realized in $A$. Since $M$ is a countable model, the type space $S^{*}_{x}(M)$ is uncountable and Polish and thus is Borel-isomorphic to $[0,1]$. As a convention, we will sometimes write $S_{x}^{*}(M)$ simply as $S_{x}^{*}$, dropping the $M$ to shorten notation. 

We let $\lambda$ be the standard fair coin-flipping measure on $\mathcal{L}_{x}(\mathcal{U})$. This measure is uniquely characterized as follows: For any sequence $a_1,...,a_n,b_1,...,b_m$ of distinct elements in $\mathcal{U}$, we have that 
\begin{equation*}
    \lambda\left( \bigwedge_{i=1}^{n} R(x,a_i) \wedge \bigwedge_{j=1}^{m} \neg R(x,b_j) \right) = \frac{1}{2^{n+m}} . 
\end{equation*}
We remark that $\lambda$ concentrates on $S_{x}^{*}(\mathcal{U})$ and the regular Borel probability measure corresponding to the restriction $\lambda|_{M}$ on $\mathcal{L}_{x}(M)$ likewise concentrates on $S^{*}_{x}(M)$. 
Recall that we use the symbol $\Leb$ to denote the Lebesgue measure restricted to $[0,1]$. 

\begin{definition} 
A \defn{graphon} is a measurable function $W \colon [0,1]^{2} \to [0,1]$ such that $W(x,y) = W(y,x)$ for all $x,y \in [0,1]$. 
A graphon is \defn{Borel} when it is a Borel measurable function. 
\end{definition}

Note that if $W$ is Borel, then for any $a \in [0,1]$, the section $W(a,-)\colon [0,1] \to [0,1]$ is also Borel.

For any graphon $W$, we can construct a $\Sym(\mathbb{N})$-invariant probability measure $\mathbb{G}(\Nats, W)$ on $\str_\mathcal{L}$ for $\mathcal{L}= \{R\}$. This measure concentrates on $\mathcal{L}$-structures that are graphs.

\begin{definition}
Let $W$ be a graphon. We define $\mathbb{G}(\Nats, W)$ to be the distribution of the following
random graph (with edge relation $R$) on the vertex set $\Nats$. 
\begin{itemize}
\item[(1)] Select an $\Leb$-i.i.d.\ sequence $(x_i)_{i \in \w}$. 

\item[(2)] Select an $\Leb$-i.i.d.\ array $(\gamma_{i, j})_{i < j \in \Nats}$.

\item[(3)] For $i < j \in \Nats$, 
\[
\mathbb{G}(\Nats,W) \models R(i, j)\text{ if and only if }\gamma_{i, j} < W(x_i, x_j).
\]
\end{itemize}
\end{definition}

For every Borel graphon $W$ we will construct a corresponding Keisler measure $\mu_{W}$. Note that since every graphon agrees with a Borel graphon almost everywhere, we lose no real generality by the restriction to Borel graphons.

\begin{definition} Let $n\ge 1$ and let $G$ be a graph on vertex set $\{1, \ldots, n\}$.
The \defn{graph formula corresponding to $G$} (in the language with binary relation symbol $R$) is 
\begin{equation*}
    \varphi_{G}(x_1,...,x_n) \defas \bigwedge_
    {\substack{i, j\in[n] \\G \models R(i,j)}} R(x_i,x_j) \wedge \bigwedge_{\substack{i,j\in[n] \\ G \models \neg R(i,j)}} \neg R(x_i,x_j). 
\end{equation*}
\end{definition}

\begin{convention} Suppose $G$ is a graph on vertex set $\{1,...,n\}$ and $W$ is a graphon. Given $i,j \in [n]$, for notational compactness we will sometimes write $W^{G(i,j)}(x,y)$ to denote $W(x,y)$ if $G \models R(i,j)$ and $1 - W(x,y)$ if $G \models \neg R(i,j)$. 
\end{convention}

\begin{lemma}
Let $W$ be a graphon, and $n \ge 1$.
Suppose $G$ is a graph on vertex set \{1, \ldots, n\}.
Then
\begin{equation*}
    \mathbb{G}(\mathbb{N}, W)(\llbracket \varphi_{G}(i_1,...,i_n)\rrbracket) =
    \int_{(t_1,..,t_n) \in [0,1]^{n}} 
    \prod_{1 \leq \ell < j \leq n}
   W^{G(\ell, j)}(t_\ell,t_j) \,
    d\Leb^{n},
\end{equation*}
for any distinct $i_1, \ldots, i_n \in \Nats$.
Notice that the value above does not depend on $\{i_1,...,i_n\}$, but only on the structure of $G$. 
\end{lemma}
\begin{proof} See the beginning of Chapter 7 of \cite{lovasz2012large}.  
\end{proof}

We recall that every ergodic $\Sym(\mathbb{N})$-invariant measure on the space of countable graphs arises via graphon sampling. Thus, by associating Keisler measures with graphons in a coherent way, we will prove that every ergodic $\Sym(\mathbb{N})$-invariant measure on the space of countable graphs arises also via generic sampling. 

\begin{lemma}\label{lemma:rep}
Suppose $\nu$ is an ergodic
$\Sym({\Nats})$-invariant measure on $\str_{\mathcal{L}}$ concentrated on the collection of graphs. Then there is a Borel graphon $W_\nu$ such that $\mathbb{G}(\Nats, W_\nu) = \nu$. 
\end{lemma}
\begin{proof}
By \cite[Theorem 11.52]{lovasz2012large},
there exists a Lebesgue measurable graphon $W'_{\nu}$ such that $\mathbb{G}(\mathbb{N},W'_{\nu}) = \nu$. By standard measure theory, there is a Borel graphon $W_{\nu}$ that differs from $W_{\nu}'$ on a nullset (e.g., see \cite[Exercise 13.11]{lovasz2012large}); this difference on a nullset does not affect the invariant measure, i.e., $\mathbb{G}(\mathbb{N},W_{\nu}) = \mathbb{G}(\mathbb{N},W'_{\nu}) = \nu$.
\end{proof}

\begin{remark}
\label{WF}
Let $W$ be a Borel graphon. Let $n \ge $1, and again 
suppose $G$ is a graph on vertex set $\{1, \ldots, n\}$.
Since $S_{x}^{*}(M)$ and $[0,1]$ are Borel-isomorphic, there exists an isomorphism of measure spaces $F\colon (S_{x}^{*}(M),\lambda) \to ([0,1],\Leb)$ (see, e.g., \cite[433X (f)]{fremlin2000measure}). 
Hence
\begin{equation*}
\mathbb{G}(\mathbb{N}, W)(\llbracket \varphi_{G}(i_1,...,i_n)\rrbracket)
=
    \int_{(t_1,..,t_n) \in S^{*}_{x}(M)^{n}}  
    \prod_{1 \leq \ell < j \leq n} 
    W^{G(\ell, j)}(F(t_\ell),F(t_j)) \,
    d\lambda^n,
\end{equation*}
for any distinct $i_1, \ldots, i_n \in \Nats$.
\end{remark}
Throughout this subsection, we will fix an isomorphism $F$ as in \cref{WF}. We will often write $W$ to mean the function $W_F$ defined by
$W_{F}(t_\ell,t_j) = W(F(t_\ell),F(t_j))$ for $t_\ell, t_j \in S^*_{x}(M)$. 

Given a Borel graphon $W$, we may now associate to it a Keisler measure $\mu_W$ on $\mathcal{L}_{x}(\mathcal{U})$, where $\mathcal{U}$ is a monster model of the theory of the Rado graph. We remark that this association implicitly depends on the choice of isomorphism $F$ from \cref{WF}. 

First, we introduce some terminology.
We say that $\varphi(x) \in \mathcal{L}_{x}(\mathcal{U})$ is \emph{basic} if it is of the form
\begin{equation*}
    \varphi(x) = \bigwedge_{a \in A}  R(x,a) \wedge  \bigwedge_{b \in B} \neg R(x,b) \wedge \bigwedge_{c \in C} R(x,c)\wedge \bigwedge_{d \in D} \neg  R(x,d) \wedge E(x),
\end{equation*}
where $A$ and $B$ are finite disjoint subsets of $M$, $C$ and $D$ are finite disjoint subsets of $\mathcal{U}\backslash M$, and $E(x)$ is a formula in the language of equality with parameters from $\mathcal{U}$. We say that a basic formula $\varphi(x)$ is \emph{trivial} if $E(x) \vdash x \in F$ for some finite set $F \subseteq \mathcal{U}$, and \emph{non-trivial} otherwise. By quantifier elimination, every formula in $\mathcal{L}_{x}(\mathcal{U})$ is equivalent to a finite disjoint union of basic formulas, and thus it suffices to define $\mu_{W}$ on basic formulas. 

\begin{definition}\label{def:measure} Suppose that $W$ is a Borel graphon. We define a measure $\mu_W$ on $\mathcal{L}_{x}(\mathcal{U})$ as follows: Let $\varphi(x)$ be a basic formula in $\mathcal{L}_{x}(\mathcal{U})$. If $\varphi(x)$ is trivial, set $\mu_{W}(\varphi(x)) = 0$. If $\varphi(x)$ is non-trivial, let 
\begin{equation*}
\mu_{W}(\varphi(x)) = \int_{p \in S_{x}^{*}(M)} \mathbf{1}_{\psi_{A}(x) \wedge \psi_{B}(x)} \prod_{c \in C} W(p,\tp(c/M)) \prod_{d \in D} (1 - W(p,\tp(d/M)))d\lambda. 
\end{equation*}
where $\psi_{A}(x) := \bigwedge_{a \in A} R(x,a)$ and $\psi_{B}(x) := \bigwedge_{b \in B} \neg R(x,b)$.

We will sometimes write the integrand as $f_{\varphi}(\cdot)$, so that $\mu_{W}(\varphi(x)) = \int_{p \in S_{x}^{*}(M)} f_{\varphi}(p) d\lambda$.  
\end{definition}

\begin{proposition}\label{prop:inv} The function $\mu_{W}$ in Definition \ref{def:measure} is an $M$-invariant Keisler measure.
Moreover, $\mu_{W}$ does not concentrate on points.
\end{proposition}

\begin{proof} It is clear that the quantity $\mu_{W}(\varphi(x))$ only depends on the types of the parameters in $\varphi(x)$ over $M$ and hence is $M$-invariant. Moreover, it is also clear that $\mu_{W}$ does not concentrate on points since the measure of any finite subset of our model is $0$. Thus it remains to show that $\mu$ is a Keisler measure.
By an induction argument, it suffices to show that if $\varphi(x)$ is a non-trivial basic formula and $e \in \mathcal{U}$ is a parameter, then 
\begin{equation*}
    \mu_{W}(\varphi(x)) = \mu_{W}(\varphi(x) \wedge R(x,e))  + \mu_{W}(\varphi(x) \wedge \neg R(x,e)). 
\end{equation*}
If $e \in M$, then this equality holds as $\mathbf{1}_{\psi_{A}(x) \wedge \psi_{B}(x)} = \mathbf{1}_{\psi_{A}(x) \wedge \psi_{B}(x) \wedge R(x, e)} + \mathbf{1}_{\psi_{A}(x) \wedge \psi_{B}(x) \wedge \neg R(x, e)}$. If $e \in \mathcal{U} \backslash M$, then we have
\begin{align*}
    \mu_{W}(\varphi(x) \wedge R(x,e)) & = \int_{p \in S_{x}^{*}(M)} f_{\varphi}(p) W(p,\tp(e/M))d\lambda
\intertext{and}
    \mu_{W}(\varphi(x) \wedge \neg R(x,e)) & = \int_{p \in S_{x}^{*}(M)} f_{\varphi}(p) (1 -W(p,\tp(e/M))d\lambda; 
\intertext{by linearity of integration, we have that}
    \mu_{W}(\varphi(x) \wedge R(x,e)) + \mu_{W}(\varphi(x) \wedge \neg R(x,e)) & = \int_{p \in S_{x}^{*}(M)} f_{\varphi}(p)d\lambda = \mu_{W}(\varphi(x)), 
\end{align*}
as required. 
\end{proof}

The following fact is a simplified version of \cite[Theorem 17.25]{kechris2012classical}. 

\begin{fact}\label{fact:kechris} Let $(X,\mathcal{S})$ be a measurable space, $Y$ a separable metrizable space, $\mu$ a Borel probability measure on $Y$, and $f\colon X \times Y \to [0,1]$ a bounded measurable map. Then the function
\begin{equation*}
    x \mapsto \int_{y \in Y} f(x,y) \, d\mu
\end{equation*}
is $\mathcal{S}$-measurable. 
\end{fact}

\begin{proposition}\label{prop:Borel} Suppose that $W$ is a Borel graphon. Then $\mu_{W}$ is Borel-definable over $M$. 
\end{proposition}

\begin{proof} By Proposition \ref{prop:inv}, $\mu_{W}$ is $M$-invariant. It suffices to consider formulas of the form $\theta(x,\bar{y}) := \bigwedge_{i=1}^{n} R(x,y_i) \wedge \bigwedge_{i=n+1}^{m} \neg R(x,y_j) \wedge \tau(\bar{y})$ where $1 \leq n \leq m$ and $\tau(\bar{y})$ is an $\mathcal{L}_{\bar{y}}$-formula, and show that the maps $F_{\mu}^{\theta}\colon S_{\bar{y}}(M) \to [0,1]$ are Borel. The argument can easily be extended to the case when $\theta(x,\bar{y})$ contains instances of equality and inequality. For each $i \leq m$, let $\pi_{i} \colon S_{\bar{y}}(M) \to S_{y_{i}}(M)$ be the natural projection map. For each $A \subseteq [m]$, we consider the Borel set 
\begin{equation*} E_{A} := \bigcap_{k \in A} \pi_{k}^{-1}(M) \cap \bigcap_{k \in [m]\backslash A} \pi_{k}^{-1}(S_{y_k}(M)\backslash M). 
\end{equation*}
Note that the set $\{E_{A}\st A \subseteq [m] \}$ forms a partition of $S_{\bar{y}}(M)$. For each $A \subseteq [m]$, we consider the function $f_{A}\colon S_{x}^{*}(M) \times E_{A} \to [0,1]$ given by
\begin{equation*}
    f_{A}(p,q) := \prod_{\substack{i \leq n\\ i \in A}}\mathbf{1}_{R(x,\pi_{i}(q))}(p) \prod_{\substack{n < i \leq m\\ i \in A}}\mathbf{1}_{\neg R(x,\pi_i(q))}(p)\prod_{\substack{i \leq n \\ i \not \in A}}W(p,\pi_{i}(q))\prod_{\substack{n < i \leq m\\ i \not \in A}}(1-W(p,\pi_{i}(q))).
\end{equation*}
Notice that the map $h: S_{x}^{*}(M) \times S_{\bar{y}}(M) \to [0,1]$ given by  
\begin{equation*}
    h(p,q) := \left( \sum_{A \subseteq [m]} \mathbf{1}_{E_{A}}(q) f_{A}(p,q) \right) \mathbf{1}_{\tau}(q)
\end{equation*}
is Borel. By Fact \ref{fact:kechris}, we conclude that the map 
\begin{equation*}
    q \mapsto \int_{p \in S_{x}^{*}(M)} h(p,q) d\lambda
\end{equation*}
is also Borel. Since $F_{\mu_{W}}^{\theta}(q) = \int_{p \in S_{x}^{*}(M)} h(p,q) d\lambda$, this concludes the proof. 
\end{proof}

\begin{proposition}\label{prop:adequate} Let $W$ be a Borel graphon. Then $\mu_{W}$ is $M$-adequate. Moreover, if $W_1,W_2,W_3$ are Borel graphons, then 
\begin{equation*}
    (\mu_{W_1} \otimes \mu_{W_2}) \otimes \mu_{W_3} = \mu_{W_1} \otimes( \mu_{W_2} \otimes \mu_{W_3}). 
\end{equation*}
\end{proposition}
\begin{proof} This follows directly from the fact that measures of the form $\mu_{W}$, where $W$ is a Borel graphon, are Borel-definable over countable models (Proposition \ref{prop:Borel}), and from the fact that measures that are Borel-definable over countable models are relatively tame. Specifically, by \cite[Theorem 2.13]{CGH} we see the following.
\begin{enumerate}
    \item The Morley product of Borel-definable measures (over countable models) is again Borel-definable. By induction, iterated products of measures of the form $\mu_{W}$ are also Borel-definable. 
    \item Triples of Keisler measures that are each Borel-definable over a countable model are associative with respect to the Morley product. Hence powers of $\mu_{W}$ are associative, and so $\mu_{W}$ is $M$-adequate. This fact also implies the \emph{moreover} portion of the statement.   \qedhere 
\end{enumerate} 
\end{proof}

\begin{definition} Let $\bar{c} = c_1,...,c_n$ be a tuple of parameters from $\mathcal{U} \backslash M$. The space $S^\ast_y(M\bar{c})$ can be canonically identified with the space $S^\ast_y(M) \times S^\ast_y(\bar{c})$ via the projection maps $\pi_M \colon S^\ast_y(M\bar{c}) \to S^\ast_y(M)$ and $\pi_{\bar{c}} \colon S^\ast_y(M\bar{c}) \to S^\ast_y(\bar{c})$. Define $\rho \colon  S^\ast_y(M) \times S^\ast_y(\bar{c}) \to S^\ast_y(M\bar{c})$ to be the map witnessing this identification. 
\end{definition}

\begin{lemma}\label{lemma:key} Let $\bar{c} = c_1,...,c_n$ be a tuple of parameters from $\mathcal{U}\backslash M$. Then for any Borel-measurable function $f \colon S^\ast_y(M\bar{c}) \to [0,1]$, we have 
\begin{equation*}
\int_{q \in S^\ast_y(M\bar{c})}f(q)  d\mu_W = \int_{r \in S^\ast_y(M)} \sum_{s \in S^\ast_y(\bar{c})}f(\rho(r,s)) \prod_{i \leq n} W^{s}(r, \tp(c_i/M))d\lambda, 
\end{equation*}
where $W^{s}(r,\tp(c_{i}/M)):=\begin{cases}
\begin{array}{cc}
W(r,\tp(c_{i}/M)) & \text{if \ }R(y,c_{i})\in s,\\
1-W(r,\tp(c_{i}/M)) & \text{if \ }\neg R(y,c_{i})\in s.
\end{array}
\end{cases}$
\end{lemma}

\begin{proof}
By an abuse of notation we will identify types in $S^\ast_y(\bar{c})$ with their isolating $R$-formulas, i.e., the corresponding conjunctions of instances of $R(y,c_i)$ and $\neg R(y,c_i)$). 

By linearity and the definition of $\mu_W(y)$, we have that the equality in the statement of the lemma holds for any function of the form $\mathbf{1}_{\alpha(y) \wedge s(y)}$, where $\alpha(y)$ is an $\mathcal{L}_{y}(M)$-formula and $s \in S^\ast_y(\bar{c})$. Any $\mathcal{L}_{y}(M\bar{c})$-formula $\alpha(y)$ can be written (up to $\mu_W(y)$-measure $0$) as $\bigvee_{s \in S^\ast_y(\bar{c})} \alpha_s(y) \wedge s(y)$, with each $\alpha_s(y)$ an $\mathcal{L}_{y}(M)$-formula. By linearity again, the equality holds for the indicator function of any such formula. Therefore, one can conclude that the equality holds for any function of the form $\sum_{i=1}^{n} r_i \mathbf{1}_{\alpha_i(y)}$, i.e., any linear combination of (indicator functions of the sets corresponding to) $\mathcal{L}_{y}(M\bar{c})$-formulas.

Now suppose that $f$ is a continuous function. Then $f$ is a uniform limit of functions of the form $\sum_{i=1}^{n} r_i \mathbf{1}_{\alpha_i(y)}$ where each $\alpha_i(y)$ is a $\mathcal{L}_{y}(M\bar{c})$-formula built only from instances of $R(x,y)$ and its negation. This is because these formulas form the basic clopens in $S_{y}^{*}(M\bar{c})$. In particular, $f$ is a pointwise limit, and so the equality holds by the dominated convergence theorem. 

Finally, since $S_{y}^{*}(M\bar{c})$ is a Polish space, every Borel function is a Baire class $\alpha$ for some $\alpha < \omega_1$. Hence the equality holds for arbitrary Borel measurable functions by a simple induction argument up the Baire class hierarchy via the dominated convergence theorem.
\end{proof}

Recall the definitions of the Morley product and of excellent measures (Definition \ref{cheat}(5) and Definition \ref{adex}, respectively). Given a Borel graphon $W$, we will prove that the measure $\mu_{W}$ is $M$-excellent by computing the Morley products $\mu_{W}(x) \otimes \mu_{W}(y)$ and $\mu_{W}(y) \otimes \mu_{W}(x)$ applied to an arbitrary formula. We will expand the Morley product in one direction and observe that by Fubini's Theorem and symmetry, the two products are equal. 

\begin{lemma}\label{lemma:excellent} Suppose that $W$ is a Borel graphon. Then $\mu_{W}$ is $M$-excellent.
\end{lemma}

\begin{proof} By Proposition \ref{prop:adequate}, it suffices to prove that $\mu_{W}$ self-commutes. By quantifier elimination
(as the underlying structure is the Rado graph),
it suffices to show that $\mu_{W}(x) \otimes \mu_{W}(y)$ and $\mu_{W}(y) \otimes \mu_{W}(x)$ agree on arbitrary intersections of literals (since these form a $\pi$-system). One can easily check that the measure of any literal involving only instances of equality and inequality has the same value (in $\{0,1\}$) with respect to both Morley products, and so these formulas may be disregarded. Fix parameters $c_1,..,c_{n} \in \mathcal{U}\backslash M$ and consider any formula of the form
\begin{equation*}
    \varphi(x,y) := \tau(x) \wedge \gamma(y) \wedge \psi(x,y) \wedge \bigwedge_{i \leq  n} \eta_{i}(x,c_i) \wedge \bigwedge_{i \leq  n} \zeta_i(y,c_i),
\end{equation*}
where $\tau(x)$ and $\gamma(y)$ are $\mathcal{L}(M)$-formulas, and where $\psi$, the $\eta_i$ and the $\zeta_i$ are all either the formula $R$ or the formula $\neg R$. We need to describe the fiber functions $F_{\mu_{W}(x)}^{\varphi}\colon S_{y}(M\bar{c}) \to [0,1]$ and $F_{\mu_{W}(y)}^{\varphi^{*}}\colon S_{x}(M\bar{c}) \to [0,1]$.
Observe that Definition \ref{cheat}(5) of the Morley product essentially describes one side of Fubini's Theorem where we sum up the values of $\mu(\varphi(x,b))$ across all possible choices of $b$. However, the Morley product in general is not commutative, and so one needs to apply a little finesse.

Let $\mathbf{1}_{\eta}$ be the indicator function for $\bigwedge_{i \leq n} \eta_i(x,c_i)$ and $\mathbf{1}_{\zeta}$ be the indicator function for $\bigwedge_{i \leq n} \zeta_i(y,c_i)$. For $z \in \{x,y\}$, we have a natural projection map $\pi_{M,z}\colon S_{z}(M\bar{c}) \to S_{z}(M)$. We will write $\pi_M$ for $\pi_{M,z}$ when there is no cause for confusion. For $\xi \in \{\eta_i,\zeta_i,\psi\}$, we let 
\begin{equation*}
    W^{\xi}(p,q):=\begin{cases}
\begin{array}{cc}
W(p,q) & \text{if }\xi=R,\\
1-W(p,q) & \text{if }\xi=\neg R.
\end{array}\end{cases}. 
\end{equation*}
For any $q \in S^{*}_{y}(M\bar{c})$, we have 
\begin{align*}
    F_{\mu_{W}(x)}^{\varphi}(q) = \mathbf{1}_{\gamma}(q)  \mathbf{1}_{\zeta}(q) \int_{p \in S_{x}^{*}(M)} \mathbf{1}_{\tau}(p) W^{\psi}(p,\pi_{M}(q)) \prod_{i \leq n} W^{\eta_i}(p,\tp(c_i/M)) d\lambda. 
\end{align*}
Likewise for any $p \in S^\ast_{x}(M\bar{c})$, we have 
\begin{align*}
F^{\varphi^{*}}_{\mu_W(y)}(p) = \mathbf{1}_\tau(p) \mathbf{1}_\eta(p)\int_{q \in S^\ast_y(M)}\mathbf{1}_\gamma(q) & W^{\psi}(q, \pi_{M}(p)) \prod_{i \leq n} W^{\zeta_i}(q,\tp(c_i/M)) d\lambda. 
\end{align*}
Computing the two Morley products, we have that the value of $(\mu_W(x)\otimes\mu_W(y))(\varphi(x,y)) $ is
\begin{align*}
  \int_{q \in S_y(M\bar{c})}\int_{p \in S^\ast_x(M)}
\mathbf{1}_\gamma(q)\mathbf{1}_\zeta(q)\mathbf{1}_\tau(p) W^{\psi}(p, \pi_{M}(q)) \prod_{i \leq n} W^{\eta_i}(p,\tp(c_i/M)) d\lambda d\mu_{W},
\end{align*}
while the value of $(\mu_W(y) \otimes \mu_W(x))(\varphi(x,y))$ is
\begin{align*}
\int_{p \in S_x(M\bar{c})}\int_{q \in S^\ast_y(M)}\mathbf{1}_\tau(p)\mathbf{1}_\eta(p)\mathbf{1}_\gamma(q) W^{\psi}(q, \pi_{M}(p)) \prod_{i \leq n} W^{\zeta_i}(q, \tp(c_i/M)) d\lambda  d\mu_W.
\end{align*}
Since $\mu_W$ does not concentrate on points (\cref{prop:inv}), we can restrict the outer integrals to $S^\ast_y(M\bar{c})$ and $S^\ast_x(M\bar{c})$, respectively. Consider the maps $f_1\colon S_{x}^{*}(M\bar{c}) \to [0,1]$ and $f_2\colon S_{y}^{*}(M\bar{c}) \to [0,1]$ given, respectively, by 
\begin{align*}
f_1(q) := \int_{p \in S^\ast_x(M)}
\mathbf{1}_\gamma(q)\mathbf{1}_\zeta(q)\mathbf{1}_\tau(p) W^{\psi}(p,\pi_{M}(q)) \prod_{i \leq n} W^{\eta_i}(p,\tp(c_i/M)) d\lambda
\end{align*}
and
\begin{align*}
f_2(p) := \int_{q \in S^\ast_y(M)}\mathbf{1}_\tau(p)\mathbf{1}_\eta(p)\mathbf{1}_\gamma(q) W^{\psi}(q, \pi_{M}(p)) \prod_{i \leq n}W^{\zeta_i}(q, \tp(c_i/M)) d\lambda. 
\end{align*}
Now apply Lemma \ref{lemma:key} to the integrals defining the maps $f_1$ and $f_2$. We compute:
\begin{align*}
    (\mu_{W}(x) \otimes \mu_{W}(y))(\varphi(x,y))
    &= \int_{q \in S_{y}(M\bar{c})} F_{\mu_{W}(x)}^{\varphi}(q) d\mu_{W}(y)\\ 
    &= \int_{q \in S^{*}_{y}(M\bar{c})} f_{1}(q) d\mu_{W}(y) \\
    &\overset{(*)}{=} \int_{r \in S_{y}^{*}(M)} \sum_{s \in S_{y}^{*}(\bar{c})} f_{1}(\rho(r,s)) \prod_{i \leq n } W^{s}(r, \tp(c_i/M)) d\lambda \\
    &\overset{}{=} \int_{r \in S_{y}^{*}(M)} \sum_{s \in S_{y}^{*}(\bar{c})} \int_{p \in S_{x}^{*}(M)} \mathbf{1}_{\gamma}(\rho(r,s)) \mathbf{1}_{\zeta}(\rho(r,s)) \mathbf{1}_{\tau}(p) W^{\psi}(p, \pi_{M}(\rho(r,s))) \\
    &\hspace*{20pt} \prod_{i \leq n} W^{\eta_i}(p, \tp(c_i/M)) d\lambda
    \prod_{i \leq n} W^{s}(r,\tp(c_i/M)) d\lambda(r),
\end{align*}
where equation $(*)$ is an application of Lemma \ref{lemma:key}. Continuing the computation, we have:
\begin{align*}
&= \int_{r \in S_{y}^{*}(M)} \sum_{s \in S_{y}^{*}(\bar{c})} \int_{p \in S_{x}^{*}(M)} \mathbf{1}_{\gamma}(r) 
\mathbf{1}_{\zeta}(s) \mathbf{1}_{\tau}(p) W^{\psi} (p,r) 
\prod_{i \leq n} W^{\eta_i}(p,\tp(c_i/M)) d\lambda  \prod_{i \leq n} W^{s}(r, \tp(c_i/M)) d\lambda \\
&\hspace*{-4pt}\overset{(**)}{=} \int_{r \in S_{y}^{*}(M)} \int_{p \in S_{x}^{*}(M)} \mathbf{1}_{\gamma}(r) \mathbf{1}_{\tau}(p) W^{\psi}(p,r) \prod_{i \leq n} W^{\eta_i}(p, \tp(c_i/M)) d\lambda \prod_{i \leq n} W^{\zeta_i}(r,\tp(c_i/M)) d\lambda \\
&= \int_{r \in S_{y}^{*}(M)} \int_{p \in S_{x}^{*}(M)} \mathbf{1}_{\gamma}(r) \mathbf{1}_{\tau}(p) W^{\psi}(p,r) \prod_{i \leq n} W^{\eta_i}(p, \tp(c_i/M))  \prod_{i \leq n} W^{\zeta_i}(r,\tp(c_i/M)) d\lambda d\lambda,
\end{align*}
where equation $(**)$ follows from the fact that $\mathbf{1}_{\zeta}(s)$ determines a unique complete type in $S_{y}^{*}(\bar{c})$. A symmetric computation and an application of Fubini's theorem completes the proof. 
\end{proof}

We prove the main theorem (\cref{theorem:2main}) by induction on the length of graph formulas. In the induction step, one needs to know how to compute the integral of the product of the characteristic function of the restriction of a graph formula (to its first $n$ variables) against a function which factors (over the product) with respect to our Morley power. It turns out that the integral of the product of these two functions with respect to the Morley power is precisely the integral of the product of the factored function and sampled graphon with respect to the product of the measure $\lambda$. So, intuitively, the computation removes the encoded graphon from $\mu_{W}$ and pulls it into the integrand.

\begin{lemma}\label{lemma:integral} For $n \geq 2$, suppose that $\hat{h}\colon \prod_{i=1}^{n} S_{x_i}^{*}(M) \to [0,1]$ is a Borel map. Let $\pi\colon S^{*}_{x_1,...,x_n}(M) \to \prod_{i=1}^n S_{x_i}^*(M)$ be the obvious surjective map and define $h = \hat{h} \circ \pi$. Let $G$ be a graph on vertex set $\{1,...n\}$ and $\varphi_{G}(x_1,...,x_n)$ be the corresponding graph formula. Then
\begin{align*}
    \int_{q \in S^{*}_{x_1,...,x_{n}}(M)} 
    \mathbf{1}_{\varphi_{G}(x_1,...,x_{n})}(q) \cdot h(q) d\mu_{W}^{(n)} 
    =
    \int_{\bar{t} \in \prod_{i \leq n} S^{*}_{x_i}(M)} \hat{h}(\bar{t}) \prod_{1 \leq i < j \leq n } W^{G(i,j)}(t_i,t_j)d\lambda^{n}.
\end{align*}
\end{lemma}

\begin{proof} To simplify notation, we let $\mu_{W} = \mu$. We prove the statement by induction on $n$. So let $n = 2$. We first prove the statement for functions of the form $\hat{h} = \mathbf{1}_{\psi_1(x_1)} \times \mathbf{1}_{\psi_2(x_2)}$ where $\psi_1(x_1) := \bigwedge_{i=1}^{m_1} R^{\epsilon(1,i)}(x_1,a_i)$ and $\psi_2(x_2) := \bigwedge_{j =1}^{m_2} R^{\epsilon(2,j)}(x_2,b_j)$ where $\epsilon(i,j) \in \{1,-1\}$ for $i \in \{1,2\}$ and $1 \leq j \leq m_i$. Notice that if we let $\psi(x_1,x_2) = \psi_1(x_1) \wedge \psi_2(x_2)$, then $h = \mathbf{1}_{\psi_1(x_1) \wedge \psi_2(x_2)}$. Without loss of generality, we let $\varphi_{G}(x_1,x_2) = R(x_1,x_2)$. Now consider the following computation:
\begin{align*}
    \int_{q \in S_{x_1,x_2}(M)} \mathbf{1}_{\varphi_{G}(x_1,x_2)} \cdot h(q) d\mu^{(2)} &= (\mu_{x_2} \otimes \mu_{x_1})(\varphi_{G}(x_1,x_2) \wedge \psi(x_1,x_2)) \\ 
    &= \int_{q \in S_{x_1}(M)} F_{\mu_{x_2}}^{\varphi_{G}(x_1,x_2) \wedge \psi(x_1,x_2)}(q) d\mu_{x_1} \\ 
    &= \int_{q \in S_{x_1}^{*}(M)} \mathbf{1}_{\psi_1}(q) \cdot F_{\mu_{x_2}}^{\varphi_{G}(x_1,x_2) \wedge \psi_2(x_2)}(q) d\mu_{x_1} \\ 
    &\overset{(a)}{=} \int_{q \in S_{x_1}^{*}(M)} \int_{p \in S_{x_2}^{*}(M)} \mathbf{1}_{\psi_1}(q) \mathbf{1}_{\psi_2}(p) W(p,q)d\lambda d\mu  \\
    &\overset{(b)}{=} \int_{q \in S_{x_1}^{*}(M)} \int_{p \in S_{x_2}^{*}(M)} \mathbf{1}_{\psi_1}(q) \mathbf{1}_{\psi_2}(p) W(p,q)d\lambda d\lambda\\ 
    &\overset{}{=} \int_{(t_1,t_2) \in S_{x_1}^{*}(M) \times S_{x_2}^{*}(M)} \hat{h}(t_1,t_2) \cdot W(t_1,t_2) \lambda^{2},
\end{align*}
where the final equation is what we require. We provide a few justifications. 
\begin{enumerate}[label=($\alph*$)]
    \item We replace the function $F_{\mu_{x_2}}^{\varphi_{G}(x_1,x_2) \wedge \psi_2(x_2)}$. Let $q \in S^*_{x_1}(M)$ and $b \models q$. Then $c \in \mathcal{U} \backslash M$ and since $\varphi_{G}(x_1,x_2) = R(x_1,x_2)$, we have  
    \begin{align*}
        F_{\mu_{x_2}}^{\varphi_{G}(x_1,x_2) \wedge \psi_2(x_2)}(q) &= \mu_{x_2}(R(x_2,c) \wedge \psi_{2}(x_2)) \\ 
        &= \int_{p \in S_{x_2}^{*}(M)} \mathbf{1}_{\psi_2}(p) W(p,\tp(c/M)) d\lambda \\ 
        &= \int_{p \in S_{x_2}^{*}(M)} \mathbf{1}_{\psi_2}(p) W(p,q) d\lambda. 
    \end{align*}
    \item In the equation, $\mu$ is actually the regular Borel probability measure corresponding to $\mu|_M$ over $S_{x_1}^{*}(M)$. We notice that by construction, $\lambda|_{M} = \mu|_{M}$. Thus the regular Borel measures on $S_{x_1}^{*}(M)$ corresponding to $\mu|_{M}$ and $\lambda|_{M}$ agree. Therefore, $\mu$ can be replaced with $\lambda$.
\end{enumerate}

This completes the proof for $\hat{h} = \mathbf{1}_{\psi_1(x_1)} \times \mathbf{1}_{\psi_2(x_2)}$. We now explain how to extend the proof to arbitrary Borel $\hat{h}$.
\begin{enumerate}
    \item Any $\mathcal{L}_{R}(M)$-formula can be written as a disjoint union of formulas of the form $\psi$ from the above argument. By linearity of integration, the statement holds for $\hat{h} = \mathbf{1}_{\theta_1(x_1)} \times \mathbf{1}_{\theta_2(x_2)}$ where $\theta_1,\theta_2$ are $\mathcal{L}_{R}(M)$-formulas.
    \item Again, by linearity of integration, the statement can be extended to functions of the form $\hat{h} = \sum_{i=1}^{n} r_i \mathbf{1}_{\theta_1^{i}(x_1)} \times \mathbf{1}_{\theta_2^{i}(x_2)}$ where $\theta_1^{i}(x_1),\theta_2^{i}(x_2)$ are $\mathcal{L}_{R}(M)$-formulas. 
    \item Now suppose that $\hat{h}\colon  \prod_{i=1}^{2} S_{x_i}^{*}(M) \to [0,1]$ is continuous. Note that $\prod_{i=1}^{2} S_{x_i}^{*}(M)$ is a Stone space. Hence any continuous function can be uniformly approximated by a finite convex combination of characteristic functions concentrating on clopens. In particular, one can find a sequence of functions $(g_i)_{i < \omega}$ which converges pointwise to $\hat{h}$ such that each $g_i$ is a finite convex combination of clopens. In the space $\prod_{i=1}^{2} S_{x_i}^{*}(M)$, clopens correspond precisely to products of $\mathcal{L}_{R}(M)$-formulas. Hence, the $g_i$ are precisely the functions from (2). Thus the statement holds via an application of the dominated convergence theorem. 
    \item Now suppose that $\hat{h}\colon \prod_{i=1}^{2} S_{x_i}^{*}(M) \to [0,1]$ is Borel. Since $\prod_{i=1}^{2} S_{x_i}^{*}(M)$ is Polish, the function $\hat{h}$ is of Baire class $\alpha$ for some $\alpha < \omega_1$. Hence, one can induct up the Baire hierarchy. Recall that Baire class $0$ functions are the continuous functions, which were handled in the previous case. Now suppose that the assertion holds for Baire class $\beta$ when $\beta < \alpha$, and that $\hat{h}$ is of Baire class $\alpha$. By definition, this means that $\hat{h}$ can be written as the pointwise limit of functions strictly lower in the Baire hierarchy. Then a straightforward application of the dominated convergence theorem completes the proof. 
\end{enumerate}

We now assume that the statement holds for $n -1$. By a similar argument to the base case, it suffices to consider $\hat{h}$ of the form $\prod_{i \leq n} \mathbf{1}_{\psi_i(x_i)}$ where $\psi_{i}(x_i) : = \bigwedge_{j=1}^{m_i} R^{\epsilon(i,j)}(x_i,b_j)$ and $\epsilon(i,j) \in \{1,-1\}$, $1 \leq i \leq n$ and $1 \leq j \leq m_i$. Let $\psi(x_1,...,x_n) = \wedge_{i=1}^{n} \psi_i(x_j)$ and so $h = \mathbf{1}_{\psi(x_1,...,x_n)}$. Let $\psi_{n-1}(x_1,...,x_{n-1})$ and $\psi_{n}(x_n)$ be the natural restrictions of $\psi(x_1,...,x_n)$ to the indicated tuples of variables. Let $\varphi_{G_{n-1}}(x_1,...,x_{n-1})$ be the graph formula corresponding to the restriction of graph $G$ to vertex set $\{1,...,n-1\}$. Notice that if we define $\widehat{h_0}\colon \prod_{i \leq n-1} S_{x_i}^{*}(M) \to [0,1]$ to be given by 
\begin{equation*}
    \widehat{h_0}(\bar{q}) := \prod_{i =1}^{n-1} \mathbf{1}_{\psi_{i}(x_i)}(q_i) \int_{t_n \in S^*_{x_n}(M)} \mathbf{1}_{\psi_{n}(x_n)}(t_n) \prod_{G \models R(n, i)} W(t_n,q_i) \prod_{G \models \neg R(n,i)} (1-W(t_n,q_i)) d\lambda
\end{equation*}
then it is straightforward to check that for each $q \in S_{x_1,...,x_n}^{*}(M)$, we have 
\begin{equation*}
F_{\mu_{x_n}}^{\varphi_{G}(x_{n};\bar{x}) \wedge \psi(x_n;\bar{x})}(q) = \mathbf{1}_{G_{n-1}(\bar{x})}(q) \cdot h_0(q)
\end{equation*}
where $h_0 = \widehat{h_0} \circ \pi$. The map $h_0$ is Borel since $W$ is Borel and by Fact \ref{fact:kechris}. Hence we have the following computation:
\begin{align*}
    \int_{q \in S_{\bar{x}}(M)} \mathbf{1}_{\varphi_{G}(x_1,...,x_n)}(q) h(q) d\mu^{(n)} 
     &=  \mu^{(n)} (\varphi_{G}(x_1,...,x_n) \wedge \psi(x_1,...,x_n)) 
    \\&= \int_{q \in S_{x_1,...,x_{n-1}}(M)} F_{\mu_{x_n}}^{\varphi_{G}(x_n;\bar{x}) \wedge \psi(x_n;\bar{x})}(q) d\mu^{(n-1)}
    \\ &= \int_{q \in S_{x_1,...x_{n-1}}^{*}(M)} \mathbf{1}_{G_{n-1}(\bar{x})}(q) h_0(q) d\mu^{(n-1)}
    \\&\overset{(*)}{=} \int_{\bar{t} \in \prod_{i \leq n-1} S^*_{x_{i}(M)}} \widehat{h_0}(\bar{t}) \prod_{1 \leq i < j \leq n -1} W^{G_{n-1}(i,j)}(t_i,t_j) d\lambda^{n-1} \\ 
  &\overset{(**)}{=} \int_{\bar{t} \in \prod_{i \leq n} S^*_{x_i}(M)} \hat{h}(\bar{t}) \prod_{1 \leq i < j \leq n } W^{G(i,j)}(t_i,t_j) d\lambda^{n}.
\end{align*}
Equation $(*)$ is an application of the induction hypothesis (reading from equation $(*)$ upwards). Equation $(**)$ follows from unpacking the function $\hat{h}_0$ and reorganizing terms. An almost identical argument to the one provided in the base case allows one to prove the result for arbitrary Borel $\hat{h}$.
\end{proof}

We are now ready to prove our main graphon representation theorem, Theorem \ref{theorem:2main}.
To prove Theorem~\ref{theorem:2main}, we need to compute a specific integral. More explicitly, we need to integrate the characteristic function of a graph formula times a Borel function which factors through the surjective map $\pi$ from \cref{lemma:integral}, i.e., $h = \hat{h} \circ \pi$. We have little control over the map $\hat{h}$. For this reason, we earlier proved \cref{lemma:integral} for all Borel $\hat{h}$. It is clear when trying to do the computation that we need an inductive argument, and so we prove the statement by induction on the number of variables and by increasing the complexity of the function at each stage. We first prove the identity for characteristic functions of clopen subsets of the product (in two variables). We then extend to linear combinations of these functions. Since the domain is a Stone space, linear combinations of characteristic functions of clopens are dense (with respect to the supremum norm) in the space of all continuous functions, and so the identity holds for all continuous $\hat{h}$. We then extend to all Borel $\hat{h}$ in two variables via the dominated convergence theorem (formally, we do so by iterating up the Borel hierarchy;  this works because the underlying space is Polish). We then induct on the number of variables and ``repeat'' a similar argument for each number of variables.

\begin{theorem}\label{theorem:2main} Let $W$ be a Borel graphon and $g\colon S^{+}_{\mathbf{x}}(\mathcal{U}) \to \str_{\mathcal{L}}$. Then $g_{*}(\kmpow{\mu_{W}}) = \mathbb{G}(\mathbb{N}, W)$. 
\end{theorem}

\begin{proof}To simplify the notation, we denote $\mu_{W}$ simply as $\mu$. By Lemma \ref{lemma:excellent}, we have that $\mu$ is $M$-excellent and so by Fact \ref{fact:excellent} and \cref{prop:inv}, $g_{*}(\kmpow{\mu})$ is $\Sym(\mathbb{N})$-invariant. Hence it suffices to show that $g_{*}(\kmpow{\mu})$ and $\mathbb{G}(\mathbb{N}, W)$ agree on sets of the form $\llbracket \varphi_{G}(1,...,n+1) \rrbracket$ where $\varphi_{G}(x_1,...,x_{n+1})$ is a graph formula and $n \geq 1$. Notice that if $n = 1$, then the base case of Lemma \ref{lemma:integral}
gives a proof. Hence we only need to consider $n > 1$. 

Let $G_{n}(\bar{x})$ be the restriction of the graph formula $\varphi_{G}(\bar{x})$ to the variables $x_1,...,x_n$,
which is also a graph formula. 
Define $\hat{h}\colon \prod_{i=1}^{n} S_{x_i}^{*}(M) \to [0,1]$ by
\begin{equation*}
\hat{h}(\bar{t}) =  \int_{t_{n+1} \in S^{*}_{x_{n+1}}(M)} \prod_{G \models R(n+1,i)} W(t_i,t_{n+1}) \prod_{G \models \neg R(n+1,i)} (1-W(t_i,t_{n+1}))d\lambda,
\end{equation*}
and let $h = \hat{h} \circ \pi$. 
By a straightforward argument, for any $q \in S_{x_1,...,x_n}^{*}(M)$, we have  $F_{\mu_{x_{n+1}}}^{\varphi_{G}(x_{n+1};\bar{x})}(q) = \mathbf{1}_{G_{n}(\bar{x})}(q) \cdot h(q)$.
The map $\hat{h}$ is Borel by Fact \ref{fact:kechris} and so we may apply Lemma \ref{lemma:integral}:
\begin{align*}
    g_{*}(\kmpow{\mu})(\llbracket \varphi_{G}(1,...,n+1)\rrbracket) &= \kmpow{\mu}(g^{-1}(\llbracket \varphi_{G}(1,...,n+1)\rrbracket)) \\ 
    &= \kmpow{\mu}(\varphi_{G}(x_1,...,x_{n+1})) \\
    &= \left( \bigotimes_{i=1}^{n +1} \mu_{x_i} \right) (\varphi_{G}(x_1,...,x_{n+1})) \\
    &= \left(\mu_{x_{n+1}}  \otimes  \bigotimes_{i = 1}^{n} \mu_{x_i} \right)(\varphi_{G}(x_1,...,x_{n+1})) \\
    &= \int_{q \in S_{x_1,...,x_{n}}(M)} F_{\mu_{x_{n+1}}}^{\varphi_{G}^{*}(x_{n+1};x_{1},...,x_{n})}(q) d\mu^{(n)} \\
    &= \int_{q \in S_{x_1,...,x_{n}}^{*}(M)} \mathbf{1}_{\varphi_{G_{n}}(\bar{x})}(q) \cdot h(q) d\mu^{(n)} \\
    &\overset{(*)}{=} \int_{\bar{t} \in \prod_{i \leq n} S^{*}_{x_i}(M)} \hat{h}(\bar{t}) \prod_{1 \leq i < j \leq n} W^{G_{n}}(t_i,t_j) d\lambda^{n} \\ 
  &\overset{(**)}{=} \int_{\bar{t} \in \prod_{i \leq n+1} S_{x_i}^{*}(M)} \prod_{1 \leq i < j \leq n+1} W^{G}(t_i,t_j) d\lambda^{n+1}. \\
  &= \mathbb{G}(\mathbb{N}, W)(\llbracket \varphi_{G}(1,...,n+1)\rrbracket). 
\end{align*}
Equation $(*)$ follows from Lemma \ref{lemma:integral} while equation $(**)$ follows from unpacking $\hat{h}$ and reorganizing terms. 
\end{proof}

\begin{corollary}\label{cor:AFP1} Let $G$ be a countable graph. Then the following are equivalent. 
\begin{enumerate}
    \item There exists an $M$-excellent Keisler measure $\mu$ that does not concentrate on points and such that $\kmpow{\mu}(\mathbb{B}_{G}) = 1$.
    \item There exists a $\Sym(\mathbb{N})$-invariant measure $\eta$ on $\str_{\mathcal{L}}$ such that $\eta(\mathbb{A}_{G}) = 1$. 
    \item The graph $G$ has trivial group-theoretic definable closure, i.e., for any finite tuple $\bar{a} = a_1,...,a_n$ from $G$, we have that $\dcl(\bar{a}) = \bar{a}$, where $\dcl(\bar{a})$ is the collection of elements from $G$ that are fixed by all the automorphisms of $G$ which fix $\bar{a}$ pointwise.
\end{enumerate}   
\end{corollary}

\begin{proof} Statements $(2)$ and $(3)$ are equivalent via \cite[Theorem 1.1]{ackerman2016invariant}. By 
Fact \ref{fact:excellent}, statement $(1)$ implies statement $(2)$.

Assume statement $(2)$. By an observation in  \cite[subsection 6.1.2] {ackerman2016invariant}, there exists a \emph{random-free graphon} $W$ (which is Borel) such that $\mathbb{G}(\mathbb{N}, W)$ concentrates on $\mathbb{A}_{G}$. By Lemma \ref{lemma:excellent}, the measure $\mu_{W}$ is $M$-excellent, and it clearly does not concentrate on points. By Theorem \ref{theorem:2main}, $g_{*}(\kmpow{\mu_{W}}) = \mathbb{G}(\mathbb{N}, W)$ and so $\mathbb{P}(\mu_{W})(\mathbb{B}_{G}) = 1$.  Hence statement $(1)$ holds.
\end{proof}

\begin{corollary}\label{cor:ergodic} Suppose $\nu$ is an ergodic $\Sym(\mathbb{N})$-invariant measure on $\str_{\mathcal{L}}$  concentrated on the collection of graphs. Then there exists an $M$-excellent Keisler measure $\mu \in \mathfrak{M}^{\inv}_{x}(\mathcal{U},M)$ that does not concentrate on points and such that $g_{*}(\mathbb{P}_{\mu}) = \nu$. 
\end{corollary} 

\begin{proof} Follows directly from \cref{lemma:rep} and \cref{theorem:2main}. 
\end{proof}

\section{Hypergraphons, generic sampling, and random hypergraphs}
\label{sec:hypergraphon}

Here we prove the hypergraphon variant of the representation theorem. We begin by fixing some conventions and recalling some basic machinery associated to hypergraphons. By quantifier elimination and the fact that the measures we construct do not concentrate on points, formulas involving equality do not substantially affect the proofs, and so they are disregarded throughout.

Throughout this section we fix a positive integer $k > 2$ and let $T_{k}$ be the theory of the random $k$-uniform hypergraph in the language $\mathcal{L}_{k} = \{R(x_1,...,x_k)\}$. In other words, $T_{k}$ is the theory of the \Fraisse\ limit of all finite $k$-uniform hypergraphs (and thus admits quantifier elimination).  
We fix $\mathcal{U}$, a monster model of $T_{k}$, and $M$, a countable elementary submodel. We let $[m] = \{1,...,m\}$. For any set $X$ and any integer $t \leq |X|$, 
we let ${X \choose t}$
denote the collection of subsets of $X$ of size $t$
and ${X \choose t\leq k}$ denote the collection of non-empty subsets of $X$ of size less than or equal to $k$. 

For any subset $A$ of $\mathcal{U}$ and tuple of variables $\bar{x}$, we define the following spaces: 
\begin{enumerate}
    \item $S_{\bar{x}}^{*}(A)$ is the collection of non-realized types over $A$. 
    \item $S_{\bar{x}}^{+}(A) := \{ p \in S_{\bar{x}}^{*}(A)\st x_i \neq x_j \in p \text{~for all~} 
    i < j \le |\bar{x}|\}$. 
    \item Given a tuple of variables $\bar{x}$ such that $|\bar{x}| \leq k$, we let $S_{\bar{x}}^{\flat}(A)$ denote the collection of $R(\bar{x};\bar{y})$-types over $A$, where $|\bar{x}| + |\bar{y}| = k$.
    \end{enumerate}
     Notice that if $|\bar{x}| = k$, then $S_{\bar{x}}^{\flat} = \{R(\bar{x}), \neg R(\bar{x})\}$. As a convention, we write $S_{\bar{x}}^{*}(M), S_{\bar{x}}^{+}(M)$ and $S_{\bar{x}}^{\flat}(M)$ simply as $S_{\bar{x}}^{*}, S_{\bar{x}}^{+}$  and $S_{\bar{x}}^{\flat}$ respectively, whenever it is convenient to do so. We usually apply this convention to shorten the length of complicated terms. 
    \begin{enumerate}[resume]
    \item Let $I \subseteq \mathbb{N}_{> 0}$ and suppose $I = \{{i_1},...,{i_m}\}$, $|m| \leq k$. Notice that by symmetry of the hyperedge relation, the elements of $S_{x_{i_1},...,x_{i_m}}^{\flat}$ are fixed under the natural action of $\Sym(m)$ and so this space does not depend on the ordering of the variables. So we let $S_{I}^{\flat} := S_{\{x_{i_1},...,x_{i_m}\}}^{\flat}$. 
    \end{enumerate}
    We will also use the following notation:
    \begin{enumerate}[resume]
    \item Suppose  $B_0 \subseteq M$ and $C_0 \subseteq \mathcal{U} \backslash M$ are such that $|B_0| + |C_0| = k-1$. Then $R^{1}(x,B_0,C_0) = R(x,B_0,C_0)$ and $R^{-1}(x,B_0,C_0) = \neg R(x,B_0,C_0)$. 
    \item Given sets of parameters $B_0 \subseteq M$, $C_0 \subseteq \mathcal{U} \backslash M$ such that $|B_0| + |C_0| = k - 1$ and a formula $\varphi(x)$, we write $\epsilon^{\varphi}(B_0,C_0) = 1$ if $\varphi \vdash R(x,B_0,C_0)$ and $\epsilon^{\varphi}(B_0,C_0) = -1$ if $\varphi \vdash \neg R(x,B_0,C_0)$. In this section, we will often have expressions which include terms of the form `$R^{\epsilon^{\varphi}(B_0,C_0)}$'. 
\end{enumerate}

Our first fact demonstrates that the space of types in this setting can be thought as a product of local type spaces. 

\begin{fact} For any $n \geq 1$, consider the map
\begin{equation*}
\tau_{n}\colon  S_{x_1,...,x_n}^{+}(M) \to  \prod_{I \in {[n]\choose t \leq k }} S_{I}^{\flat}(M) 
\end{equation*}
defined by
\begin{equation*}
\tau_{n}(p) = (\tp_{I}^{\flat}(p))_{I \in {[n] \choose t \leq k }},
\end{equation*}
where for any $I$ as above,  
\begin{equation*}
    \tp_{I}^{\flat}(p) = \left\{R(\{x_l\}_{l \in I},\bar{b}) \st \bar{b} \in M^{k - |I|}\right\} \cap p. 
\end{equation*}
The map $\tau_{n}$ is a homeomorphism.
\end{fact}

\begin{proof} Follows directly from quantifier elimination and the symmetry of the hyperedge relation. 
\end{proof}

\begin{convention} Suppose $D$ is a finite set of parameters from $\mathcal{U}$ such that $|D| \leq k$. Then we let $\tp^{\flat}(D) = \tp_{|D|}^{\flat}(\tp(D/M))$. In other words, $\tp^{\flat}(D)$ is precisely $\{R(\bar{x},C) \st C \in M^{k-|D|} \text{~and~} M \models R(D,C)\}$. Moreover, if $d \in \mathcal{U}$, we let $\tp^{\flat}(d) = \tp^{\flat}(\{d\})$. 
\end{convention}

We now recall a useful family of Keisler measures. These measures will play the role of $\lambda$ in the previous section and mimic the Lebesgue measure on $[0,1]$. They are primarily used for scaffolding.

\begin{definition}\label{def:product_measure} Let $A_1,...,A_m,B_1,...,B_n$ be a sequence of distinct sets of size $k-1$ from $\mathcal{U}$. Then there exists a unique Keisler measure $\lambda$ on $\mathcal{L}_{x}(\mathcal{U})$ such that 
\begin{equation*}
    \lambda \left(\bigwedge_{i \leq n} R(x,A_i) \wedge \bigwedge_{j \leq m} \neg R(x,B_j) \right) = \frac{1}{2^{n+m}}. 
\end{equation*}
Moreover, this measure is definable (over the empty set) as well as $M$-excellent. For any $I \subseteq \mathbb{N}_{>0}$ such that $|I| \leq k$, we construct a measure $\lambda_{I}$ on $S_{I}^{\flat}$ by considering the pushforward of the measure $\bigotimes_{i \in I} \lambda_{x_i}$ on $S_{(x_i)_{i \in I}}(\mathcal{U})$ to $S_{I}^{\flat}$ via the natural projection map, i.e., $\tp_{I}^{\beta}(-|_{M})$. It is straightforward to check that if $C_{1},...,C_{m},D_{1}$,...,$D_{n}$ is a sequence of distinct sets of size $k - |I|$ from $M$, then 
\begin{equation*}
    \lambda_{I} \left( \bigwedge_{l \leq n} R((x_i)_{i \in I},C_i) \wedge \bigwedge_{j \leq m }R((x_i)_{i \in I},D_j) \right) = \frac{1}{2^{n + m }}. 
\end{equation*}
\end{definition}

\begin{definition}\label{def:hypergraphon} A \textbf{$k$-ary hypergraphon} is a measurable map $W$ where
\begin{equation*}W\colon \prod_{I \in {[k]\choose t \leq k-1}} [0,1]^{I} \to [0,1],
\end{equation*}
and $W$ is invariant under the natural action of $\Sym(k)$ on the indexed spaces. As in the previous section, we will be concerned with Borel hypergraphons (i.e., those satisfying that the map $W$ is Borel-measurable).  As in \cref{WF}, we may identify any Borel hypergraphon $W$ with a map $W'$ where 
\begin{equation*}
    W'\colon \prod_{I \in {[k]\choose t \leq k -1}} S_{I}^{\flat} \to [0,1],
\end{equation*}
where each $S_{I}^{\flat}$ is  equipped with the measure $\lambda_{I}$. For the rest of the section, we will often abuse this identification and think of Borel hypergraphons as maps of the form $W'$.
\end{definition}

\begin{convention} For a hypergraphon $W$, we write $W^{1}(\cdot) = W(\cdot)$ and $W^{-1}(\cdot) = 1- W(\cdot)$. This is just book keeping since we will have functions from index sets to $\{\pm 1\}$. In practice, we will often see terms of the form $W^{\epsilon^{\xi}(B_0,C_0)}$.  
\end{convention}

Just as in the graphon case, given a $k$-ary hypergraphon, one may construct a $\Sym(\mathbb{N})$-invariant probability measure $\mathbb{G}(\mathbb{N},W)$ on $\str_{\mathcal{L}}$ for 
$\mathcal{L} = \{R(x_1,...,x_k)\}$. This measure will concentrate on $\mathcal{L}$-structures which are uniform $k$-ary hypergraphs. 

The following definition can be found in \cite[Page 3]{zhao2015hypergraph}. Technically, the construction there gives a definition for random $k$-uniform hypergraphs on a proper initial segment of $\mathbb{N}$; however, the extension to $\mathbb{N}$ is straightforward. We also refer the reader to \cite[Definition 2.8]{austin2008exchangeable} for an equivalent definition in terms of recipes.

\begin{definition} Let $W$ be a $k$-ary hypergraphon. We define $\mathbb{G}(\mathbb{N},W)$ to be the distribution of the following random $k$-uniform hypergraph (with hyperedge relation $R$) on the vertex set $\mathbb{N}$:

\begin{enumerate}
    \item Select an $\Leb$-i.i.d. array $\mathbf{x}_{A \in N}$ where $N = {\mathbb{N} \choose t \leq k - 1 }$. 
    \item Select an $\Leb$-i.i.d. array $\gamma_{i_1 < ... < i_k \in \mathbb{N}}$. 
    \item For $i_1 < ... < i_k \in \mathbb{N}$, 
    \begin{equation*}
        G(\mathbb{N},W) \models R(i_1,...,i_k) \text{ if and only if } \gamma_{i_1 <...<i_k} < W(\mathbf{x}_{B \subseteq A}),
    \end{equation*}
    where $A = \{i_1,...,i_k\}$. 
\end{enumerate}
\end{definition}

Given a Borel $k$-ary hypergraphon $W$, we will construct a corresponding Keisler measure $\mu_{W}$. Again, no generality will be lost by restricting to the Borel setting.

\begin{definition} Let $n \geq k$ and let $G$ be a $k$-uniform hypergraph on the set $\{1,...,n\}$. The \textbf{hypergraph formula corresponding to $G$} (in the language with a $k$-ary relation symbol $R$) is 
\begin{equation*}
    \varphi_{G}(x_1,...,x_n):= \bigwedge_{G \models R(i_1,...,i_k)}R(x_{i_1},...,x_{i_k}) \wedge \bigwedge_{G \models \neg R(i_1,...,i_k)} \neg R(x_{i_1},...,x_{i_k}). 
\end{equation*}
\end{definition}

\begin{convention} Suppose that $G$ is a hypergraph on a vertex set $\{1,...,n\}$. Given $I \in {[n]\choose k}$, for notational compactness we will sometimes write $W^{G(I)}(-)$ to mean $W(-)$ if $G \models R(I)$ and $1 - W(I)$ if $G \models \neg R(I)$. 
\end{convention} 

\begin{definition}For each $n \geq 1$, we let 
\begin{equation*}
    P_{n} := \prod_{I \in {[n]\choose t \leq k-1}} S_{I}^{\flat}. 
\end{equation*}
\end{definition}

\begin{lemma} Let $W$ be a hypergraphon and $n \geq k$. Suppose that $G$ is a hypergraphon on the vertex set $\{1,...,n\}$. Then 
\begin{align*}
        \mathbb{G}(\mathbb{N}, W)(\llbracket \varphi_{G}(i_{1},...,i_{n})\rrbracket) &= \int_{\bar{t} \in P_{n}} \prod_{ I_0 \in {[n] \choose k}} W^{G(I_0)}(\bar{t}_{I_0}) \, d\bar{\lambda} \\ 
        &= \int_{\bar{t} \in P_{n}} \prod_{\substack{ I_0 \in {[n] \choose k} \\ G \models R(I_0) }} W(\bar{t}_{I_0}) \prod_{\substack{ J_0 \in {[n] \choose k} \\ G \models \neg R(J_0) }} (1 - W(\bar{t}_{J_0})) \, d\bar{\lambda}, 
\end{align*}
where $\bar{t}_{I_0}$ (as well as $\bar{t}_{J_0}$) is the appropriate input for the graphon $W(-)$ with unary data corresponding to the elements from $I_0$, binary data corresponding to pairs of elements from $I_0$, etc. In other words, the tuple $\bar{t}_{I_0}$ is an element of  $\prod_{L \in {I_0 \choose t \leq k-1}} S_{L}^{\flat}$. 
\end{lemma}
\begin{proof} 
This is a straightforward generalization to hypergraphons 
of the standard equation defining induced homomorphism densities (see 
\cite[\S7.2]{lovasz2012large} for induced homomorphism densities of graphons, and \cite[\S4.3]{zhao2015hypergraph} for homomorphism densities of hypergraphons).
\end{proof}

\begin{lemma}\label{lemma:hyper-rep}
Suppose $\nu$ is an ergodic
$\Sym({\Nats})$-invariant measure concentrated on the collection of $k$-uniform hypergraphs with underlying set $\Nats$. Then there is a Borel hypergraphon $W_\nu$ such that $\mathbb{G}(\Nats, W_\nu) = \nu$. 
\end{lemma}
\begin{proof}
The $\Sym({\Nats})$-invariant measure $\nu$  has an 
Aldous--Hoover--Kallenberg representation, considered as a distribution of a jointly exchangeable array \cite[Lemma 7.22]{kallenberg2005probabilistic}.
Moreover, because $\nu$ is ergodic, this representation can be taken to be one that ignores the $\emptyset$-indexed randomness 
\cite[Lemma 7.35]{kallenberg2005probabilistic}.

By the cryptomorphism between hypergraphons and Aldous--Hoover--Kallenberg representations (see \cite{austin2008exchangeable} for hypergraphons, and \cite{diaconis2008graph} and 
\cite[Theorem~11.52 and Proposition~14.62]{lovasz2012large} for more explicit graphon versions), there is a Lebesgue hypergraphon $W$ such that $\mathbb{G}(\Nats, W) = \nu$.

Then by (the hypergraphon analogue of) 
\cite[Exercise~11.31]{lovasz2012large} 
there is a Borel hypergraphon $W'$ that differs from $W$ on a nullset; we likewise have
$\mathbb{G}(\Nats, W') = \nu$. 
\end{proof}

\begin{definition}\label{def:measure2} Given a Borel $k$-ary hypergraphon $W$, we associate an $M$-invariant Keisler measure $\mu_W$. Given finite sets of parameters $B \subseteq M$ and $C \subseteq \mathcal{U}\setminus M$, a \defn{complete formula (over $BC$)} 
is a conjunction  $\Xi(x,B,C)$ of the form
\[ 
\Xi(x,B,C) := \bigwedge_{0 \leq t \leq k - 1}\, \bigwedge_{\substack{C_0 \subseteq C \\ |C_0| = t}}\, \bigwedge_{\substack{B_0 \subseteq B \\ |B_0| = k - 1 - t}} R^{\epsilon^{\Xi}(B_0,C_0)}(x,B_0,C_0), 
\] 
where each $\epsilon^{\Xi}(C_0,B_0)$ is a value in $\{1,-1\}$. Note that every definable set differs by a finite set from a disjoint union of (sets defined by) complete formulas. If we stipulate that the measure of any finite set has value $0$, then it is sufficient to define $\mu_W$ only for complete formulas, and extend to all formulas in the obvious way. 

Given a complete formula $\Xi(x,B,C)$, we let
\begin{equation*}
\Xi_{B}(x) := \bigwedge_{\substack{B_0 \subseteq B \\ |B_0| = k-1}} R^{\epsilon^{\Xi}(B_0,\EmptySet)}(x,B_0). 
\end{equation*} 
Now the value of $\mu_W(\Xi(x,B,C))$ is given by the quantity
\begin{align*}
    \int_{p \in S_{x}^{\flat}} \int_{\bar{q} \in \mathbb{S}_{C} }\mathbf{1}_{\Xi_{B}}(p) \left( \prod_{t = 1}^{ k -2} \prod_{C_0 \in {C \choose t}} \prod_{B_0 \in {B \choose k - 1 - t}} \mathbf{1}_{R^{\epsilon^{\Xi}(B_0,C_0)}(x,B_0,\bar{z}_{t})} (q_{C_0}) \right) \left( \prod_{C_0 \in {C \choose k- 1}} W^{\epsilon^{\Xi}(\EmptySet,C_0)}(p,\mathring{\mathbf{q}}_{[C_0]}) \right)d\bar{\lambda}, 
\end{align*}
where, 
\begin{enumerate}
    \item The set $\mathbb{S}_{C} = \prod_{t = 1}^{k-2} \prod_{C_0 \in {C \choose t}} S_{x\bar{z}_{t}}^{\flat}$ where $\bar{z}_{t} = z_1,...,z_{t}$. Each element $\bar{q}$ of  $\mathbb{S}_{C}$ may be written as $\bar{q} = (q_{C_0})_{C_0 \subseteq C, 0 <|C_0| \leq k - 2}$. This space is also equipped with the appropriate product measure, i.e., each indexed space is equipped with the appropriate measure as defined in Definition \ref{def:product_measure}. The measure $\bar{\lambda}$ is the product of the appropriate measure on $S_{x}^{\flat} \times \mathbb{S}_{C}$. Finally, we remark that the space $\mathbb{S}_{C}$ does not depend on the choice of $C$, but only on its cardinality. Thus we may also write $\mathbb{S}_{C}$ as $\mathbb{S}_{|C|}$.
    \item The tuple $(p,\mathring{\mathbf{q}}_{[C_0]})$ is 
    \begin{equation*}
        (p, ((\tp^{\flat}(C_0'))_{\EmptySet \neq C_0' \subsetneq C_0}, (q_{C_0'})_{\EmptySet \neq C_0' \subsetneq C_0})), 
    \end{equation*}
    i.e., it is the appropriate argument for the hypergraphon $W$. In particular, if $C_0 = \{c_1,...,c_{k-1}\}$, we think of $p,\tp^{\flat}(c_1),...,\tp^{\flat}(c_{k-1})$ as our 1-dimensional data to feed into $W$. For $t \leq k -1$ and index $I \in {[k] \choose t}$, if $I$ does not contain $1$, then the input for this index is $\tp^{\flat}(\{c_{i - 1} \st i \in I\}/M)$. If the index $I$ does contain $1$, then the input for this index is $q_{\{c_{i-1}\st i \in I\}}$. We will also sometimes write $(p,\mathring{\mathbf{q}}_{[C_0]})$ as 
    \begin{equation*}
        (p,[\tp^{\flat}(C_0)],\bar{q}_{[C_0]}), 
    \end{equation*} where $[\tp^{\flat}(C_0)] = (\tp^{\flat}(C_0'/M))_{\EmptySet \neq C_0' \subsetneq C_0}$ and $\bar{q}_{[C_0]} = (q_{C_0'})_{\EmptySet \neq C_0' \subsetneq C_0}$. 
\end{enumerate}
\end{definition}

The following proposition is conceptually straightforward but requires many technical calculations, and so we include its proof in \cref{sec:nightmare-proof}.

\begin{proposition}\label{prop:nightmare} Suppose that $W$ is a Borel $k$-ary hypergraphon. Then $\mu_{W}$ is a Keisler measure. 
\end{proposition}

The proofs of the next several propositions and lemmas are similar to the proofs of the corresponding statements for graphons in Section~\ref{sec:graphon}. The complications arise because of the existence of non-trivial lower-arity data which need to be accounted for. However, their general methods are similar.

\begin{proposition} Suppose that $W$ is a Borel $k$-ary hypergraphon. Then $\mu_{W}$ is $M$-invariant. 
\end{proposition}

\begin{proof} It suffices to check the statement for complete formulas, $\Xi(x,B,C)$. The space $\mathbb{S}_{C}$ depends only on the cardinality of $C$; when we write $C_0$ in the subscript of $q_{C_0}$ where $C_0 \subseteq C$ is such that $|C_0| \leq k - 2$, the $C_0$ is used only for indexing, not for its value.
Moreover, for any set $C_0 \subseteq C$ such that $|C_0| = k-1$, we have that $W(p,\mathring{\mathbf{q}}_{[C_0]})$ only depends on the type of $C_0$ over $M$. And so the value of $\mu(\Xi(x,B,C))$ only depends on the type of $C$ over $M$. 
\end{proof}

\begin{proposition} Suppose that $W$ is a Borel $k$-ary hypergraphon. Then $\mu_{W}$ is Borel-definable over $M$. 
\end{proposition}

\begin{proof} It suffices to consider the case $\Xi(x,\bar{y}) \wedge \tau(\bar{y})$ where if we plug in parameters $D$ for $\bar{y} = y_1,...,y_m$, the resulting formula $\Xi(x,D)$ is complete. We prove that the map $F_{\mu_{W}}^{\Xi(x,\bar{y}) \wedge \tau(\bar{y})}$ is Borel. The proof is quite similar to Proposition \ref{prop:Borel}. For each $i \leq m$, we let $\pi_{i}\colon S_{\bar{y}}(M) \to S_{y_{i}}(M)$ be the natural projection map. For each $A \subseteq [m]$, we consider the Borel set 
\begin{equation*} E_{A} := \bigcap_{k \in A} \pi_{k}^{-1}(M) \cap \bigcap_{k \in [m]\backslash A} \pi_{k}^{-1}(S_{y_k}(M)\backslash M). 
\end{equation*}
Note that $\{E_{A} \st A \subseteq [m] \}$ forms a partition of $S_{\bar{y}}(M)$. For each $A \subseteq [m]$, we consider the function $f_{A}\colon S_{x}^{\flat}(M) \times E_{A} \to [0,1]$ given by
\begin{align*}
    f_{A}(p,r) := \int_{\bar{q} \in \mathbb{S}_{m}} \mathbf{1}_{\Xi(x,\pi_{A}(r))}(p) \prod_{t=1}^{k-2} &\prod_{I_0 \in {[m] \backslash A\choose t}} \prod_{A_0 \in {A\choose k -t - 1}} \mathbf{1}_{R^{\epsilon(A_0,I_0)}(x,\pi_{A_0}(r),\bar{z}_{t})}(q_{I_0}) \\
    & \prod_{I_0 \in {[m] \backslash A \choose k - 1}} W^{\epsilon(\EmptySet,I_0)}(p,[\tp^{\flat}(\pi_{I_0}(r))],\bar{q}_{[I_0]})\bar{\lambda}.
\end{align*}
The above map is Borel by Fact \ref{fact:kechris}. Define the map $h\colon S_{x}^{\flat}(M) \times S_{\bar{y}}(M) \to [0,1]$ via
\begin{equation*}
    h(p,r) := \left( \sum_{A \subseteq [m]} \mathbf{1}_{E_{A}}(q) f_{A}(p,r) \right) \cdot \mathbf{1}_{\tau(\bar{y})}(r). 
\end{equation*}
Again, by Fact \ref{fact:kechris}, we have that 
\begin{equation*}
    r \mapsto \int_{p \in S_{x}^{\flat}} h(p,r)d\lambda_{\{x\}}, 
\end{equation*}
is Borel. However, $F_{\mu_{W}}^{\Xi(x,\bar{y}) \wedge \tau(\bar{y})}(r) = \int_{p \in S_{x}^{\flat}} h(p,r)d\lambda_{\{x\}}$ and so our map is Borel. 
\end{proof}

\begin{proposition}\label{prop:adequate2} Suppose that $W$ is a Borel $k$-ary hypergraphon. Then $\mu_{W}$ is $M$-adequate. Moreover, if $W_1,W_2,W_3$ are Borel $k$-ary hypergraphons, then 
\begin{equation*}
    (\mu_{W_1} \otimes \mu_{W_2}) \otimes \mu_{W_3} = \mu_{W_1} \otimes( \mu_{W_2} \otimes \mu_{W_3}). 
\end{equation*}
\end{proposition}

\begin{proof} This follows directly from the fact that $\mu_{W}$ is Borel-definable over a countable model. See Proposition \ref{prop:adequate} for a similar explanation. 
\end{proof}

Fix a finite set $C \subseteq \mathcal{U} \backslash M$. In the statement of the following lemma, note that the type space $S^*_x(MC)$ factors as $S^\flat_x \times \mathbb{S}_{C}\times S^\ast_x(C)$. 

\begin{lemma}\label{lemma:key3} Suppose $C \subseteq \mathcal{U} \backslash M$ and $|C| = m$. For any Borel-measurable function $f \colon S^{*}_x(MC) \to [0,1]$, we have
    \begin{gather*}
\int_{s \in S_x^*(MC)} f(s) d \mu_W = 
\int_{p \in S^{\flat}_{x}}\int_{\bar{q} \in \mathbb{S}_{C}} \sum_{r \in S_x^\ast(C)} \hat{f}(p,\bar{q},r) W^{\dagger}(p,\bar{q},r)d\bar{\lambda},
    \end{gather*}
   where $W^\dagger(p,\bar{q},r) = \prod_{C_0 \in {C\choose k -1}} W^{\epsilon^{r}(\EmptySet,C_0)}(p,\mathring{\mathbf{q}}_{[C_0]})$ and $r$ is identified with the unique complete formula which isolates it over $C$.
\end{lemma}

\begin{proof} By a similar line of reasoning provided in Lemma \ref{lemma:key} (i.e., linearity of integration and the dominated convergence theorem), it suffices to check the statement for indicator functions of complete formulas. 

Fix a complete formula
\[ 
\Xi(x,B,C)=
\bigwedge_{0 \leq t \leq k - 1}\, \bigwedge_{\substack{B_0 \subseteq B \\ |B_0| = t}}\, \bigwedge_{\substack{C_0 \subseteq C \\ |C_0| = k - 1 - t}} R^{\epsilon^{\Xi}(B_0,C_0)}(x,B_0,C_0), 
\] 
where $B \subseteq M$ and $C \subseteq \mathcal{U}$. Let $f(s) = \mathbf{1}_{\Xi}(s)$. Then we have that 
\begin{equation*}
    \int_{s \in S_{x}^{*}(MC)} f(s) d\mu_{W} = \mu_{W}(\Xi(x,B,C)). 
\end{equation*}

We also have that $\hat{f}(p,\bar{q},r) = 1$ if and only if $\Xi_{B}(x) \in p(x)$ for every non-empty $C_0 \subseteq C$ and non-empty $B_0 \subseteq B$ such that $|C_0| + |B_0| = k -1$, and $R^{\epsilon^{\Xi}(B_0,C_0)}(x,B_0,\bar{z}_{t}) \in q_{C_0}$ and 
$R^{\epsilon^{\Xi}(\EmptySet,C_0)}(x,C_0) \in s$
for all $C_0 \subseteq C$ such that $|C_0| = k -1$.

Let $\Xi_{\mathcal{U}}(x)$ be the formula $\bigwedge_{C_0 \in {C\choose k-1}}R^{\epsilon^{\Xi}(\EmptySet,C_0)}(x,C_0)$. Our argument just established that \[
 \hat{f}(p,\bar{q},r) = \mathbf{1}_{\Xi_{B}}(p)\mathbf{1}_{\Xi_{\mathcal{U}}}(r)\prod_{t = 1}^{k-2} \prod_{C_0 \in {C\choose t}} \prod_{B_0 \in {B\choose k -1 - t}} \mathbf{1}_{R^{\epsilon^{\Xi}(B_0,C_0)}(x,B_0,\bar{z}_{t})}(q_{C_0}).  
\]

Note that there is a unique $r' \in S_{x}^\ast(C)$ such that $\mathbf{1}_{\Xi_{\mathcal{U}}}(r') = 1$. Moreover, note that $r'$ is consistent with $\Xi$ and so 

   \[
    W^{\dagger}(p,\bar{q},r') =  \prod_{C_0 \in {C\choose k -1}} W^{\epsilon^{r'}(\EmptySet,C_0)}(p,\mathring{\mathbf{q}}_{[C_0]}) = \prod_{C_0 \in {C\choose k -1}} W^{\epsilon^{\Xi}(\EmptySet,C_0)}(p,\mathring{\mathbf{q}}_{[C_0]}).  
   \] 
Therefore we have that
\begin{equation*}
     \int_{p \in S^{\flat}_{x}}\int_{\bar{q} \in \mathbb{S}_{C}} \sum_{r \in S_x^\ast(C)} \hat{f}(p,\bar{q},r) W^{\dagger}(p,\bar{q},r)d\bar{\lambda}
\end{equation*}
\begin{equation*}
      = \int_{p \in S^{\flat}_{x}}\int_{\bar{q} \in \mathbb{S}_{C}} \sum_{r \in S_x^\ast(C)} 1_{\Xi_{B}}(p) \left[\prod_{t = 1}^{k-2} \prod_{C_0 \in {C\choose t}} \prod_{B_0 \in {B\choose k -1 - t}} \mathbf{1}_{R^{\epsilon(B_0,C_0)}(x,B_0,\bar{z}_{t})}(q_{C_0}) \right] W^{\Xi}(p,\bar{q},r')d\bar{\lambda},
\end{equation*}
which is precisely $\mu(\Xi(x,B,C))$. 
\end{proof}

\begin{lemma}\label{lemma:hyper-excellent} Suppose that $W$ is a Borel $k$-ary hypergraphon. Then $\mu_{W}$ is $M$-excellent. 
\end{lemma}
\begin{proof} Let $B \subseteq M$ and $C \subseteq \mathcal{U}\backslash M$. For simplicity of notation, we let $\mu_{W} = \mu$. Fix a complete formula $\Xi(x,y,B,C)$ in $xy$ over $BC$. Then $\Xi$ can be written as $\Xi(x,y,B,C) = \psi(x) \wedge \theta(y) \wedge \tau(x,y)$ where
\begin{enumerate}
\item $\psi(x)$ is a complete formula in $x$ over $BC$. We let $\psi_{B}(x)$ and $\psi_{C}(x)$ be the complete formulas restricted to formulas from $B$ and $C$ 
respectively. We let $\psi_{BC}(x)$ be the restriction to the literals which contain parameters non-trivially from both $B$ and $C$. 
\item $\theta(y)$ is a complete formula in $y$ over $BC$. We let $\theta_{B}(y)$ and $\theta_{C}(y)$ be the complete formulas restricted to formulas from $B$ and $C$ 
respectively. We let $\theta_{BC}(y)$ be the restriction to the literals which contain parameters non-trivially from both $B$ and $C$. 
\item Let 
\begin{align*}
    &\tau_1(x,y) := \bigwedge_{B_0 \in {B \choose k- 2}} R^{\epsilon^{\Xi}(B_0,\EmptySet)}(x,y,B_0), \\ 
    &\tau_2(x,y) := \bigwedge_{t=1}^{k-3} \bigwedge_{C_0 \in {C\choose t}} \bigwedge_{B_0 \in {B \choose k - t - 3}} R^{\epsilon^{\Xi}(B_0,C_0)}(x,y,B_0,C_0), \\
    &\tau_3(x,y) := \bigwedge_{C_0 \in {C \choose k- 2}} R^{\epsilon^{\Xi}(\EmptySet,C_0)}(x,y,C_0),
\end{align*}
and let $\tau(x,y) = \tau_1(x,y) \wedge \tau_2(x,y) \wedge \tau_3(x,y)$.
\end{enumerate}
It suffices to show that $(\mu(x) \otimes \mu(y))(\Xi(x,y,B,C)) = (\mu(y) \otimes \mu(x))(\Xi(x,y,B,C))$. We compute that 
\begin{align*}
    (\mu_{x} \otimes \mu_{y})(\Xi(x,y,B,C)) = \int_{S_{x}^{*}(MC)} F_{\mu_{x}}^{\Xi(x,y,B,C)} d\mu_{y}.
\end{align*}
Fix $s \in S_{x}^{*}(MC)$ and let $d \models s$. Then 

\begin{align*}
F_{\mu_{x}}^{\Xi(x,y,B,C)}(s) &= \mu(\Xi(x,d,B,C)) \\
&= \mathbf{1}_{\theta_{B} \wedge \theta_{BC} \wedge \theta_{C}}(s) \int_{v \in S_{x}^{\flat}} \int_{\bar{u} \in \mathbb{S}_{Cd}} \mathbf{1}_{\psi_{B}}(v) \prod_{t = 1}^{k-2} \prod_{C_0 \in {C\choose t}} \prod_{B_0 \in {B\choose k - 1 -t}} \mathbf{1}_{R^{\epsilon^{\psi}(B_0,C_0)}(x,B_0,\bar{z}_t)}(u_{C_0}) \\ 
&\hspace*{30pt}
\prod_{t =0 }^{k-3} \prod_{C_0 \in {C\choose t}}\prod_{B_0 \in {B \choose k - 2 - t}} \mathbf{1}_{R^{\epsilon^{\Xi}(C_0,B_0)}(x,B_0,\bar{z}_{t})}(u_{C_0 \cup \{d\}}) \\
&\hspace*{20pt}
\prod_{C_0 \in {C \choose k-1} } W^{\epsilon^{\psi}(\EmptySet,C_0)}(v,\mathring{\mathbf{u}}_{[C_0]})  \prod_{C_0 \in {C \choose k-2} } W^{\epsilon^{\Xi}(\EmptySet,C_0 \cup \{d\})}(v,\mathring{\mathbf{u}}_{[C_0 \cup \{d\}]}) d\bar{\lambda}.
\end{align*}
Notice that in the term above, the space $\mathbb{S}_{Cd}$ does not depend on $d$, only the cardinality of the set $Cd$. The usage of $d$ in the \emph{virtual types} of the form $u_{C_0 \cup \{d\}}$ is simply for indexing.  Moreover, we have that the only non-trivial usage of $d$ occurs in the term $\prod_{C_0 \in {C\choose k -2}}W^{\epsilon^{\Xi}(\EmptySet,C_0)}(v,\mathring{\mathbf{u}}_{[C_0 \cup \{d\}]})$. However, by definition, the term $\mathring{\mathbf{u}}_{[C_0 \cup \{d\}]}$ depends only on the type of $d$ over $MC$ and so the entire term above is a function of $s$. Hence, in the following computation, the indexing of $\mathbb{S}_{Cd}$ can be replaced by $\mathbb{S}_{C\cup \{*\}}$. Finally notice that if we associated $S_{x}^{*}(MC)$ with $S_{x}^{\flat} \times \mathbb{S}_{C} \times S_{x}^{*}(C)$ as in Lemma \ref{lemma:key3} and let $s = (p,\bar{q},r)$, then 
\begin{align*}
    \mathring{\mathbf{u}}_{[C_0 \cup \{d\}]} &= ([\tp^{\flat}(C_0 \cup \{d\})], \bar{u}_{[C_0 \cup \{d\}]}) \\ 
    &= (\tp^{\flat}(C_0'/M)_{\EmptySet \neq C_0' \subseteq C_0}, \tp^{\flat}(d),(\tp^{\flat}(D \cup \{d\}/M)_{\EmptySet \neq D \subseteq C_0}), \bar{u}_{[C_0 \cup \{d\}]}) \\ 
    &=  ([\tp^{\flat}(C_0)],p,\bar{q}_{[C_0]}, \bar{u}_{[C_0 \cup \{d\}]})).
\end{align*}

Let $F = F_{\mu_{x}}^{\Xi(x,y,B,C)}$. Then by Lemma \ref{lemma:key3}, we see that 
\begin{align*}
(\mu_{x} \otimes \mu_{y})(\Xi(x,y,B,C)) &= \int_{S_{x}^{*}(MC)} F_{\mu_{x}}^{\Xi(x,y,B,C)} d\mu_{y} \\ 
&=  
\int_{p \in S^{\flat}_{x}}\int_{\bar{q} \in \mathbb{S}_{C}} \sum_{r \in S_x^\ast(C)} \widehat{F}(p,\bar{q},r) W^{\dagger}(p,\bar{q},r)d\bar{\lambda} \\ 
&\overset{(*)}{=} \int_{p \in S_{x}^{\flat}} \int_{\bar{q} \in \mathbb{S}_{C}} 
\int_{v \in S_{x}^{\flat}}
\int_{ \bar{u} \in \mathbb{S}_{C\cup \{*\}}} \mathbf{1}_{\theta_{B}}(p) \cdot \mathbf{1}_{\psi_{B}}(v) \\
&\hspace*{30pt}
\prod_{t = 1}^{k-2} \prod_{C_0 \in {C \choose t}} \prod_{B_0 \in {B \choose k - 1 - t}} \mathbf{1}_{R^{\epsilon^{\theta}}(x,B_0,\bar{z}_{t})}(q_{C}) \cdot \mathbf{1}_{R^{\epsilon^{\psi}}(x,B_0,\bar{z}_{t})}(u_{C}) \\ 
&\hspace*{20pt}
\prod_{C_0 \in {C\choose k-1 }} W^{\epsilon^{\psi}(\EmptySet,C_0)}(v,\mathring{\mathbf{u}}_{[C_0]}) \cdot  W^{\epsilon^{\theta}(\EmptySet,C_0)}(p,\mathring{\mathbf{q}}_{[C_0]}) \\ 
&\hspace*{31pt}
\prod_{t = 0}^{k-3} \prod_{C_0 \in {C \choose t}} \prod_{B_0 \in {B \choose k - 2 - t}} 1_{R^{\epsilon^{\Xi}}(x,B_0,\bar{z}_{t})}(u_{C_0 \cup \{*\}}) \\ 
&\hspace*{21pt}
\prod_{C_0 \in {C\choose k - 2}} W^{\epsilon^{\Xi}(\EmptySet,C_0 \cup \{*\})}(v,p, [\tp^{\flat}(C_0)],\bar{q}_{[C_0]}, \bar{u}_{[C_0 \cup \{*\}]}))d\bar{\lambda}. 
\end{align*}
In the above term, the first three lines contain data from $x$ and $y$ separately (and are clearly symmetric). Notice that there exists a unique $r \in S_{x}^{*}(C)$ such that the quantity is non-zero and this is completely determined by $\theta_{C}$. Hence, the term $W^{\dagger}(p,\bar{q},r)$ becomes $\prod_{C_0 \in {C\choose k-1 }} W^{\epsilon^{\theta}(\EmptySet,C_0)}(p,\mathring{\mathbf{q}}_{[C_0]})$.

\noindent Notice that for $C_0 \in {C \choose k-2}$, the tuple $\bar{u}_{[C_0 \cup \{*\}]}$ can be written as 
\begin{equation*}
    (\bar{u}_{[C_0]},(u_{D \cup \{*\}})_{\EmptySet \subseteq D \subseteq C_0}).
\end{equation*}
Hence, up to permutation of variables, the final term becomes 
\begin{equation*}
    \prod_{C_0 \in {C\choose k - 2}} W^{\epsilon^{\Xi}(\EmptySet,C_0 \cup \{*\})}(v,p, [\tp^{\flat}(C_0)],\bar{q}_{[C_0]}, \bar{u}_{[C_0]},(u_{D \cup \{*\}})_{\EmptySet \subseteq D \subseteq C_0}). 
\end{equation*}
Computing $(\mu(y) \otimes \mu(x))(\Xi(x,y,B,C))$ yields a similar result. An application of Fubini's theorem, a change of variable from $u_{D \cup \{*\}}$ to $q_{D \cup \{*\}}$, and the underlying symmetry of our hypergraphon shows that these two terms are equal. 
\end{proof}

\begin{definition} For each $m \geq k-1$, consider the map $\gamma_{m} \colon S_{x_1,...,x_m}^{+}(M) \to P_{m}$ given by
\begin{equation*}
    \gamma_{m}(s) := (\tp_{I_0}^{\flat}(s))_{I_0 \in {[m]\choose t \leq k -1}}. 
\end{equation*}
Notice that this map is a surjective homomorphism between compact Hausdorff spaces and thus a quotient map. 
\end{definition}

\begin{lemma}\label{lemma:hyper-int} Suppose that $W$ is a Borel $k$-ary hypergraphon. For $n \geq k$, suppose that $\hat{h}\colon P_{n} \to [0,1]$ is a Borel measurable map. We let $h = \hat{h} \circ \gamma_{n}$. Let $G$ be a hypergraph on the vertex set $\{1,...n\}$ and $\varphi_{G}(x_1,...,x_n)$ be the corresponding complete hypergraph formula. Then
\begin{align*}
    &\int_{q \in S^{+}_{x_1,...,x_{n}}(M)} \mathbf{1}_{\varphi_{G}(x_1,...,x_{n})}(q) \cdot h(q) d\mu_{W}^{(n)} = \int_{\bar{t} \in P_{n}} \hat{h}(\bar{t}) \prod_{I \in {[n] \choose k}}  W^{\epsilon^{G}(I)}(\bar{t}_{I}) d\lambda_{n}, 
\end{align*}
where $\lambda_{n} = \prod_{I \in {[n]\choose t \leq k-1}} \lambda_{I}$.
\end{lemma}

\begin{proof} For notational purposes, we let $\mu_{W} = \mu$. We prove the statement by induction on $n$. By a similar argument to Lemma \ref{lemma:integral}, it suffices to consider maps of the form 
\begin{equation*}
    \hat{h}(\bar{t}) := \prod_{I_0 \in {[n] \choose t \leq k -1 }}  \mathbf{1}_{\psi_{I_0}(\bar{x}_{I_0})}(\bar{t}_{I_0}), 
\end{equation*}
where $\bar{x}_{I_0} = \{x_i\}_{i \in I_0}$
and $\psi_{I_0}(\bar{x}_{I_0}) = \bigwedge_{B_0 \in {B \choose k-t}} R^{\epsilon^{\psi_{I_0}}(B_0,I_0)}(\bar{x}_{I_0},B_0)$. 
Then $h = \mathbf{1}_{\psi}$ where
\begin{equation*}
    \psi(x_1,...,x_n) := \bigwedge_{t = 1}^{k-1} \bigwedge_{I_0 \in {[n] \choose t }} \bigwedge_{B_0 \in {B \choose k-t}} R^{\epsilon^{\psi}(B_0,I_0)}(\bar{x}_{I_0},B_0). 
\end{equation*}
We let $\psi_{n-1}(\bar{x})$ be the restriction of $\psi$ to the variables $x_1,...,x_{n-1}$, let $\psi_{n}(x_{n})$ be the restriction of $\psi$ to the variable $x_{n}$, and let $\psi'(\bar{x})$ be the restriction to the literals in variables $x_1,...,x_n$ such that every literal contains at least two free variables and one of those variables is $x_{n}$. By definition,  
\begin{equation*}
    \psi(x_1,...,x_n) = \psi_{n}(x_n) \wedge \psi'(\bar{x}) \wedge  \psi_{n-1}(\bar{x}). 
\end{equation*}

\vspace{.1in}

\noindent \textbf{Base Case:} Suppose that $n = k$. Then $\varphi_{G}(x_1,...,x_k) = R(x_1,...,x_k)$ or $\varphi_{G}(x_1,...,x_k) = \neg R(x_1,...,x_k)$. Without loss of generality, suppose that $\varphi_{G}(x_1,...,x_k) = R(x_1,...,x_k)$. Then 
\begin{align*}
    \int_{q \in S^{+}_{x_1,...,x_k}(M)} \mathbf{1}_{\varphi_{G}(x_1,...,x_k)} (q) \cdot h(q) d\mu^{(n)} &= \mu^{(n)}(\varphi_{G}(x_1,...,x_k) \wedge \psi(x_1,...,x_k)) \\  
    &= (\mu_{x_k} \otimes \mu_{\bar{x}}^{(k-1)})(R(x_1,..,x_k) \wedge \psi(x_1,...,x_k))  \\
    &= \int_{s \in S^{+}_{x_1,..,x_{k-1}}(M)} F_{\mu_{x_k}}^{R(x_k;\bar{x}) \wedge \psi(x_k;\bar{x})} d\mu_{\bar{x}}^{(k-1)}. 
\end{align*}
Fix $s \in S^{+}_{x_1,...,x_{k-1}}(M)$ and $C \subseteq \mathcal{U}$ such that $C \models s$. Then
\begin{align*}
    F_{\mu_{x_k}}^{R(x_k;\bar{x}) \wedge \psi(x_k;\bar{x})}(s) &= \mu(R(x_{k};C) \wedge \psi(x_{k},C)) \\
    &= \mathbf{1}_{\psi_{k-1}(\bar{x})}(s) \cdot \mu(R(x_{k},C) \wedge \psi_{k}(x_k) \wedge \psi'(x_{k},C)).
\end{align*}
Now, notice that $R(x_{k},C) \wedge \psi_{k}(x_k) \wedge \psi'(x_{k},C)$ is a complete formula, and so
\begin{align*}
    &\mu(R(x_{k},C) \wedge \psi_{k}(x_k) \wedge \psi'(x_{k},C)) \\ 
    &\hspace*{30pt}=  \int_{p \in S_{x}^{\flat}} \int_{\bar{q} \in \mathbb{S}_{C}} \mathbf{1}_{\psi_{k}}(p) \prod_{t=1}^{k-2}\prod_{C_0 \in  {C \choose t}}\prod_{B_0 \in  {B \choose k-1-t}} \mathbf{1}_{R^{\epsilon^{\psi}(B_0,C_0)}(x,B_0,\bar{z}_t)}(q_{C_0}) W(p,[\tp^{\flat}(C)], \bar{q}_{[C]}) d\bar{\lambda} \\
    &\hspace*{30pt}=  \int_{p \in S_{x}^{\flat}} \int_{\bar{q} \in \mathbb{S}_{k-1}} \mathbf{1}_{\psi_{k}}(p) \prod_{t=1}^{k-2}\prod_{I_0 \in  {[k-1] \choose t}}\prod_{B_0 \in  {B \choose k-1-t}} \mathbf{1}_{R^{\epsilon^{\psi}(B_0,I_0)}(x,B_0,\bar{z}_t)}(q_{I_0}) W(p,\gamma_{k-1}(s),\bar{q}_{[k-1]}) d\bar{\lambda}, 
\end{align*}
where $\bar{\lambda}$ is the appropriate measure. Notice that the expression above is a function of $s$ and does not depend on the realization $C$. It only depends on the image of $s$ under our map $\gamma_{k-1}$. Moreover, we note that the value of $\mathbf{1}_{\psi_{k-1}(\bar{x})}(s)$ also only depends on the image of $\gamma_{k-1}(s)$. We now compute: 
\begin{align*}
    &\int_{q \in S^{+}_{x_1,...,x_k}(M)} \mathbf{1}_{\varphi_{G}(x_1,...,x_n)}(q) \cdot h(q) d\mu^{(k)}\\
    &\hspace*{30pt}=
    \int_{s \in S^{+}_{x_1,...,x_{k-1}}(M)} F_{\mu_{x_k}}^{R(x_k;\bar{x}) \wedge \psi(x_k;\bar{x})} d\mu^{(k-1)} \\
    &\hspace*{30pt}=
    \int_{s \in S^{+}_{x_1,...,x_{k-1}}(M)} \mathbf{1}_{\psi_{k-1}(\bar{x})}(s) \int_{p \in S_{x}^{\flat}} 
    \int_{\bar{q} \in \mathbb{S}_{k-1}}  \mathbf{1}_{\psi_{k}}(p) \prod_{t=1}^{k-2}\prod_{I_0 \in  {[k-1] \choose t}}\prod_{B_0 \in  {B \choose k-1-t}} \\
    &\hspace*{60pt}
    \mathbf{1}_{R^{\epsilon^{\psi}(B_0,I_0)}(x,B_0,\bar{z}_t)}(q_{I_0}) W(p,\gamma_{k-1}(s),\bar{q}_{[k-1]}) d\bar{\lambda} d\mu^{(k-1)}\\
    &\hspace*{30pt}\overset{(a)}{=}
    \int_{\bar{u} \in P_{k-1}} \int_{p \in S_{x_{k}}^{\flat}} \int_{\bar{q} \in \mathbb{S}_{k-1}} 
    \left( \prod_{I_0 \in {[k-1]\choose t \leq k -1}} \mathbf{1}_{\psi_{I_0}(\bar{x}_{I_0})}(u_{I_0}) \right) \cdot [\mathbf{1}_{\psi_{k}}(p)] \prod_{t=1}^{k-2}\prod_{I_0 \in  {[k-1] \choose t}}\prod_{B_0 \in  {B \choose k-1-t}} \\
    &\hspace*{60pt}
\mathbf{1}_{R^{\epsilon^{\psi}(B_0,I_0)}(x,B_0,\bar{z}_t)}(q_{I_0}) W(p,\bar{u},\bar{q}_{[k-1]}) d\bar{\lambda}' d(\gamma_{k-1})_{*}(\mu^{(k-1)})\\
    &\hspace*{30pt}\overset{(b)}{=}
\int_{(\bar{u},p,\bar{q}) \in P_{k-1} \times S_{x_k}^{\flat} \times \mathbb{S}_{k-1}} \hat{h}(p,\bar{u},\bar{q}) W(p,\bar{u},\bar{q}_{[k-1]}) d(\bar{\lambda}' \times (\gamma_{k-1})_*(\mu^{(k-1)}))
\\
    &\hspace*{30pt}\overset{(c)}{=}
    \int_{\bar{t} \in P_{n}} \hat{h}(\bar{t}) W(\bar{t})d\lambda_{n}.
\end{align*}
We provide some justifications: 
\begin{enumerate}[label=($\alph*$)]
    \item This equation follows by realizing the function in the integrand factors through the quotient map $\gamma_{k-1}$. This is explained in the paragraph before the computation. Hence we may integrate with respect to the pushforward measure $\gamma_{k-1}(\mu^{(k-1)})$ over $P_{k-1}$. We also renamed the variable $x$ to $x_{k}$, which thus changes $\bar{\lambda}$ to $\bar{\lambda}'$ where $\bar{\lambda}' = \prod_{I_0 \in{ [k-1]\choose t \leq k - 2}} \lambda_{I_0 \cup \{k\}}$.
    \item Notice that the space $P_{k-1} \times S_{x_{k}}^{\flat} \times \mathbb{S}_{k-1}$ is canonically identified with $P_{n}$. Intuitively, the unary data is given by $\{u_{\{1\}},...,u_{\{k-1\}},p\}$, the higher-arity data connecting elements of $\{u_{\{i\}}\}_{i \leq k-1}$ is already present in the tuple $\bar{u}$, and the higher-arity data connecting $p$ with $\{u_{\{i\}}\}_{i \in I_0}$ is precisely given by $q_{I_0}$. 
    \item It is clear that $(\gamma_{k-1})_{*}(\mu^{(k-1)})= \prod_{I_0 \in {[k-1] \choose t \leq k-1}} \lambda_{I_0}$, and so $\bar{\lambda}' \times (\gamma_{k-1})_*(\mu^{(k-1)})$ gives the appropriate measure on $P_{k}$. 
\end{enumerate}

\noindent \textbf{Induction Step:} Suppose the statement holds for $n-1$. We prove the statement for $n$. Let $\varphi_{G_{n-1}}(x_1,...,x_{n-1})$ be the hypergraph formula corresponding to the restriction of $G$ to the vertex set $\{1,...,n\}$. In other words, it is the restriction of the formula $\varphi_{G}(x_1,...,x_n)$ to the variables $x_1,...,x_{n-1}$. Define the function
\begin{align*}
    \widehat{h_0}(\bar{t}) = \prod_{I_0 \in {[n-1]\choose t \leq k-1}} \mathbf{1}_{\psi_{I_0}(\bar{x}_{I_0})}(t_{I_0}) \int_{p \in S_{x_{n}}^{\flat}} \int_{\bar{q} \in \mathbb{S}_{n-1}} &\mathbf{1}_{\psi_{n}}(p) \prod_{t=1}^{k-2} \prod_{L_0 \in {[n-1] \choose t}} \prod_{B_0 \in {B \choose k - 1 -t}} \mathbf{1}_{R^{\epsilon^{\psi}(B_0,L_0)}}(q_{L_0})  \\
    &\prod_{J_0 \in {[n-1] \choose k -1}} W^{\epsilon^{G}(J_0 \cup \{n\})}(p,\mathring{\mathbf{t}}_{[J_0]},\bar{q}_{[J_0]})d\bar{\lambda},
\end{align*}
where $\mathring{\mathbf{t}}_{[J_0]} = (t_{I_0})_{\EmptySet \neq I_0 \subseteq J_0}$. If $J_0 \in {[n]\choose k-1}$ and $s \in S_{x_1,...,x_n}^{+}(M)$, then we let $[\tp^{\flat}_{J_0}(s)] = (\tp_{L_0}^{\flat}(s))_{\EmptySet \neq L_0 \subseteq J_0}$. Notice that if $s \in S_{x_1,...,x_{n-1}}^{+}(M)$ and $C \models s$, then
\begin{align*}
    F_{\mu_{x_{n}}}^{\varphi_{G}(x_n,\bar{x}) \wedge \psi(x_n;\bar{x})}(s)   
    &= \mu_{x_n}(\varphi_{G}(x_n,C) \wedge \psi(x_n,C)) \\
    &= \mathbf{1}_{G_{n-1}(\bar{x})}(s) \cdot \mathbf{1}_{\psi_{n-1}(\bar{x})}(s) \cdot \int_{p \in S_{x_n}^{\flat}} \int_{\bar{q} \in \mathbb{S}_{C}} \mathbf{1}_{\psi_{n}}(p) \prod_{t-1}^{k-2} \prod_{C_0 \in {C\choose t}} \prod_{B_0 \in {B\choose k-t-1}} \mathbf{1}_{R^{\epsilon^{\psi}(B_0,C_0)}}(q_{C_0})  \\
    &\hspace*{30pt}
     \prod_{D_0 \in {C \choose k-1}} W^{\epsilon(\EmptySet,D_0)}(p,[\tp^{\flat}(D_0)], \bar{q}_{[D_0]})d\bar{\lambda} \\ 
    &= \mathbf{1}_{G_{n-1}(\bar{x})}(s) \cdot \mathbf{1}_{\psi_{n-1}(\bar{x})}(s) \cdot \int_{p \in S_{x_n}^{\flat}} \int_{\bar{q} \in \mathbb{S}_{n-1}} \mathbf{1}_{\psi_{n}}(p) \prod_{t-1}^{k-2} \prod_{I_0 \in {[n-1]\choose t}} \prod_{B_0 \in {B\choose k-t-1}} \mathbf{1}_{R^{\epsilon^{\psi}(B_0,I_0)}}(q_{I_0})  \\
    &\hspace*{30pt}
     \prod_{J_0 \in {[n-1] \choose k-1}} W^{\epsilon^{G}(J_0 \cup \{n\})}(p, [\tp_{J_0}^{\flat}(s)],
     \mathring{\mathbf{q}}_{[J_0]})d\bar{\lambda} \\ 
    &=\mathbf{1}_{G_{n-1}(x_1,...,x_{n-1})}(s) \cdot (\widehat{h_0} \circ \gamma_{n-1})(s). 
\end{align*}
Setting $h_0 = (\widehat{h_0} \circ \gamma_{n-1})$, we have the following sequence of equations.
\begin{align*}
    \int_{q \in S^{+}_{x_1,...,x_n}(M)} \mathbf{1}_{\varphi_{G}(x_1,...,x_{n})}(q) \cdot h(q) d\mu^{(n)}   
    &= \mu^{(n)}(\varphi_{G}(\bar{x}) \wedge \psi(\bar{x}))  \\
    &= \int_{q \in S^{+}_{x_1,...,x_{n-1}}(M)} F_{\mu_{x_n}}^{\varphi_{G}(x_n;\bar{x}) \wedge \psi(x_n;\bar{x})}(q) d\mu^{(n-1)}  \\
    &= \int_{q \in S^{+}_{x_1,...x_{n-1}}(M)} \mathbf{1}_{G_{n-1}(\bar{x})}(q) h_0(q) d\mu^{(n-1)}  \\
    &\overset{(*)}{=} \int_{\bar{t} \in P_{n-1}} \widehat{h_0}(\bar{t}) \prod_{I \in {[n-1] \choose k - 1}} W^{\epsilon^{G_{n-1}(I)}}(\mathring{\mathbf{t}}_{[I]}) d\lambda_{n-1} \\ 
    &\overset{(**)}{=} \int_{\bar{t} \in P_{n}} \hat{h}(\bar{t}) \prod_{I \in {n\choose k-1}} W^{\epsilon^{G}(I)}(\mathring{\mathbf{t}}_{[I]}) d\lambda_{n}.
\end{align*}
Equation $(*)$ is an application of our induction hypothesis. Equation $(**)$ follows from unpacking the function $\hat{h}_0$ and reorganizing the terms. 
\end{proof}

\begin{theorem}\label{theorem:main} Let $W$ be a Borel $k$-ary hypergraphon and $g\colon S^{+}_{\mathbf{x}}(\mathcal{U}) \to \str_{\mathcal{L}}$. Then $g_{*}(\kmpow{\mu_{W}}) = \mathbb{G}(\mathbb{N}, W)$. 
\end{theorem}

\begin{proof}To simplify notation, we denote $\mu_{W}$ simply as $\mu$. By Lemma \ref{lemma:hyper-excellent} we have that $\mu$ is $M$-excellent and so $g_{*}(\kmpow{\mu})$ is $\Sym(\mathbb{N})$-invariant by Fact \ref{fact:excellent}. Hence it suffices to show that $g_{*}(\kmpow{\mu})$ and $\mathbb{G}(\mathbb{N}, W)$ agree on sets of the form $\llbracket \varphi_{G}(1,...,n +1) \rrbracket$ where $G$ describes a hypergraph on $\{1,...,n +1\}$ and $n \geq k-1$. Notice that $g^{-1}(\llbracket \varphi_{G}(1,...,n +1)\rrbracket)$ is the clopen set corresponding to $[\varphi_{G}(x_1,...,x_{n+1})]$. We let $\varphi_{G_{n}}(x_1,...,x_{n})$ be the hypergraph formula corresponding to the hypergraph $G$ restricted to $\{1,...,n\}$. Notice that if $n = k-1$, then the base case of Lemma \ref{lemma:hyper-int} gives a proof. Hence we only need to consider $n \geq k$. 

We define $\hat{h}\colon P_{n} \to [0,1]$ via 
\begin{equation*}
\hat{h}(\bar{t}):=\int_{p\in S_{x_{n+1}}^{\flat}}\int_{\bar{q}\in\mathbb{S}_{n}}\prod_{I\in{[n] \choose k-1}}W^{\epsilon^{G}(I \cup \{n+1\})}(p,\mathring{\mathbf{t}}_{[I]},\bar{q}_{[I]})d\bar{\lambda}, 
\end{equation*}
where $\mathring{\mathbf{t}}_{[I]} = (t_{I_0})_{\EmptySet \neq I_0 \subseteq I}$. If $I \in {[n]\choose k-1}$ and $s \in S_{x_1,...,x_n}^{+}(M)$, then we let $[\tp^{\flat}_{I}(s)] = (\tp_{I_0}^{\flat}(s))_{\EmptySet \neq I_0 \subseteq I}$. Now notice that if $s \in S_{x_1,...,x_n}^{+}(M)$, we have that 
\begin{align*}
    F_{\mu_{x_{n+1}}}^{\varphi_{G}(x_{n+1};\bar{x})}(s) &= \mu(\varphi_{G}(x_{n+1},C)) \\ 
    &= \mathbf{1}_{\varphi_{G_{n}}(x_1,...,x_n)}(s) \cdot \int_{p \in S_{x_{n+1}}^{\flat}} \int_{\bar{q} \in \mathbb{S}_{C}} \prod_{C_0 \in {C \choose k -1 }} W^{\epsilon^{G}(\EmptySet,C_0)}(p,\mathring{\mathbf{q}}_{[C_0]})d\bar{\lambda} \\ 
    &= \mathbf{1}_{G_{n}(\bar{x})}(s) \int_{p \in S_{x_{n+1}}^{\flat}} \int_{\bar{q} \in \mathbb{S}_{C}} \prod_{C_0 \in {C \choose k -1 }} W^{\epsilon^{G}(\EmptySet,C_0)}(p,[\tp^{\flat}(C_0)], \bar{q}_{[C_0]})d\bar{\lambda} \\ 
    &= \mathbf{1}_{G_{n}(\bar{x})}(s) \int_{p \in S_{x_{n+1}}^{\flat}} \int_{\bar{q} \in S_{n}} \prod_{I \in {[n] \choose k -1}} W^{\epsilon^{G}(I \cup \{n+1\})}(p, [\tp_{I}^{\flat}(s)] ,\bar{q}_{[I]})d\bar{\lambda} \\ 
    &= \mathbf{1}_{G_{n}(\bar{x})}(s) \cdot (\hat{h} \circ \gamma_{n})(s).
\end{align*}
\noindent Setting $h = \hat{h} \circ \gamma_{n}$ and applying Lemma \ref{lemma:hyper-int}, we see that
\begin{align*}
    g_{*}(\kmpow{\mu})(\llbracket \varphi_{G}(1,...,n+1)\rrbracket) 
    &=\kmpow{\mu}(g^{-1}(\llbracket \varphi_{G}(1,...,n+1)\rrbracket)) \\ 
    &= \kmpow{\mu}(\varphi_{G}(x_1,...,x_{n+1})) \\
    &= \left( \bigotimes_{i=1}^{n+1} \mu_{x_i} \right) (\varphi_{G}(x_1,...,x_{n+1})) \\
    &= \left(\mu_{x_{n+1}}  \otimes  \bigotimes_{i = 1}^{n} \mu_{x_i} \right)(\varphi_{G}(x_1,...,x_{n+1})) \\
    &= \int_{s \in S_{x_1,...,x_{n}}^{+}(M)} F_{\mu_{x_{n+1}}}^{\varphi_{G}^{*}(x_{n+1};\bar{x})}(s) d\mu^{(n)} \\
    &= \int_{s \in S_{x_1,...,x_{n}}^{+}(M)} \mathbf{1}_{G_{n}(x_1,...,x_{n})}(s) \cdot h(s) d\mu^{(n)} \\
    &\overset{(*)}{=} \int_{\bar{t} \in P_{n}} \hat{h}(\bar{t}) \prod_{I \in {[n] \choose k}}  W^{\epsilon^{G_{n}}(I)}(\bar{t}_{I}) d\lambda_{n} \\
  &\overset{(**)}{=} \int_{\bar{t} \in P_{n}} \int_{p \in S_{x_{n+1}}^{\flat}} \int_{\bar{q} \in \mathbb{S}_{n}} \prod_{I_0 \in {[n] \choose k-1}} W^{\epsilon^{G}(I_0 \cup \{x_{n+1}\})}(p,\mathring{\mathbf{t}}_{[I_0]},\bar{q}_{[I_0]})\prod_{I \in {[n] \choose k}} W^{\epsilon^{G}(I)}(\bar{t}_I) d\lambda_{n+1} \\ 
  &= \int_{\bar{t} \in P_{n+1}} \prod_{J \in {n+1 \choose k}} W^{\epsilon^{G}(J)}(\bar{t}_{J})d\lambda_{n+1} \\ 
  &= \mathbb{G}(\mathbb{N}, W)(\varphi_{G}(x_1,...,x_{n+1})). 
\end{align*}
Equation $(*)$ follows from Lemma \ref{lemma:hyper-int} while equation $(**)$ follows from unpacking $\hat{h}$ and reorganizing terms. 
\end{proof}

\begin{corollary}\label{cor:AFP2} Let $G$ be a countable $k$-uniform hypergraph. Then the following are equivalent. 
\begin{enumerate}
    \item There exists an $M$-excellent Keisler measure $\mu$ which does not concentrate on points and such that $\kmpow{\mu}(\mathbb{B}_{G}) = 1$.
    \item There exists a $\Sym(\mathbb{N})$-invariant measure $\eta$ on $\str_{\mathcal{L}_{k}}$ such that $\eta(\mathbb{A}_{G}) = 1$. 
    \item The hypergraph $G$ has trivial group theoretic definable closure, i.e., for any finite tuple $\bar{a} = a_1,...,a_n$ from $G$, we have that $\dcl(\bar{a}) = \bar{a}$, where $\dcl(\bar{a})$ is the collection of elements from $G$ that are fixed by all the automorphisms of $G$ which fix $\bar{a}$ pointwise.
\end{enumerate}   
\end{corollary}

\begin{proof} Statements $(2)$ and $(3)$ are equivalent via \cite[Theorem 1.1]{ackerman2016invariant}. By 
Fact \ref{fact:excellent}, statement $(1)$ implies statement $(2)$.

Assume statement $(2)$. By a straightforward generalization of an observation in \cite[subsection 6.1.2] {ackerman2016invariant}, there exists a \emph{random-free hypergraphon} $W$ (which is Borel) such that $\mathbb{G}(\mathbb{N}, W)$ concentrates on $\mathbb{A}_{G}$. 
By Lemma \ref{lemma:excellent}, the measure $\mu_{W}$ is $M$-excellent and obviously does not concentrate on points. By Theorem \ref{theorem:main}, $g_{*}(\kmpow{\mu_{W}}) = \mathbb{G}(\mathbb{N}, W)$ and so $\mathbb{P}(\mu_{W})(\mathbb{B}_{G}) = 1$.
Hence statement $(1)$ holds.
\end{proof}

\begin{problem} In \cite[Corollary 5.1 +  Corollary 5.2]{braunfeld2024invariance}, Braunfeld, Jahel and Marimon classified $\EmptySet$-invariant measures over $T_{k}$ (as well as over some omitting subgraph variants). What does generic sampling look like with respect to these measures? Prove a theorem in the same vein as \cref{prop:emptyset1and2,prop:emptyset3,prop:emptyset4}.  
\end{problem}

\begin{corollary}\label{cor:hyper_ergodic} Suppose $\nu$ is an ergodic $\Sym(\mathbb{N})$-invariant measure concentrated on the collection of $k$-uniform hypergraphs with underlying set $\mathbb{N}$. Then there exists an $M$-excellent Keisler measure $\mu \in \mathfrak{M}^{\inv}_{x}(\mathcal{U},M)$ which does not concentrate on points and $g_{*}(\mathbb{P}_{\mu}) = \nu$. 
\end{corollary} 

\begin{proof} Follows directly from \cref{lemma:hyper-rep} and \cref{theorem:main}. 
\end{proof}

\appendix

\section{Proof of \cref{prop:nightmare}}
\label{sec:nightmare-proof}

In this appendix, we give the proof that $\mu_{W}$ is a Keisler measure for any Borel $k$-ary hypergraphon.
We will use the following elementary result.

\begin{lemma}\label{lemma:sumprod} For $a_1, \ldots ,a_k \in [0,1]$, we have
\begin{equation*}
    \sum_{f \colon \{1,...,k\} \to \{0,1\}} \prod_{i=1}^{k} a_i^{f(i)}(1-a_i)^{f(i)} = 1. 
\end{equation*}
\end{lemma}
\begin{proof}
Flip a sequence of $k$ independent coins with respective weights $a_1$ through $a_k$. Each function $f$ describes one possible outcome, and exactly one such outcome must occur.
\end{proof}

The rest of this section constitutes the proof of \cref{prop:nightmare}.

\begin{proof}[Proof of \cref{prop:nightmare}]
Suppose that $B \subseteq M$ and $C \subseteq \mathcal{U}\backslash M$, and let $\Xi(x,B,C)$ be a complete formula over $BC$. We will argue that for any instance of the formula $R(x,\bar{y})$, say $\gamma(x)$, 
\begin{equation*}
    \mu(\Xi(x,B,C) \wedge \gamma(x)) + \mu(\Xi(x,B,C) \wedge \neg \gamma(x)) = \mu(\Xi(x,B,C) ).  
\end{equation*}
By quantifier elimination and induction, our proof is complete. We prove 3 cases: 

\vspace{.1in}

\noindent \textbf{Case 1:} Suppose that $D \subseteq M$ where $|D| = k - 1$ and $D \cap B = \EmptySet$. Set $\gamma(x) =R(x,D)$. We remark that the case of $D \subseteq M$ where $|D| = k - 1$ and $D \cap B \neq \EmptySet$ is almost identical but with slightly different bookkeeping. 
\begin{enumerate}
    \item  We let $S_{\Xi}^{-}(BCD)$ be the collection of formulas (with parameters from $BCD$) such that for every $r \in S_{\Xi}^{-}(BCD)$, $r \vdash \Xi$ and both $r \wedge {R(x,D)}$ and $r \wedge \neg R(x,D)$ are complete (over $BCD$). In other words, these formulas make decision about all possible $(k-1)$-tuples of elements from $BCD$, they agree with $\Xi$, but remain agnostic on $R(x,D)$. 
    \item If $r \in S_{\Xi}^{-}(BCD)$, then $\psi_{BD}^{r}$ is the formula corresponding to the intersection of the  subcollection of literal formulas from $r$ where each literal contains parameters only from $B$ and $D$. In other words, $\psi_{BD}^{r} \vdash R(x,E)$ if and only if both $E \subseteq B \cup D$ and $r \vdash R(x,E)$. Notice that $\psi_{BD}^{r}$ remains agnostic on $R(x,D)$.  
    \item Recall if $Q_0 \subseteq BD$ and $C_0 \subseteq C$ such that $|C_0| + |Q_0| = k-1$, then ${\epsilon^{r}}(Q_0,C_0) = 1$ if $r \vdash R(x,Q_0,C_0)$ and ${\epsilon^{r}}(Q_0,C_0) = -1$ if $r \vdash \neg R(x,Q_0,C_0)$. 
\end{enumerate}
Now consider the following computation: 
\begin{align*}
    &\mu(\Xi(x,B,C) \wedge \gamma(x)) + \mu(\Xi(x,B,C) \wedge \neg \gamma(x)) \\  
    &\hspace*{2pt}=
    \sum_{r \in S_{\Xi}^{-}(BCD)}  \mu(\Xi(x,B,C) \wedge \gamma(x) \wedge r(x)) + \mu(\Xi(x,B,C) \wedge \neg \gamma(x) \wedge r(x)) \\
    &\hspace*{2pt}=
    \sum_{r \in S_{\Xi}^{-}(BCD)} \int_{p \in S^{\flat}_{x}} \int_{\bar{q} \in \mathbb{S}_{C}} [\mathbf{1}_{\Xi_{B} \wedge \psi^{r}_{BD} \wedge \gamma} + \mathbf{1}_{\Xi_{B} \wedge \psi^{r}_{BD} \wedge \neg \gamma}] \\ 
    &\hspace*{20pt}
    \left( \prod_{t=1}^{k-2} \prod_{C_0 \in { C\choose t }} \prod_{Q_0 \in {BD \choose k - 1 -t }} \mathbf{1}_{R^{\epsilon^{r}(Q_0,C_0)}(x,Q_0,\bar{z}_t)}(q_{C_0}) \right) \left(\prod_{C_0 \in {C\choose k-1}}W^{\epsilon(\EmptySet,C_0)} (p,\mathring{\mathbf{q}}_{[C_0]}) \right) d\bar{\lambda} \\
    &\hspace*{2pt}=
      \int_{p \in S^{\flat}_{x}} \int_{\bar{q} \in \mathbb{S}_{C}}   \left( \prod_{C_0 \in {C\choose k-1}}W^{\epsilon(\EmptySet,C_0)} (p,\mathring{\mathbf{q}}_{[C_0]}) \right)  \\  
    &\hspace*{20pt}
    \left( \sum_{r \in S_{\Xi}^{-}(BCD)} \mathbf{1}_{\Xi_{B} \wedge \psi^{r}_{BD}}(p) \prod_{t=1}^{k-2} \prod_{C_0 \in { C\choose t }} \prod_{Q_0 \in {BD \choose k - 1 -t }} \mathbf{1}_{R^{\epsilon^{r}(Q_0,C_0)}(x,Q_0,\bar{z}_t)}(q_{C_0}) \right) d\bar{\lambda} \\
    &\hspace*{2pt}\overset{(*)}{=}  
    \int_{p \in S^{\flat}_{x}} \int_{\bar{q} \in \mathbb{S}_{C}}   \left( \prod_{C_0 \in {C\choose k-1}}W^{\epsilon(\EmptySet,C_0)} (p,\mathring{\mathbf{q}}_{[C_0]}) \right)  \left(  \mathbf{1}_{\Xi_{B}}(p) \prod_{t=1}^{k-2} \prod_{C_0 \in { C\choose t }} \prod_{B_0 \in {B \choose k - 1 -t }} \mathbf{1}_{R^{\epsilon^{\Xi}(B_0,C_0)}(x,B_0,\bar{z}_t)}(q_{C_0}) \right) d\bar{\lambda} \\ 
    &\hspace*{2pt}=
    \mu(\Xi(x,B,C)).
\end{align*}
We now justify Equation $(*)$. Consider the functions given by 
\begin{equation*}
    G(p,\bar{q}) :=  \sum_{r \in S_{\Xi}^{-}(BCD)} \mathbf{1}_{\Xi_{B} \wedge \psi^{r}_{BD}}(p) \prod_{t=1}^{k-2} \prod_{C_0 \in { C\choose t }} \prod_{Q_0 \in {BD \choose k - 1 -t }} \mathbf{1}_{R^{\epsilon^{r}(C_0,Q_0)}(x,Q_0,\bar{z}_t)}(q_{C_0})
\end{equation*}
and
\begin{equation*}
    H(p,\bar{q}) :=   \mathbf{1}_{\Xi_{B}}(p) \prod_{t=1}^{k-2} \prod_{C_0 \in { C\choose t }} \prod_{B_0 \in {B \choose k - 1 -t }} \mathbf{1}_{R^{\epsilon^{\Xi}(B_0,C_0)}(x,B_0,\bar{z}_t)}(q_{C_0}).
\end{equation*}
We claim that $G(p,\bar{q}) = 1$ if and only if $H(p,\bar{q}) = 1$; since the functions are $\{0,1\}$-valued, they are equal. Indeed, if $G(p,\bar{q}) = 1$, then there exists a unique $r$ such that 
\begin{equation*}
    \mathbf{1}_{\Xi_{B} \wedge \psi^{r}_{BD}}(p) \prod_{t=1}^{k-2} \prod_{C_0 \in { C\choose t }} \prod_{Q_0 \in {BD \choose k - 1 -t }} \mathbf{1}_{R^{\epsilon^{r}(C_0,Q_0)}(x,Q_0,\bar{z}_t)}(q_{C_0}) =1.
\end{equation*}
But then, $H(p,\bar{q}) = 1$ since $H$ is a subproduct (of characteristic functions). On the other hand, if $H(p,\bar{q}) = 1$, we consider the unique $r_{*} \in S_{\Xi}^{-}(BCD)$ such that 
\begin{equation*}
    r_{*} \vdash p|_{BD} \cup \bigcup_{\substack{C_0 \subseteq C \\ Q_0 \subseteq BD}} \left\{ \begin{array}{cc}
R(x,Q_{0},C_{0}) & R(x,Q_{0},\bar{z}_{t})\in q_{C_{0}}\\
\neg R(x,Q_{0},C_{0}) & \neg R(x,Q_{0},\bar{z}_{t})\in q_{C_{0}}
\end{array} \right\} .
\end{equation*}
Hence the following equality holds:
\begin{equation*}
    \mathbf{1}_{\Xi_{B} \wedge \psi^{r_*}_{BD}}(p) \prod_{t=1}^{k-2} \prod_{C_0 \in { C\choose t }} \prod_{Q_0 \in {BD \choose k - 1 -t }} \mathbf{1}_{R^{\epsilon^{r_*}(C_0,Q_0)}(x,Q_0,\bar{z}_t)}(q_{C_0}) = 1.
\end{equation*}
Since $r_*$ is unique element in $S_{\Xi}^{-}(BCD)$ with such property, we conclude that $G(p,\bar{q}) = 1$. 

\vspace{.1in}

\noindent \textbf{Case 2:} Suppose that $E \subseteq \mathcal{U} \backslash M$ where $|E| = k -1$ and $E \cap C = \EmptySet$. Set $\gamma(x) = R(x,E)$. Again, we remark that the case of $E \subseteq M$ where $|E| = k - 1$ and $E \cap C \neq \EmptySet$ is almost identical but with slightly different bookkeeping. We define the following:
\begin{enumerate}
    \item  Here $S_{\Xi}^{-}(BCE)$ is the collection of formulas (with parameters from $BCE$) such that for every $r \in S_{\Xi}^{-}(BCE)$, $r \vdash \Xi$ and both $r \wedge {R(x,E)}$ and $r \wedge \neg R(x,E)$ complete (over $BCE$). In other words, these formulas make decision about all possible tuples of elements from $BCD$, they agree with $\Xi$, but remain agnostic on $R(x,E)$. 
    \item We let $K\colon S_{x}^{\flat} \times \mathbb{S}_{C} \to [0,1]$ via
    \begin{equation*}
        K(p,\bar{q}) := \mathbf{1}_{\Xi_{B}}(p) \left( \prod_{t=1}^{k-2} \prod_{C_0 \in {C\choose t}} \prod_{B_0 \in {B\choose k - 1 -t}} \mathbf{1}_{R^{\epsilon^{\Xi}(B_0,C_0)}(x,B_0,\bar{z}_t)}(q_{C_0}) \right) \left( \prod_{C_0 \in {C\choose k - 1}} W^{\epsilon^{\Xi}(\EmptySet,C_0)}(p,\mathring{\mathbf{q}}_{[C_0]}) \right). 
    \end{equation*} 
    \item For each $r \in S_{\Xi}^{-}(BCE)$, we let $H_{r}: S_{x}^{\flat} \times \mathbb{S}_{CE} \to [0,1]$ via 
    \begin{equation*}
            H_r(p,\bar{q}) := \left( \prod_{t=1}^{k-2} \prod_{\substack{P_0 \in { CE\choose t } \\ P_0 \cap E \neq \EmptySet}} \prod_{B_0 \in {B \choose k - 1 -t }} \mathbf{1}_{R^{\epsilon^{r}(B_0,P_0)}(x,B_0,\bar{z}_t)}(q_{P_0}) \right). 
    \end{equation*}
    \item For each $r \in S_{\Xi}^{-}(BCE)$, we let $F_{r}: S_{x}^{\flat} \times \mathbb{S}_{CE} \to [0,1]$ via
    \begin{equation*}
            F_r(p,\bar{q}) := \left(\prod_{\substack{P_0 \in {CE\choose k-1} \\ P_0 \neq E \\ P_0 \cap E \neq \EmptySet}}W^{\epsilon^{r}(\EmptySet,P_0)} (p,\mathring{\mathbf{q}}_{[P_0]}) \right).
    \end{equation*}
    \item For each $r \in S_{\Xi}^{-}(BCE)$, we let $G_{r}(p,\bar{q}) := H_{r}(p, \bar{q}) \cdot F_{r}(p,\bar{q})$. 
\end{enumerate}
Now consider the following computation:
\begin{align*}
   &\mu(\Xi(x,B,C) \wedge \gamma(x)) + \mu(\Xi(x,B,C) \wedge \neg \gamma(x)) \\  
    &\hspace*{10pt}=
    \sum_{r \in S_{\Xi}^{-}(BCE)}  \mu(\Xi(x,B,C) \wedge \gamma(x) \wedge r(x)) + \mu(\Xi(x,B,C) \wedge \neg \gamma(x) \wedge r(x)) \\
    &\hspace*{10pt}=
    \sum_{r \in S_{\Xi}^{-}(BCE)} \int_{p \in S^{\flat}_{x}} \int_{\bar{q} \in \mathbb{S}_{CE}} \mathbf{1}_{\Xi_{B}}(p) \left( \prod_{t=1}^{k-2} \prod_{P_0 \in { CE\choose t }} \prod_{B_0 \in {B \choose k - 1 -t }} \mathbf{1}_{R^{\epsilon^{r}(B_0,P_0)}(x,B_0,\bar{z}_t)}(q_{P_0}) \right) \\ 
    &\hspace*{30pt}
     \left(\prod_{\substack{P_0 \in {CE\choose k-1} \\ P_0 \neq E}}W^{\epsilon^{r}(\EmptySet,P_0)} (p,\mathring{\mathbf{q}}_{[P_0]}) \right) \cdot (W(p,\mathring{\mathbf{q}}_{[E]}) - (1 - W(p,\mathring{\mathbf{q}}_{[E]}))) d\bar{\lambda} \\
    &\hspace*{10pt}=
    \int_{p \in S^{\flat}_{x}} \int_{\bar{q} \in \mathbb{S}_{CE}} \mathbf{1}_{\Xi_{B}}(p) \sum_{r \in S_{\Xi}^{-}(BCE)} \left( \prod_{t=1}^{k-2} \prod_{P_0 \in { CE\choose t }} \prod_{B_0 \in {B \choose k - 1 -t }} \mathbf{1}_{R^{\epsilon^{r}(B_0,P_0)}(x,B_0,\bar{z}_t)}(q_{P_0}) \right) \\ 
    &\hspace*{30pt}
     \left(\prod_{\substack{P_0 \in {CE\choose k-1} \\ P_0 \neq E}}W^{\epsilon^{r}(\EmptySet,P_0)} (p,\mathring{\mathbf{q}}_{[P_0]}) \right)d\bar{\lambda} \\
    &\hspace*{10pt}\overset{}{=}  
    \int_{p \in S^{\flat}_{x}} \int_{\bar{q} \in \mathbb{S}_{CE}} K(p,\pi_{\mathbb{S}_{C}}(\bar{q})) \left(\sum_{r \in S_{\Xi}^{-}(BCE)}G_{r}(p,\bar{q})\right) d\bar{\lambda} \\
    &\hspace*{10pt}\overset{(\dagger)}{=}  
    \int_{p \in S^{\flat}_{x}} \int_{\bar{q} \in \mathbb{S}_{CE}} K(p,\pi_{\mathbb{S}_{C}}(\bar{q})) d\bar{\lambda} \\
    &\hspace*{10pt}=
    \int_{p \in S^{\flat}_{x}} \int_{\bar{q} \in \mathbb{S}_{C}} \mathbf{1}_{\Xi_{B}}(p) \left( \prod_{t=1}^{k-2} \prod_{C_0 \in {C\choose t}} \prod_{B_0 \in {B\choose k - 1 -t}} \mathbf{1}_{R^{\epsilon(B_0,C_0)}(x,B_0,\bar{z}_t)}(q_{C_0}) \right) \left( \prod_{C_0 \in {C\choose k - 1}} W^{\epsilon(\EmptySet,C_0)}(p,\mathring{\mathbf{q}}_{[C_0]}) \right) d\bar{\lambda} \\ 
    &\hspace*{10pt}=
    \mu(\Xi(x,B,C)).
\end{align*}
It suffices to justify equation $(\dagger)$. Fix $(p,\bar{q})$. We will show that if 
\begin{equation*}
    K(p,\pi_{\mathbb{S}_{C}}(\bar{q})) \left(\sum_{r \in S_{\Xi}^{-}(BCE)}G_{r}(p,\bar{q})\right) > 0,
\end{equation*}
then
\begin{equation*}
    K(p,\pi_{\mathbb{S}_{C}}(\bar{q})) \left(\sum_{r \in S_{\Xi}^{-}(BCE)}G_{r}(p,\bar{q})\right) = K(p,\pi_{\mathbb{S}_{C}}(\bar{q})). 
\end{equation*}
Notice that every element $r \in S_{\Xi}^{-}(BCE)$ can be decomposed into $r_1 \wedge r_2$ where the following hold:
\begin{enumerate}
    \item $r_1 \in S_{\Xi}^{B}(BCE)$ where $S_{\Xi}^{B}(BCE)$ is the collection formulas $\alpha$ which are consistent with $\Xi$ and 
    \begin{equation*}
        \alpha = \bigwedge_{\substack{B_0 \subseteq B \\ 1 < |B_0| < k - 1}} \bigwedge_{\substack{C_0 \subseteq CE \\ |C_0| + |B_0| = k -1}} R^{\epsilon^{\alpha}(B_0,C_0)}(x,B_0,C_0). 
    \end{equation*}
    So each element of $S_{\Xi}^{B}(BCE)$ makes choices about every atomic formula with parameters among $BCE$ which has at least one parameter from $B$.
    \item $r_2 \in S_{\Xi}^{-}(CE)$, i.e., the collection of formulas (with parameters from $CE$) and which are consistent with $\Xi$, remain agnostic on $R(x,E)$, and whose intersection with $R(x,E)$ and $\neg R(x,E)$ are both consistent and complete over $CE$.
\end{enumerate}
Since all possible pairings are consistent, we may write $S_{\Xi}^{-}(BCE) = S_{\Xi}^{B}(BCE) \times S_{\Xi}^{-}(CE)$. Moreover, if $r = (r_1,r_2)$, notice that the map $H_{r}$ is dependent only on $r_1$ while the map $F_{r}$ is only dependent on $r_2$. If $K(p,\pi_{\mathbb{S}_{C}}(\bar{q}) ) > 0$, then we have
\begin{align*}
    \sum_{r \in S_{\Xi}^{-}(BCE)} G_{r}(p,\bar{q}) &= \sum_{r \in S_{\Xi}^{-}(BCE)} H_{r}(p,\bar{q}) \cdot F_{r}(p,\bar{q}) \\ 
    &= \sum_{(r_1,r_2) \in S_{\Xi}^{B}(BCE) \times S_{\Xi}^{-}(CE)} H_{(r_1,r_2)}(p,\bar{q}) \cdot F_{r}(p,\bar{q}) \\
    &= \sum_{r_1 \in S_{\Xi}^{B}(BCE)} H_{r_1}(p,\bar{q}) \sum_{r_2 \in S_{\Xi}^{-}(CE)}F_{r_2}(p,\bar{q})\\
    &\overset{(a)}{=} \sum_{r_1 \in S_{\Xi}^{B}(BCE)} H_{r_1}(p,\bar{q}) \\
    &\overset{(b)}{=} 1, 
\end{align*}
where ($a$) and ($b$) hold for the following reasons:
\begin{enumerate}[label=($\alph*$)]
    \item This follows directly from \cref{lemma:sumprod}. 
    \item Each $(p,\bar{q})$ corresponds to precisely one $r_1$ and thus $H_{r_1}(p,\bar{q}) = 1$ for any choice of $(p,\bar{q})$. 
\end{enumerate}
\noindent \textbf{Case 3:} Suppose that $D \subseteq M$, $E \subseteq \mathcal{U} \backslash M$ are non-empty sets and  $|D| + |E| = k - 1$, $D \cap B = \EmptySet$ and $E \cap C = \EmptySet$. Again, the proofs of the case where $D \cap B \neq \EmptySet$ and $E \cap C \neq \EmptySet$ are similar, but with slightly different bookkeeping. Consider the atomic formula $R(x,D,E)$ and set it equal to $\gamma$. The following definitions will be helpful in the computation.  
\begin{enumerate}
    \item Recall $\psi_{BD}^{r}$ from \emph{Case 1}.
    \item Recall $K\colon S_{x}^{\flat} \times \mathbb{S}_{C} \to [0,1]$ from \emph{Case 2}. 
    \item We let $S_{\Xi}^{-}(BCDE)$ be the collection of all formulas $r$ such that $r \vdash \Xi$ and both $r \wedge R(x,D,E)$ and $r \wedge \neg R(x,D,E)$ are complete (over $BCDE$). 
    In other words, these formulas make decisions about all possible $(k-1)$-tuples of elements from $BCDE$, yet remain agnostic about $R(x,D,E)$. 
    \item For each $r \in S_{\Xi}^{-}(BCDE)$, we let $\hat{H}\colon S_{x}^{\flat} \times \mathbb{S}_{CE} \to [0,1]$ via 
    \begin{equation*}
        H_r(p,\bar{q}) \defas \prod_{t=1}^{k-2} \prod_{\substack{P_0 \in {CE\choose t} \\ Q_0 \in {BD\choose k - 1 -t} \\ (P_0,Q_0) \neq (E,D) }} 1_{R^{\epsilon^{r}(Q_0,P_0)}(x,Q_0,\bar{z}_{t})}(q_{P_0}).
    \end{equation*}
    \item For each $r \in S_{\Xi}^{-}(BCDE)$, we let $\hat{F}_{r}\colon S_{x}^{\flat} \times S_{CE} \to [0,1]$ via 
    \begin{equation*}
        \hat{F}_{r}(p,\bar{q}) \defas \prod_{\substack{P_0 \in {CE\choose k -1} \\ P_0 \cap E \neq \EmptySet}} W^{\epsilon^{r}(\EmptySet,P_0)}(p,\mathring{\mathbf{q}}_{[P_0]}).
    \end{equation*}
    \item For each $r \in S_{\Xi}^{-}(BCDE)$, we let $\hat{G}_{r}(p,\bar{q}) \defas \hat{H}_{r} (p,\bar{q}) \cdot \hat{F}_{r}(p,\bar{q})$. 
\end{enumerate}

Notice that 
\begin{align*}
&\mu(\Xi(x,B,C) \wedge \gamma(x)) + \mu(\Xi(x,B,C) \wedge \neg \gamma(x)) \\  
&\hspace*{20pt}=
    \sum_{r \in S_{\Xi}^{-}(BCDE)}  \mu(\Xi(x,B,C) \wedge \gamma(x) \wedge r(x)) + \mu(\Xi(x,B,C) \wedge \neg \gamma(x) \wedge r(x)) \\
&\hspace*{20pt}=
    \sum_{r \in S_{\Xi}^{-}(BCDE)} \int_{p \in S^{\flat}_{x}} \int_{\bar{q} \in \mathbb{S}_{CE}} [\mathbf{1}_{\Xi_{B} \wedge \psi^{r}_{BD}}](p) \left( \prod_{t=1}^{k-2} \prod_{\substack{P_0 \in {CE \choose t} \\ Q_0 \in {BD \choose k-1 -t } \\ (P_0,Q_0) \neq (E,D)}} \mathbf{1}_{R^{\epsilon^{r}(Q_0,P_0)}(x,Q_0,\bar{z}_{t})}(q_{P_0}) \right) \\ 
&\hspace*{35pt}
    \left( \prod_{P_0 \in {CE \choose k-1}} W^{\epsilon^{r}(\EmptySet,P_0)}(p,\mathring{\mathbf{q}}_{[P_0]}) \right) \left( 1_{R(x,D,\bar{z}_{|E|})}(q_{E}) + 1_{\neg R(x,D,\bar{z}_{|E|})}(q_{E})  \right) d\bar{\lambda} \\
&\hspace*{20pt}=
    \int_{p \in S^{\flat}_{x}} \int_{\bar{q} \in \mathbb{S}_{CE}}  K(p,\pi_{\mathbb{S}_{C}}(\bar{q})) \left(\sum_{r \in S_{\Xi}^{-}(BCDE)} \mathbf{1}_{\psi_{BD}^{r}} \cdot \widehat{G}_{r}(p,\bar{q})\right) d\bar{\lambda} \\
&\hspace*{20pt}\overset{(\ddagger)}{=} 
\int_{p \in S^{\flat}_{x}} \int_{\bar{q} \in \mathbb{S}_{CE}}  K(p,\pi_{\mathbb{S}_{C}}(\bar{q})) d\bar{\lambda} \\
&\hspace*{20pt}=
    \mu(\Xi(x,B,C)). 
\end{align*}
We justify equation $(\ddagger)$ in a similar way to our justification of equation $(\dagger)$ from \emph{Case 2}. Indeed, we argue that if $K(p,\pi_{\mathbb{S}_{C}}(\bar{q})) > 0$, then 
\begin{equation*}
    K(p,\pi_{\mathbb{S}_{C}}(\bar{q})) = K(p,\pi_{\mathbb{S}_{C}}(\bar{q})) \cdot \left( \sum_{r \in S_{\Xi}^{-}(BCDE)} 1_{\psi_{BD}^{r}} \cdot \hat{G}_{r}(p,\bar{q})\right). 
\end{equation*}

Every element $r \in S_{\Xi}^{-}(BCDE)$ can be uniquely decomposed into a triple of elements $r_1 \wedge r_2 \wedge r_3$ where $r_1$ corresponds to literals from $r$ with parameters from inside the model, $r_2$ corresponds to literals from $r$ with parameters from both inside and outside the model, and $r_3$ corresponds to literals with parameters from outside the model. More explicitly, the following hold.
\begin{enumerate}
    \item $r_1 \in S_{\Xi}(BD)$ where $S_{\Xi}(BD)$ is the collection of formulas $\alpha$ which are consistent with $\Xi$ and 
    \begin{equation*}
        \alpha = \bigwedge_{Q_0 \in {BD\choose k-1}} R^{\epsilon^{\alpha}}(x,Q_0). 
    \end{equation*}
    Given $r$, notice that $r_1 = \psi_{BD}^{r}$. 
    \item $r_2 \in S_{\Xi}^{BD}(BCDE)$ where $S_{\Xi}^{BD}(BCDE)$ is the collection of formulas $\beta$ such are consistent with $\Xi$ and
    \begin{equation*}
    \beta = \bigwedge_{t = 1}^{k- 2}\bigwedge_{\substack{P_0 \in {CE\choose t} \\ Q_0 \in {BD \choose k - 1 -t} \\ (P_0,Q_0) \neq (E,D)}} R^{\epsilon^{\beta}(Q_0,P_0)}(x,Q_0,P_0). 
    \end{equation*}
    \item $r_3 \in S_{\Xi}(CE)$ where $S_{\Xi}(CE)$ is the collection of formula $\gamma$ which are consistent with $\Xi$ and 
    \begin{equation*}
        \gamma = \bigwedge_{P_0 \in {CE \choose k -1 }} R^{\epsilon^{\beta}(\EmptySet,P_0)}(x,P_0). 
    \end{equation*}
\end{enumerate}
Since all possible triples of elements from these three sets are consistent, we may write $S^{-}_{\Xi}(BCDE) = S_{\Xi}(BD) \times S_{\Xi}^{BD}(BCDE) \times S_{\Xi}(CE)$. If $r = (r_1,r_2,r_3)$, notice that $\mathbf{1}_{\psi_{BD}^{r}}$ only depends on $r_1$, $\hat{H}_{r}(p,\bar{q})$ only depends on $r_2$, and $\hat{F}_{r}(p,\bar{q})$ only depends on $r_3$. Therefore if $K(p,\pi_{\mathbb{S}_{C}}(\bar{q})) > 0$, then a similar argument to \emph{Case 2} demonstrates the following: 
\begin{align*}
    \sum_{r \in S_{\Xi}^{-}(BCE)} \mathbf{1}_{\psi_{BD}^{r}}(p) \cdot \hat{G}_{r}(p,\bar{q}) &= \sum_{r_1 \in S_{\Xi}(BD)} \mathbf{1}_{\psi^{r_1}_{BD}}(p) \sum_{r_2 \in S_{\Xi}^{BD}(BCDE)} \hat{H}_{r_2}(p,\bar{q}) \sum_{r_3 \in S_{\Xi}(CE)} \hat{F}_{r_3}(p,\bar{q}) \\
    &= 1. \qedhere
\end{align*}
\end{proof}

\bibliographystyle{plain}
\bibliography{refs}

\end{document}